\documentclass[10pt,letterpaper]{amsart}
\usepackage{babel}
\usepackage{amsmath,amsthm,bm}
\usepackage{amsbsy}
\usepackage{amstext}
\usepackage{amsfonts}
\usepackage{floatflt}
\usepackage[pdftex]{graphicx}
\usepackage{mathcomp}
\usepackage{mathtools}
\usepackage[dvipsnames]{xcolor}
\usepackage{tikz,pgfplots}
\usepackage{dsfont}
\usepackage{latexsym}
\usepackage{amssymb}
\usepackage{stmaryrd}
\usepackage{bm}
\usepackage{enumerate}
\usepackage{esint}	
\usepackage[colorlinks,pdfpagelabels,pdfstartview = FitH,bookmarksopen = true,bookmarksnumbered = true,linkcolor = blue,plainpages = false,hypertexnames = false,citecolor = red]{hyperref}
\usepackage{cite}

\allowdisplaybreaks
\makeatletter
\newsavebox{\@brx}
\newcommand{\llangle}[1][]{\savebox{\@brx}{\(\m@th{#1\langle}\)}%
  \mathopen{\copy\@brx\kern-0.5\wd\@brx\usebox{\@brx}}}
\newcommand{\rrangle}[1][]{\savebox{\@brx}{\(\m@th{#1\rangle}\)}%
  \mathclose{\copy\@brx\kern-0.5\wd\@brx\usebox{\@brx}}}
\makeatother

\newtheorem{theorem}{Theorem}
\newtheorem{lemma}[theorem]{Lemma}
\newtheorem{proposition}[theorem]{Proposition}

\theoremstyle{definition}
\newtheorem{definition}[theorem]{Definition}

\newtheorem{remark}[theorem]{Remark}

\numberwithin{theorem}{section}
\numberwithin{equation}{section}

\renewcommand{\div}{\operatorname{div}}

\newcommand{\N}{\ensuremath{\mathbb{N}}}
\newcommand{\R}{\ensuremath{\mathbb{R}}}
\newcommand{\mint}{- \mskip-19,5mu \int}
\newcommand{\tmint}{- \mskip-16,5mu \int}

\newcommand{\spt}{\operatorname{spt}}

\newcommand{\dx}{\mathrm{d}x}
\newcommand{\dy}{\mathrm{d}y}
\newcommand{\dt}{\mathrm{d}t}

\newcommand\sfE{{\boldsymbol{\mathsf E}}}
\newcommand\sfK{{\boldsymbol{\mathsf K}}}

\newcommand\sfI{{\boldsymbol{\mathsf I}}}
\newcommand\sfV{{\boldsymbol{\mathsf V}}}

\newcommand\sfT{{\boldsymbol{\mathsf T}}}
\newcommand\sfD{{\boldsymbol{\mathsf D}}}

\makeatletter
\@namedef{subjclassname@2020}{%
  \textup{2020} Mathematics Subject Classification}
\makeatother

\def\Xint#1{\mathchoice
    {\XXint\displaystyle\textstyle{#1}}%
    {\XXint\textstyle\scriptstyle{#1}}%
    {\XXint\scriptstyle\scriptscriptstyle{#1}}%
    {\XXint\scriptscriptstyle\scriptscriptstyle{#1}}%
    \!\int}
\def\XXint#1#2#3{\setbox0=\hbox{$#1{#2#3}{\int}$}
    \vcenter{\hbox{$#2#3$}}\kern-0.5\wd0}
\def\bint{\Xint-}
\def\dashint{\Xint{\raise4pt\hbox to7pt{\hrulefill}}}

\def\XXiint#1#2#3{\setbox0=\hbox{$#1{#2#3}{\iint}$}
    \vcenter{\hbox{$#2#3$}}\kern-0.5\wd0}

\renewcommand{\epsilon}{\varepsilon}
\newcommand{\eps}{\varepsilon}
\renewcommand{\rho}{\varrho}

\renewcommand{\epsilon}{\varepsilon}
\renewcommand{\rho}{\varrho}
\renewcommand{\d}{\:\! \mathrm{d}}

\DeclareMathOperator{\loc}{loc}
\DeclareMathOperator{\Tail}{Tail}

\numberwithin{equation}{section}
\allowdisplaybreaks

\subjclass[2020]{}

\begin{document}
\renewcommand{\refname}{References} 
\renewcommand{\abstractname}{Abstract} 
\title[Sharp gradient integrability for $(s,p)$-Poisson type equations]{Sharp gradient integrability for\\ $(s,p)$-Poisson type equations}

\author[V. B\"ogelein]{Verena B\"{o}gelein}
\address{Verena B\"ogelein\\
Fachbereich Mathematik, Universit\"at Salzburg\\
Hellbrunner Str. 34, 5020 Salzburg, Austria}
\email{verena.boegelein@plus.ac.at}

\author[F. Duzaar]{Frank Duzaar}
\address{Frank Duzaar\\
Fachbereich Mathematik, Universit\"at Salzburg\\
Hellbrunner Str. 34, 5020 Salzburg, Austria}
\email{frankjohannes.duzaar@plus.ac.at}

\author[N. Liao]{Naian Liao}
\address{Naian Liao\\
Fachbereich Mathematik, Universit\"at Salzburg\\
Hellbrunner Str. 34, 5020 Salzburg, Austria}
\email{naian.liao@plus.ac.at}

\author[K.~Moring]{Kristian Moring}
\address{Kristian Moring\\
Fachbereich Mathematik, Universität Salzburg\\
Hellbrunner Str.~34, 5020 Salzburg, Austria}
\email{kristian.moring@plus.ac.at}

\date{\today}

\begin{abstract}
We prove local $W^{1,q}$-regularity for weak solutions to fractional $p$-Laplacian type equations with right-hand side $f\in L^r_{\loc}(\Omega)$.
Assuming $p>1$, $s\in(0,1)$, and $sp'>1$, solutions belong to $W^{1,q}_{\loc}(\Omega)$ for the optimal exponent $q=q(n,p,s,r)$.
We obtain quantitative local gradient estimates involving nonlocal tail terms.
The optimality of $q$ is confirmed by a counterexample.
\end{abstract}

\makeatother

\maketitle

\tableofcontents

\section{Introduction}
The celebrated works \cite{CZ-1,CZ-2} of Calder\'on and Zygmund on singular integrals in the 1950s lie at the heart of modern harmonic analysis and pave the way for the analysis of the optimal integrability and differentiability properties of solutions to equations with linear differential operators. To make things concrete, suppose $u$ solves the Poisson equation $-\Delta u=f$ in a bounded open set $\Omega \subset \R^n$, then the theory of Calder\'on and Zygmund yields that
\[
f\in L^r_{\loc}(\Omega)\quad\overset{\text{Calder\'on-Zygmund}}{\implies}\quad |D^2u| \in L^r_{\loc}(\Omega) \quad\overset{\text{Sobolev}}{\implies}\quad\nabla u\in L^{\frac{rn}{n-r}}_{\loc}(\Omega,\R^n)
\]
for any $r\in(1,n)$.
In other words, the integrability of $\nabla u$ can be calibrated by that of $f$ in a sharp way. Such results motivated the study of analogous properties for nonlinear
problems, such as the \(p\)-Poisson equation
\[
-\Delta_p u=-\div\big(|\nabla u|^{p-2}\nabla u\big)=f.
\]
In particular, one has
\[
f\in L^r_{\loc}(\Omega)
\quad\Longrightarrow\quad
\nabla u\in L^{\frac{rn(p-1)}{n-r}}_{\loc}(\Omega)
\]
for every
\[
r\in\bigg(\max\Big\{1,\frac{np}{n(p-1)+p}\Big\},\,n\bigg);
\]
see~\cite{AcMi, CP-98, DiBe-Man, Iwaniecz, Ki-Zh}. In these works, the right-hand side is given in divergence form. 
To obtain the above implication, one first needs to solve
$\operatorname{div} (|F|^{p-2}F )= f$. The absence of representation formulae poses tangible challenges, entailing new perspectives as well as the introduction of novel tools capable of emulating representation formulae. This field, known as {\it nonlinear Calder\'on--Zygmund theory}, has continued to expand its scope over the past four decades.  For more details, see the survey of Mingione \cite{Min-survey}. 

\subsection{Fractional $p$-Poisson type equations}In recent years research into the regularity theory for fractional equations of the $p$-Laplacian type, written as $\left( -\Delta_p\right)^s u=0$, has gained considerable momentum. Such an equation is motivated by considerations from the calculus of variations, similar to the $p$-Laplace operator. The focus of current interest is on gradient regularity; several groups of authors have made significant contributions in this direction. In particular, it has been proven in \cite{Brasco-Lindgren, BDLMS-1,  BDLMS-2, DKLN-higherdiff} under various conditions that weak solutions, which are merely fractionally differentiable, possess gradient, which actually lies in proper Lebesgue spaces. The method of proof is based on finite differences, originally devised in \cite{Brasco-Lindgren, Brasco-Lindgren-Schikorra} for the fractional $p$-Laplacian equation. Subsequently, and more recently, boundedness and continuity of the gradient have been obtained in \cite{BS-25, BT-25, GJS-25} using completely different approaches. Notwithstanding important progress, a theory in the vein of Calder\'on--Zygmund eludes these works.

Here in this work, we aim to pinpoint how the integrability of $f$ determines that of $\nabla u$, thereby establishing a complete analogue of the Calder\'on-Zygmund type estimates for the fractional $p$-Poisson equation.
To be specific, consider the $(s,p)$-Poisson equation with a coefficient, that is,
\begin{equation} \label{eq:frac-p-lap}
\left( - a \Delta_p\right)^s u = f \quad \text{ in } \Omega,
\end{equation}
 where $s\in(0,1)$, $p>1$ and $f$ is given in a proper Lebesgue space whereas 
$$
(- a\Delta_p )^s u(x) := \mathrm{p.v.} \int_{\R^n} a(x,y)\frac{ |u(x) - u(y)|^{p-2} (u(x) - u(y) )}{|x-y|^{n+sp}}  \, \d y.
$$
The coefficient $a(\cdot,\cdot)\colon\R^n\times\R^n\to \R_{\ge0}$ is measurable, symmetric and satisfies
\begin{equation}\label{eq:a-condition}
\left\{
    \begin{array}{cc}
       C_o\le a(x,y)\le C_1 \quad  \mbox{a.e. $(x,y)\in\R^n\times\R^n$} \\[5pt]
       \displaystyle\sup_{x,y\in B_R(x_o)}|a(x,y)-a(x_o,x_o)|\le C_1 R^{\chi}\quad \mbox{for any $R\in(0,R_o]$}
    \end{array}
    \right.
\end{equation}
for any $x_o\in \Omega$ and for some $C_1\ge C_o>0$, $R_o>0$,  and $\chi\in(0,1)$. 

As we explained earlier, the issue is to quantify how the integrability of $f$ determines that of $\nabla u$. Previously in \cite{BDLM} we established, under the restrictions that $a\equiv 1$ and $s\in(\frac{p-1}{p},1)$, the following gradient regularity
\[
f\in L^{r}_{\loc}(\Omega)\quad\implies\quad \nabla u\in L^{r(p-1)}_{\loc}(\Omega,\R^n)
\]
where $r$ is any number greater than $\frac{p}{p-1}$.

The main objective here is to upgrade such gradient regularity and establish the sharp version.
Moreover, non-differentiable coefficients as in \eqref{eq:a-condition} are considered. The precise statement is as follows.

\begin{theorem}\label{Thm:1}
Let $p>1$ and $s\in(0,1)$ with $sp'>1$. Assume that $f\in L^{r}_{\loc}(\Omega)$ for some
\begin{equation}\label{Eq:range-r}
r\in\bigg(\max\bigg\{1,\frac{np}{(p-1)\big[n+p(sp^{\prime}-1)\big]}\bigg\},
\frac{n}{(p-1)(sp^{\prime}-1)}\bigg).
\end{equation}
Then any local weak solution $u$ to \eqref{eq:frac-p-lap} satisfying \eqref{eq:a-condition} in the sense of Definition~\ref{def:weak-sol} belongs to
\[
u\in W^{1,q}_{\loc}(\Omega),
\qquad
q=\frac{rn(p-1)}{\,n-r(p-1)(sp^{\prime}-1)}.
\]
Moreover, there exist a constant $c=c(n,p,s,C_o,C_1,\chi,r)>0$ and a radius $\widetilde R_o=\widetilde R_o(n,p,s,C_o,C_1,\chi, R_o,r)\in (0,R_o]$ such that
\begin{align*}
\bigg[ \bint_{B_{\frac12 R}}& |\nabla u|^{q} \, \d x \bigg]^{\!\frac{1}{q}}\\
&\leq 
c \bigg[\bint_{B_{R}}  |\nabla u|^p  \, \d x\bigg]^{\!\frac{1}{p}}
+ \frac{c}{R}\, \bigg[ \bint_{B_R} \big|R^{sp}f\big|^{r} \, \d x \bigg]^{\!\frac{1}{r(p-1)}}  + \frac{c}{R} \, \mathrm{Tail}\big(u-(u)_R;B_R\big),
\end{align*}
for every ball $B_{4R}\subset\Omega$ with $R\in(0,\widetilde R_o]$.
\end{theorem}

One quickly realizes that formally letting $s=1$ in this theorem yields the integrability of $\nabla u$ that has been known for the classical $p$-Laplace equation; letting $p=2$ we recover the integrability illustrated in \cite[Eqn.~(1.24)]{DKLN-pot}. Moreover, we have $q\to\infty$ as $r\uparrow \frac{n}{(p-1)(sp^{\prime}-1)}$. As a matter of fact, the integrability of $\nabla u$ in Theorem~\ref{Thm:1} is optimal as shown by the explicit example in \S \ref{sec:optimal}. In other words, Theorem~\ref{Thm:1} gives an optimal quantification of the integrability of $\nabla u$ above the $L^p$-level.

\begin{remark}
As a byproduct of Chapter~4, in particular Theorem~\ref{thm:higher-diff-2}, the \(L^p\)-norm of the gradient
appearing in the theorem can be controlled by the
\(L^p\)-norm of \(u\) together with the relevant tail term. The explicit quantitative estimate is then given by
\begin{align*}
    \bigg[ \bint_{B_{\frac12 R}} & |\nabla u|^{q} \, \d x \bigg]^\frac{1}{q} \\
    &\leq 
    \frac{c}{R} \,\bigg[\bint_{B_{R}}  |u|^p  \, \d x\bigg]^\frac{1}{p} + \frac{c}{R}\, \bigg[ \bint_{B_R} \big|R^{sp}f\big|^{r} \, \d x \bigg]^{\!\frac{1}{r(p-1)}} 
    + \frac{c}{R} \, \mathrm{Tail} \big(u;B_R\big),
\end{align*}
for every ball $B_{4R}\subset\Omega$ with $R\in(0,\widetilde R_o]$.
\end{remark}

An important difference from the classical $p$-Poisson equation is that the concept of ``gradient" is not inherent in the notion of weak solutions for the fractional $p$-Poisson equation. Prior to our work, it has been known from \cite{DKLN-higherdiff} that $f\in L^{\frac{p}{p-1}}_{\loc}(\Omega)$ guarantees $\nabla u\in L^{p}_{\loc}(\Omega,\R^n)$. An essential contribution of our work is that the integrability requirement of $f$ has been significantly reduced. In fact, the required amount of integrability corresponds to the lower bound of $r$ in \eqref{Eq:range-r}. More precisely, we have the following result which is an abridged statement of Theorem~\ref{thm:higher-diff-2}; see Theorem~\ref{thm:higher-diff} for the $(s,p)$-Poisson equation without coefficients.

\begin{theorem}\label{Thm:2}
Let $p>1$ and $s\in(0,1)$ with $sp'>1$. Assume that 
\[
f\in L^{r}_{\loc}(\Omega)
\quad \text{for some} \quad
r> r_*:=\max\bigg\{1,\frac{np}{(p-1)\,[\,n+p(sp'-1)\,]}\bigg\}.
\]
Then any local weak solution to \eqref{eq:frac-p-lap} satisfying \eqref{eq:a-condition} 
in the sense of Definition~\ref{def:weak-sol} belongs to $W^{1+\gamma,p}_{\loc}(\Omega)$ 
for some $\gamma=\gamma(n,p,s,\chi,r)\in(0,1)$.
\end{theorem}

\begin{remark}
For the parameter $r_\ast$, we note that 
\begin{align*}
    r_\ast 
    =\left\{
    \begin{array}{cl}
    \displaystyle
        \frac{np}{(p-1)[n+p(sp'-1)]}, & \mbox{if $\displaystyle s\in(\tfrac{p-1}{p},s_o]$,}\\[12pt]
        1, &  \mbox{if $\displaystyle s\in(s_o,1)$.}
    \end{array}
    \right.
\end{align*}
where 
\begin{equation*}
    s_o:=\frac{1}{p'}\bigg[1+\frac{n}{p(p-1)}\bigg].
\end{equation*}
In the case $s\in(s_o,1)$, the statement of Theorem~\ref{Thm:2} continues to hold for $f\in L^{1}_{\loc}(\Omega)$. 
\end{remark}

As the reader may have noted, the target Lebesgue space of $\nabla u$ in our theorems aims at the $L^p$-scale or higher, and the threshold integrability of $f$ is provided by Theorem~\ref{Thm:2} for this purpose. 
This is precisely in line with the philosophy of nonlinear Calder\'on–Zygmund theory.

\subsection{Optimality}\label{sec:optimal}
Let $sp<n$ and 
\[
u(x)=\frac{1}{|x|^{\gamma}},
\qquad\gamma=\frac{n}{r(p-1)}-sp',
\]
where
\begin{equation*}
r\in\bigg(\max\bigg\{1,\frac{np}{(p-1)\big[n+p(sp^{\prime}-1)\big]}\bigg\},
\frac{n}{(p-1)(sp^{\prime}-1)}\bigg).
\end{equation*}
A direct computation shows that
\[
(-\Delta_p)^su=f
\quad \text{pointwise in } \mathbb{R}^n\setminus\{0\},
\]
where
\[
f(x)=\frac{C}{|x|^{\frac{n}{r}}},
\qquad
C={\rm p.v.}\int_{\mathbb{R}^n}
\frac{|1-|y|^{-\gamma}|^{p-2}(1-|y|^{-\gamma})}{|e-y|^{n+sp}}
\, \d y
\]
and $e$ denotes a unit vector in $\R^n$. 
Moreover,
\[
u\in W^{s,p}_{\loc}(\mathbb{R}^n)\cap L^{p-1}_{sp}(\mathbb{R}^n),
\]
Combining these facts with the pointwise identity in 
\(\mathbb{R}^n\setminus\{0\}\), it follows that \(u\) is a local weak solution of
\[
(-\Delta_p)^s u=f
\qquad \text{in } B_1 .
\]
Furthermore, \(f\in L^{\tilde r}(B_1)\) for every \(\tilde r<r\), and
\[
|\nabla u|\in L^{\tilde q}(B_1)
\quad \text{for every } \tilde q<q=\frac{rn(p-1)}{\,n-r(p-1)(sp^{\prime}-1)},
\]
while
\[
|\nabla u|\notin L^{q}(B_1).
\]
This shows that the integrability assumption \(f\in L^{r}_{\rm loc}\) in the
Calderón--Zygmund estimate is optimal to conclude that
\(|\nabla u|\in L^{q}_{\rm loc}\). In particular, it cannot be relaxed to
\(f\in L^{\tilde r}_{\rm loc}\) for some \(\tilde r<r\) while still guaranteeing
\(|\nabla u|\in L^{q}_{\rm loc}\).

\subsection{Comments on techniques}
Theorem~\ref{Thm:2} rests upon our previous works \cite{BDLMS-1, BDLMS-2} on gradient estimates of $(s,p)$-harmonic functions, that is, local solutions to \eqref{eq:frac-p-lap} with $a\equiv1$ and $f\equiv0$. 
The improvement is achieved by a rather involved comparison technique, coupled with delicate finite iteration arguments as developed in \cite{BDLMS-1, BDLMS-2}. In essence, the comparison scheme unfolds in two stages: First, we consider $(-\Delta_p)^s u=f$ in \S~\ref{S:3}; second, we deal with coefficients in \S~\ref{S:4}. Unlike local equations, a fractional equation cannot be differentiated directly to obtain the equation satisfied by the gradient of solutions. As a result, the strategy consists in implementing the comparison argument at a discrete level by considering second order finite differences. Moreover, the low integrability of $f$ leads to a loss of increment in the standard finite difference scheme. To remedy this issue, a finite difference scheme must be implemented at a finer scale, sometimes called atomic decomposition. Such an issue is not unique to fractional equations. In fact, this idea appeared in \cite{DuMinKr, Kri-Min:05, Kri-Min:06, Min-07, Min-11, AKM-18, DM-23} in different contexts and has been recently adapted to certain fractional equations in \cite{DKLN-pot, KNS-22}. To the best of our knowledge, it is developed here for the first time in the setting of $(s,p)$-harmonic functions. Once a refined estimate is found for second order differences, we employ finite iterations to upgrade the preset $W^{s,p}$-regularity of solutions to the $W^{1+\gamma,p}$-regularity.

The proof of Theorem~\ref{Thm:1} is built on Theorem~\ref{Thm:2}, which  essentially serves to quantify the difference between solutions and their $(s,p)$-harmonic replacement at the gradient level. Together with the gradient estimates obtained in \cite{BDLMS-1, BDLMS-2} for $(s,p)$-harmonic functions in higher Lebesgue spaces, we will obtain gradient estimates on level sets, within a family of carefully selected balls. These balls are small enough to capture the essential largeness of the gradient above the designated level, yet large enough to generate a covering of the level set by appropriate enlargement.  This scheme was already  presented in our previous work \cite{BDLM}; it further develops the technique of \cite{AcMi, Acerbi-Min, CP-98} in a different context and features direct use of the equation under consideration, without resorting to standard tools of harmonic analysis. In contrast to \cite{AcMi, Acerbi-Min} our argument is based on the weakest possible a priori estimates and does not use boundedness of the gradient of $(s,p)$-harmonic functions or $C^{1,\alpha}$-estimates; this point stems from \cite{DuMinSt, Kri-Min:06}. 

Last but not least, most of the literature on nonlinear Calderón–Zygmund theory focuses on inhomogeneities in {\it divergence-form}, cf.~\cite{AcMi, CP-98, DiBe-Man, Iwaniecz, Ki-Zh}, while the corresponding results in the {\it non-divergence} case can be obtained as a corollary via a so-called Bogovski\i\ argument.
Here, our argument provides a new approach that directly handles
inhomogeneities in non-divergence form, and may be applicable in
settings where the Bogovski\i\ argument fails.

\section{Notation and preliminaries}\label{S:notation}

\noindent
Throughout the paper, we write \(p' := \tfrac{p}{p-1}\) for the H\"older conjugate of \(p\), and \(\mathbb{N}_0 := \mathbb{N}\cup\{0\}\) for the set of nonnegative integers. 
The symbol \(c\) denotes a positive constant whose value may vary from line to line. 
Whenever it is relevant to record the dependence of such constants on specific quantities, this will be made explicit by writing, for instance, \(c=c(n,p,s)\) to indicate dependence on the parameters \(n,p,s\). 
We denote by \(B_R(x_o)\Subset\mathbb{R}^n\) the open ball of radius \(R>0\) centered at \(x_o\in\mathbb{R}^n\). Moreover, the average of a function $f\in L^1(B_R(x_o))$ is denoted by $(f)_{x_o,R}:=\tmint_{B_R(x_o)} f\,\dx$. 
If the center plays no essential role, or when it is clear from the context, we simply write \(B_R\) instead of \(B_R(x_o)\) and $(f)_{R}$ instead of $(f)_{x_o,R}$.

Let $\Omega\subset\mathbb{R}^n$ be a bounded open set. 
For $\gamma\in(0,1)$, $q\in[1,\infty)$, and $k\in\mathbb{N}$, we define the \emph{fractional Sobolev space} $W^{\gamma,q}(\Omega,\mathbb{R}^k)$ as the collection of all measurable maps $w\colon\Omega\to\mathbb{R}^k$ for which
\[
    \|w\|_{W^{\gamma,q}(\Omega,\mathbb{R}^k)}
    := 
    \|w\|_{L^q(\Omega,\mathbb{R}^k)} 
    + [w]_{W^{\gamma,q}(\Omega,\mathbb{R}^k)}
    <\infty.
\]
The Gagliardo seminorm $[w]_{W^{\gamma,q}(\Omega,\mathbb{R}^k)}$ is given by
\[
    [w]_{W^{\gamma,q}(\Omega,\mathbb{R}^k)}
    :=
    \biggl[
    \iint_{\Omega\times\Omega}
    \frac{|w(x)-w(y)|^q}{|x-y|^{n+\gamma q}}\,\mathrm{d}x\,\mathrm{d}y
    \biggr]^{1/q}.
\]
Whenever the target dimension $k$ is understood from the context (or when $k=1$), we abbreviate
$[w]_{W^{\gamma,q}(\Omega)} := [w]_{W^{\gamma,q}(\Omega,\mathbb{R}^k)}$.  
In particular, if $w=\nabla v$ for some scalar function $v\colon\Omega\to\mathbb{R}$, we shall write 
$[\nabla v]_{W^{\gamma,q}(\Omega)}$ in place of 
$[\nabla v]_{W^{\gamma,q}(\Omega,\mathbb{R}^n)}$.  
Further properties of fractional Sobolev spaces are collected in \S~\ref{sec:fractional}; for a comprehensive reference, see~\cite{Hitchhikers-guide}.

In what follows, we set 
\begin{align*}
    p_s^\ast 
    :=
    \left\{
    \begin{array}{cl}
        \frac{np}{n-sp}, & \mbox{if $sp<n$,}\\[5pt]
        \infty, & \mbox{if $sp>n$,}
    \end{array}
    \right.
    \quad\mbox{and}\quad
    (p_s^\ast)'
    :=
    \left\{
    \begin{array}{cl}
        \frac{np}{np-n+sp}, & \mbox{if $sp<n$,}\\[5pt]
        1, & \mbox{if $sp>n$.}
    \end{array}
    \right.
\end{align*}

\begin{definition}\label{def:weak-sol}
Let $\Omega\subset\mathbb R^n$ be a bounded open set, $p\in(1,\infty)$, $s\in(0,1)$, and $f\in L^{r}_{\rm loc}(\Omega)$, with
$$
    r\ge (p_s^\ast)'\quad\mbox{if $sp\not=n$,}
    \qquad\mbox{or}\qquad 
    r>1 \quad\mbox{if $sp=n$.}
$$
A function $u\in W^{s,p}_{\rm loc}(\Omega)$
is called a  {\bf local weak solution} to the fractional $p$-Poisson equation \eqref{eq:frac-p-lap} in $\Omega$ if and only if
\begin{equation}\label{Lebesgue-weight}
\int_{\mathbb R^n}\frac{|u(x)|^{p-1}}{(1+|x|)^{n+sp}}\, \mathrm{d}x <\infty,
\end{equation}
and
\begin{equation*}
        \iint_{\mathbb R^n\times\mathbb R^n}a(x,y)\frac{|u(x)-u(y)|^{p-2}(u(x)-u(y))(\varphi(x)-\varphi(y))}{|x-y|^{n+sp}}\,\mathrm{d}x \mathrm{d}y =2\int_{\Omega} f \varphi \, \d x
\end{equation*}
for every $\varphi\in W^{s,p}(\R^n)$ with $\spt(\varphi)\Subset\Omega$.
\end{definition}
If a function $u$ satisfies the integrability condition~\eqref{Lebesgue-weight}, we shall write 
$u\in L^{p-1}_{sp}(\mathbb{R}^n)$ and refer to this as the \emph{weighted Lebesgue space}.  
Given a point $x_o\in\Omega$ and a radius $R>0$ such that $B_R(x_o)\Subset\Omega$, we associate to~$u$ the nonlocal quantity
\begin{equation*}
    \operatorname{Tail}\big(u;B_R(x_o)\big)
    :=
    \biggl[
    R^{sp}\int_{\mathbb{R}^n\setminus B_R(x_o)}
    \frac{|u(x)|^{p-1}}{|x-x_o|^{n+sp}}\,\mathrm{d}x
    \biggr]^{\frac{1}{p-1}}.
\end{equation*}

Next, to construct comparison functions, we also need to consider Dirichlet problems
\begin{equation}\label{Dirichlet} 
\left\{
    \begin{array}{cl}
      (- a\Delta_p)^s u  = 0 & \mbox{in $\Omega$,} \\[6pt]
      u = g & \mbox{a.e.~in $\R^n \setminus \Omega$.}
    \end{array}
\right.
\end{equation}
Here, the boundary datum $g$ belongs to $L^{p-1}_{sp}(\R^n)$.
This problem is well-posed in the following sense. 
\begin{definition}\label{def:weak-sol-D}
Let $\Omega\subset\mathbb R^n$ be a bounded open set, $p\in(1,\infty)$ and $s\in(0,1)$. A function $u\in W^{s,p}(\Omega)\cap L^{p-1}_{sp}(\R^n)$ is called a {\bf weak solution} to the Dirichlet problem \eqref{Dirichlet} in $\Omega$ if and only if $u=g$ a.e. in $\R^n\setminus\Omega$
and
\begin{equation*}
        \iint_{\mathbb R^n\times\mathbb R^n}a(x,y)\frac{|u(x)-u(y)|^{p-2}(u(x)-u(y))(\varphi(x)-\varphi(y))}{|x-y|^{n+sp}}\,\mathrm{d}x \mathrm{d}y =0
\end{equation*}
for every $\varphi\in W^{s,p}(\R^n)$ satisfying $\varphi\equiv0$ a.e. in $\R^n\setminus\Omega$.
\end{definition}

\subsection{Algebraic inequalities}
We recall the following elementary inequality from \cite[Lemma~2.3]{BDLM}.

\begin{lemma} \label{lem:elementary-superlevel}
Let $\gamma \in [1,\infty)$, $a,b \in \R^n$ and $|a| \geq K$ for some $K > 0$. Then 
for every $\alpha \in [\gamma,\infty)$
we have
\begin{equation*}
    |a|^\gamma \leq 2^\gamma |a-b|^\gamma + 2^{\alpha -1} K^{\gamma -\alpha} |b|^\alpha.
\end{equation*}
\end{lemma}

We also recall Minkowski’s inequality for finite sums:
\begin{lemma}
Let $p\ge1$, $i\in\{1,\dots,\ell\}$ and $k\in\{1,\dots,m\}$. For the sequence of numbers $\{a_{k,i}\}$ we have
    \begin{equation}\label{eq:sum-minkowski}
\bigg( \sum_{i=1}^{n} \Big| \sum_{k=1}^{m} a_{k,i} \Big|^{p} \bigg)^{1/p}
\le
\sum_{k=1}^{m} 
\bigg( \sum_{i=1}^{n} |a_{k,i}|^{p} \bigg)^{1/p}.
\end{equation}

\end{lemma}

\subsection{A covering Lemma}
We record the following elementary geometric covering lemma. See \cite[Lemma~2.11]{DKLN-pot}.

\begin{lemma} [Covering lemma]\label{lem:covering}
Let $x_o \in \R^n$, $R> 0$ and $r \in (0, R/2)$. There exists a constant $c = c(n)$, a finite index set $I$ and a sequence $\{z_i\}_{i \in I} \subset B_R(x_o)$ such that
\begin{equation*}
B_R(x_o) \subset \bigcup_{i \in I} B_r(z_i),\qquad  |I| \leq c \left( \frac{R}{r} \right)^n,
\end{equation*}
and 
\begin{equation*}
    \sup_{x \in \R^n} \sum_{i \in I} \chi_{B_{2^k r}(z_i)} \leq c\, 2^{nk}
    \qquad \forall\, k\in\N.
\end{equation*}
\end{lemma}

\subsection{Fractional Sobolev spaces}\label{sec:fractional}
\noindent
We recall several auxiliary results concerning fractional Sobolev spaces. 
The fractional Sobolev embedding is well known; see \cite[Chapter~7]{Adams-75}. An elementary approach can be found in \cite{Hitchhikers-guide}.
\begin{lemma}\label{lem:Sobolev}
Let $u\in W^{\gamma,q}(B_1)$ for some $\gamma\in(0,1)$ and $q\ge1$.
Then, we have
\begin{equation*}
    \|u\|_{L^{\kappa}(B_1)}
    \le 
    c\, \|u\|_{W^{\gamma,q}(B_1)},
\end{equation*}
where $c=c(n,\gamma,q,\kappa)$ and 
\begin{equation}\label{def:kappa}
\kappa = \left\{
\begin{array}{cl}
\frac{nq}{n-\gamma q},\quad  &\gamma q<n,\\[5pt]
\text{any number in} \> [1,\infty),\quad &\gamma q= n,  \\[5pt]
\infty,\quad &\gamma q>n.
\end{array}
\right.
\end{equation}
The dependence of the constant $c$ on $\kappa$ disappears when $\gamma q\neq n$.
\end{lemma}

We recall two variants of the fractional Poincaré inequality. The first is standard; for the second, see~\cite[Lemma 4.7]{Coz17b}.


\begin{lemma}\label{lem:poin}
Let $q \geq 1$, $\gamma \in(0, 1)$ and $R > 0$. Then there is a constant $c=c(n,q,\gamma)\ge 1$ such that for any 
$u \in W^{\gamma,q}(B_R)$ we have
\begin{equation*}
    \int_{B_R} |u-(u)_{R}|^q \, \d x 
    \leq c\,  R^{\gamma q}  \iint_{B_R\times B_R} \frac{|u(x) - u(y)|^q}{|x-y|^{n+\gamma q}} \, \d x \d y.
\end{equation*}
\end{lemma}

\begin{lemma} \label{lem:coz17b-4.7}
Let $q \geq 1$, $\gamma \in(0, 1)$,
$R > 0$, and $\alpha \in (0,1]$. Then there is a constant $c = c(n,\gamma,q,\alpha) \geq 1$ such that
for any $u \in W^{\gamma,q}(B_R)$ that satisfies $u = 0$ a.e.~on a set $\Omega_o \subset B_R$ with $|\Omega_o| \geq \alpha |B_R|$, we have
\begin{equation*}
    \int_{B_R} |u(x)|^q \, \d x 
    \leq 
    c\,   R^{\gamma q} \iint_{B_R \times B_R} \frac{|u(x)-u(y)|^q}{|x-y|^{n+\gamma q}} \, \d x \d y.
\end{equation*}
\end{lemma}

\begin{remark}
It is straightforward to compute that there exists some $c=c(\gamma,q,n)$, such that
$$
[u]_{W^{\gamma,q}(B_1)}^q \leq c  \left( \|u\|_{L^q(B_1)}^q + [u]_{W^{\sigma,q}(B_1)}^q \right),
$$
for any $0<\gamma<\sigma<1$.
Since $[u]_{W^{\gamma,q}(B_1)} = [u - (u)_{1}]_{W^{\gamma,q}(B_1)}$, we can use the fractional Poincar\'e inequality (Lemma~\ref{lem:poin}) to obtain that 
$$
[u]_{W^{\gamma,q}(B_1)}^p \leq c\,   [u]_{W^{\sigma,q}(B_1)}^p,
$$ or its scaled version:
\begin{equation} \label{eq:fract-embedding-s-sigma}
R^{\gamma q} [u]_{W^{\gamma,q}(B_R)}^q \leq c\, R^{\sigma q}  [u]_{W^{\sigma,q}(B_R)}^q.
\end{equation}
\end{remark}


The following embedding is well-known.

\begin{lemma}[Embedding $W^{1,q}\hookrightarrow W^{\gamma,q}$]
\label{lem:FS-S}
Let $q\ge 1$ and $\gamma\in (0,1)$. Then  for any  $w\in W^{1,q}(B_R)$ there exists a constant $c = c(n)$ such that
\begin{align*}
    \iint_{B_R\times B_R} 
    \frac{|w(x)-w(y)|^q}{|x-y|^{n+\gamma q}} \,\dx\dy
    \le 
    c \frac{R^{(1-\gamma)q}}{(1-\gamma)q} 
    \int_{B_R} |\nabla w|^q \,\dx. 
\end{align*}
\end{lemma}

Fractional differentiability can be translated to the language of finite differences. Here, we record some of such results. The first one can be found in \cite[Proposition~2.6]{Brasco-Lindgren}.
\begin{lemma}\label{lem:N-FS}
Let $q\in(1,\infty)$ and $\gamma \in(0,1)$. Then, there exists a constant $c=c(n,q)$ such that for any $w\in W^{\gamma,q}(B_{R})$, we have
\begin{align*}
    \int_{B_{7R/8}} |\boldsymbol{\tau}_h w|^q \,\dx 
    &\le
    c\,|h|^{\gamma q} 
    \Big[[w]^q_{W^{\gamma,q}(B_{R})} +
    R^{-\gamma q}\|w\|^q_{L^q(B_{R})}
    \Big]
\end{align*}
for any  $h\in\R^n\setminus\{0\}$ that satisfies  $|h|\le \frac{1}{8}R$. In particular, if $w\in W^{\gamma,q}(\R^n)$ has a compact support, then
\begin{align*}
    \int_{\R^n} |\boldsymbol{\tau}_h w|^q \,\dx 
    &\le
    c\,|h|^{\gamma q} [w]^q_{W^{\gamma,q}(\R^n)}.
\end{align*}
\end{lemma}

\noindent The second result is a consequence of the previous one.
\begin{lemma}\label{lem:W-1+gamma}
Let $q\in(1,\infty)$ and $\gamma \in(0,1)$. Then, there exists a constant $c=c(n,q)$ such that for any $w\in W^{1+\gamma,q}(B_{2R})$, we have
\begin{align*}
    \int_{B_{7R/8}} |\boldsymbol{\tau}^2_h w|^q \,\dx 
    &\le
    c\,|h|^{(1+\gamma) q} 
    \Big[[\nabla w]^q_{W^{\gamma,q}(B_{2R})} +
    R^{-\gamma q}\|\nabla w\|^q_{L^q(B_{2R})}
    \Big]
\end{align*}
for any  $h\in\R^n\setminus\{0\}$ that satisfies  $|h|\le \frac18R$. 
\end{lemma}
\begin{proof}
    First of all, for any $v\in W^{1,q}(B_R)$ it is well known that
    \[
    \|\boldsymbol{\tau}_h v\|^q_{L^q(B_{7R/8})}\le c(n,q) |h|^q \|\nabla v\|^q_{L^q(B_R)}
    \]
    for all $h\in\R^n\setminus\{0\}$ that satisfy  $|h|\le \frac18R$. On the other hand, since $\nabla w\in W^{\gamma,q}(B_{R})$, we apply Lemma~\ref{lem:N-FS} to get
    \begin{align*}
    \int_{B_{R}} |\boldsymbol{\tau}_h \nabla w|^q \,\dx 
    &\le
    c\,|h|^{\gamma q} 
    \Big[[\nabla w]^q_{W^{\gamma,q}(B_{2R})} +
    R^{-\gamma q}\|\nabla w\|^q_{L^q(B_{2R})}
    \Big]
\end{align*}
To conclude, we use the first estimate with $v=\boldsymbol{\tau}_h w$ and combine it with the second estimate.
\end{proof}

Finite differences can also be used to identify functions in fractional Sobolev spaces. Such a result can be interpreted as the reverse of Lemma \ref{lem:N-FS}. 
We refer to \cite[7.73]{Adams-75}. The version given here is from \cite[Lemma 2.7]{Brasco-Lindgren}. 
\begin{lemma}\label{lem:FS-N}
Let $q\in [1,\infty)$, $\gamma \in(0,1]$, $M\ge 1$, and $0<d<R$. Then, there exists a constant $c=c(n,q)$ such that whenever $w\in L^{q}(B_{R+d})$ satisfies 
\begin{align*}
    \int_{B_R} |\boldsymbol{\tau}_h w|^q \,\dx 
    \le 
    M^q|h|^{\gamma q} 
    \qquad \mbox{for any $h\in\R^N\setminus\{0\}$ with $|h|\le d$,}
\end{align*}
then $w\in W^{\beta ,q}(B_R)$ whenever $0<\beta <\gamma$. Moreover, we have
\begin{align*}
      \iint_{B_R\times B_R}  
    \frac{|w(x)-w(y)|^q}{|x-y|^{n+\beta q}} \,\dx\dy
    \le 
    c\bigg[\frac{d^{(\gamma-\beta)q}}{\gamma-\beta} M^q +
    \frac{1}{\beta d^{\beta q}} \|w\|_{L^q(B_R)}^q\bigg].
\end{align*}
\end{lemma}


We now state two results concerning second-order differences in the fractional setting. The first can be found in \cite[Chapter 5]{Stein}; see also~\cite[Lemma 2.2.1]{Domokos-1},~\cite[Proposition 2.4]{Brasco-Lindgren}, and~\cite[Lemma 2.17]{BDLMS-1}.

\begin{lemma} \label{lem:Domokos}
Let $q\in [1,\infty)$, $\gamma>0$, $M\ge 0$, $0<r<R$, and $0<d\le \tfrac1{2}(R-r)$. Then, there exists a constant $c=c(q)$ such that whenever  $w\in L^q(B_R)$ satisfies
\begin{align*}
    \int_{B_r} |\boldsymbol{\tau}^2_h w|^q \,\dx 
    \le 
    M^q|h|^{\gamma q},
    \qquad \mbox{for any $h\in\R^N\setminus\{0\}$ with $|h|\le d$,}
\end{align*}
then in the  case $\gamma\in (0,1)$ we have for any $ 0<|h|\le \tfrac12 d$ that
\begin{align*}
     \int_{B_r} |\boldsymbol{\tau}_hw|^q  \,\dx
     &\le
     c\, |h|^{\gamma q}
     \Bigg[\Big( \frac{M}{1-\gamma}\Big)^q +\frac{1}{d^{q\gamma }}
     \int_{B_R}|w|^q\,\dx
     \Bigg]
     ,
\end{align*}
while in the case $\gamma >1$ there holds
\begin{align}\label{est-1st-diffquot>1}  
     \int_{B_r} |\boldsymbol{\tau}_hw|^q  \,\dx
     &\le
     c\,|h|^q\bigg[
      \Big(\frac{M }{\gamma -1}\Big)^qd^{(\gamma-1)q}
      +
        \frac1{d^q}
        \int_{B_R}|w|^q\,\dx
      \bigg].
\end{align}
In the limiting case $\gamma =1$  we have for any $0<\beta<1$ that
\begin{align*}
     \int_{B_r} |\boldsymbol{\tau}_hw|^q  \,\dx
     &\le
      c\, |h|^{\beta q}
     \bigg[
     \Big(\frac{ M}{1-\beta}\Big)^q d^{(1-\beta)q}+\frac{1}{d^{\beta q}} \int_{B_R}|w|^q\,\dx\bigg].
\end{align*}
\end{lemma}


The next result is a sort of reverse of Lemma~\ref{lem:W-1+gamma}; it is taken from 
\cite[Lemma 2.9]{DKLN-pot}; see also \cite[Proposition 2.4]{Brasco-Lindgren}, and \cite[Lemma 2.6]{Brasco-Lindgren-Schikorra}.

\begin{lemma}\label{lem:2nd-Ni-FS}
Let $q\in [1,\infty)$, $\gamma \in (0,1)$, $M>0$, $R>0$, and $d \in (0,R)$. Then, for any 
 $w\in W^{1,q}(B_{R+6d })$ that satisfies
\begin{equation*}
    \int_{B_{R+4d }}\big|\boldsymbol{\tau}^2_h w\big|^q\,\dx
    \le
    M^q|h|^{q (1+\gamma )}\qquad \forall\, 0<|h|\in (0,d ],
\end{equation*}
we have 
$$
\nabla w\in W^{\beta,q}(B_R)\quad \mbox{for any $\beta\in (0,\gamma )$.}
$$
Moreover, there  exists a constant $c$ depending only on $n$ and $q$,
such that 
\begin{align*}
     [\nabla w]_{W^{\beta,q}(B_R)}^q
    &\le
    \frac{c\, d ^{q(\gamma -\beta)}}{(\gamma -\beta) \gamma ^q (1-\gamma )^q}
    \bigg[ M^q + \frac{(R+4d )^{q}}{\beta d ^{q(1+\gamma )}} 
    \int_{B_{R+4d }}|\nabla w|^q\,\dx
    \bigg].
\end{align*}
\end{lemma}

Finally, we state an interpolation estimate from~\cite[Corollary 2.6]{BDLM}, which is obtained by applying a Gagliardo--Nirenberg type inequality of~\cite[Theorem~1]{BM}.

\begin{lemma}\label{lem:GN}
Let $q>1$ and $\gamma\in(0,1)$. 
Then there exists a constant $c=c(n,q,\gamma)>0$ such that every function $u\in W^{1+\gamma,q}(B_R)$ satisfies
\begin{equation*}
    \|\nabla u\|_{L^q(B_R)}
    \le 
    \frac{c}{R}\,\|u\|_{L^q(B_R)}
    + 
    c\, 
    \|u\|_{L^q(B_R)}^{\frac{\gamma}{1+\gamma}}
    [\nabla u]_{W^{\gamma,q}(B_R)}^{\frac{1}{1+\gamma}}.
\end{equation*}
\end{lemma}

\subsection{Tail estimates}

We start with a basic estimate that compares the tail of a function over concentric balls of different radii.

\begin{lemma}\label{lem:t}
Let $p\in(1,\infty)$ and $s\in(0,1)$. 
Then, for any $u\in L^p_{\rm loc}(\Omega)\cap L^{p-1}_{sp}(\mathbb{R}^n)$, any ball $B_{R}(x_o)\Subset\Omega$, and any $r\in(0,R)$, we have
\begin{align*}
    \Tail \big(u; B_r(x_o)\big)^{p-1}
    \le 
     c(n)\Big(\frac{R}{r}\Big)^n
    \Big[
    \Tail \big(u, B_R(x_o)\big)
    + 
    R^{-\frac{n}{p}}\| u\|_{L^p(B_R(x_o))}
    \Big]^{p-1}
\end{align*}
\end{lemma}

\begin{proof}
Consider 
$B_r(x_o)\Subset B_R(x_o)$ and estimate as follows:
\begin{align*}
    \Tail &\big(u; B_r(x_o)\big)^{p-1}\\
    &=r^{sp}\int_{\R^n\setminus B_R(x_o)}\frac{|u(x)|^{p-1}}{|x-x_o|^{n+sp}}\dx+r^{sp}\int_{B_R(x_o)\setminus B_r(x_o)}\frac{|u(x)|^{p-1}}{|x-x_o|^{n+sp}}\dx\\
    &\le
    \Big(\frac{r}{R}\Big)^{sp} \Tail \big(u; B_R(x_o)
    \big)^{p-1} +
    r^{-n} \int_{B_R(x_o)} |u|^{p-1} \,\dx \\
    &\le
    \Big(\frac{r}{R}\Big)^{sp} \Tail \big(u;  B_R(x_o)
    \big)^{p-1} +
    |B_1| \Big(\frac{R}{r}\Big)^n 
    \bigg[\bint_{B_R(x_o)}
    |u|^p\,\dx
    \bigg]^{\frac{p-1}p}\\
    &\le
    c(n)\Big(\frac{R}{r}\Big)^n
    \Big[
    \Tail \big(u, B_R(x_o)\big)
    + 
    R^{-\frac{n}{p}}\| u\|_{L^p(B_R(x_o))}
    \Big]^{p-1},
\end{align*}
which completes the proof.
\end{proof}

The following lemma provides estimates of the tail term on a smaller ball in terms of its value on a larger ball and the local $L^p$-oscillation.

\begin{lemma}\label{lem:tail}
Let \(p\in(1,\infty)\) and \(s\in(0,1)\).
There exists a constant \(c=c(n,s,p)\) with the following property.
For every function
\(u\in L^p_{\mathrm{loc}}(\Omega)\cap L^{p-1}_{sp}(\mathbb{R}^n)\)
and every pair of balls
\(B_r(x_o)\subset B_R(y_o)\Subset\Omega\),
the following estimates hold.
If \(p\ge 2\), then
\begin{align*}
    \Tail \big(u&-(u)_{x_o,r}; B_r(x_o)\big) \\
    &\le 
    \bigg[1 + \frac{|x_o-y_o|}{r}\bigg]^{\frac{n}{p-1}+sp'} 
    \Big(\frac{r}{R}\Big)^{sp'} \Tail \big(u-(u)_{y_o,R}; B_R(y_o)\big) \\
    &\phantom{\le\,} + 
    c\Big(\frac{R}{r}\Big)^{\frac{n}{p-1}}
    \bigg[ \bint_{B_{R}(y_o)} \big|u-(u)_{y_o,R}\big|^{p} \, \d x \bigg]^\frac{1}{p}.
\end{align*}
If \(1<p<2\), then
\begin{align*}
    \Tail \big(u&-(u)_{x_o,r}; B_r(x_o)\big)^{p-1} \\
    &\le
    \bigg[1 + \frac{|x_o-y_o|}{r}\bigg]^{n+sp} 
    \Big(\frac{r}{R}\Big)^{sp} \Tail \big(u-(u)_{y_o,R}; B_R(y_o)\big)^{p-1} \\
    &\phantom{\le\,} + 
    c\Big(\frac{R}{r}\Big)^{n}
    \bigg[ \bint_{B_{R}(y_o)} \big|u-(u)_{y_o,R}\big|^{p} \, \d x \bigg]^\frac{p-1}{p}.
\end{align*}
\end{lemma}

\begin{proof}
Consider first the case $p\ge 2$. Adding and subtracting the average $(u)_{y_o,R}$ gives
\begin{align*}
    \Tail \big(u-(u)_{x_o,r}; B_r(x_o)\big)
    &\le 
    \bigg[r^{sp}
    \int_{\mathbb{R}^n\setminus B_{r}(x_o)}
        \frac{|u-(u)_{y_o,R}|^{p-1}}{|x-x_o|^{n+sp}}\, \d x \bigg]^{\frac{1}{p-1}} \\
    &\phantom{\le\,}
    +
    \bigg[r^{sp}
    \int_{\mathbb{R}^n\setminus B_{r}(x_o)}
        \frac{|(u)_{y_o,R}-(u)_{x_o,r}|^{p-1}}{|x-x_o|^{n+sp}}\, \d x \bigg]^{\frac{1}{p-1}}.
\end{align*}
Decompose
\[
\mathbb{R}^n\setminus B_r(x_o)
=
\big(\mathbb{R}^n\setminus B_R(y_o)\big)
\cup
\big(B_R(y_o)\setminus B_r(x_o)\big),
\]
which gives
\begin{align*}
    \Tail\big(u-(u)_{x_o,r}; B_r(x_o)\big)
    &\le 
    \bigg[r^{sp}
    \int_{\mathbb{R}^n\setminus B_R(y_o)}
        \frac{|u-(u)_{y_o,R}|^{p-1}}{|x-x_o|^{n+sp}}\,\mathrm dx
    \bigg]^{\frac{1}{p-1}} \\
    &\phantom{\le\,} +
    \bigg[r^{sp}
    \int_{B_R(y_o)\setminus B_r(x_o)}
        \frac{|u-(u)_{y_o,R}|^{p-1}}{|x-x_o|^{n+sp}}\,\mathrm dx
    \bigg]^{\frac{1}{p-1}} \\
    &\phantom{\le\,}+
    c\, |(u)_{y_o,R}-(u)_{x_o,r}| \\
    &:= \sfI + \sfI\sfI + c\,\sfI\sfI\sfI,
\end{align*}
where the last term uses that the kernel depends only on the radial variable $|x-y_o|$. Here and throughout, \(c=c(n,s,p)\).

For $1<p<2$, the same argument applies after raising all terms to the power $p-1$.
Adding and subtracting the average $(u)_{y_o,R}$ yields
\begin{align*}
    \Tail\big(u-(u)_{x_o,r}; B_r(x_o)\big)^{p-1}
    &\le 
    r^{sp}
    \int_{\mathbb{R}^n\setminus B_r(x_o)}
        \frac{|u-(u)_{y_o,R}|^{p-1}}{|x-x_o|^{n+sp}}\,\mathrm dx  \\
    &\phantom{\le\,} +
    r^{sp}
    \int_{\mathbb{R}^n\setminus B_r(x_o)}
        \frac{|(u)_{y_o,R}-(u)_{x_o,r}|^{p-1}}{|x-x_o|^{n+sp}}\,\mathrm dx .
\end{align*}
Decomposing the domain of integration as in the case $p\ge 2$ leads to
\begin{align*}
    \Tail\big(u-(u)_{x_o,r}; B_r(x_o)\big)^{p-1}
    &\le 
    r^{sp}
    \int_{\mathbb{R}^n\setminus B_R(y_o)}
        \frac{|u-(u)_{y_o,R}|^{p-1}}{|x-x_o|^{n+sp}}\,\mathrm dx \\
    &\phantom{\le\,} +
    r^{sp}
    \int_{B_R(y_o)\setminus B_r(x_o)}
        \frac{|u-(u)_{y_o,R}|^{p-1}}{|x-x_o|^{n+sp}}\,\mathrm dx  \\
    &\phantom{\le\,} +
    c\, |(u)_{y_o,R}-(u)_{x_o,r}|^{p-1} \\
    &=
    \sfI^{p-1} + \sfI\sfI^{p-1} + c\,\sfI\sfI\sfI^{p-1},
\end{align*}
where $c=c(n,s,p)$ and $\sfI,\sfI\sfI,\sfI\sfI\sfI$ are defined as in the case $p\ge 2$.
It remains to estimate the terms $\sfI$--$\sfI\sfI\sfI$ for all $p>1$. 
Since $|x-x_o|\ge r$ for every $x\in \mathbb{R}^n\setminus B_R(y_o) \subset \mathbb{R}^n\setminus B_r(x_o)$, 
we have
\begin{align*}
    |x-y_o|
    &\le 
    |x-x_o| + |x_o-y_o|
    \le 
    \bigg(1 + \frac{|x_o-y_o|}{r}\bigg) |x-x_o|.
\end{align*}
Hence,
\begin{align*}
    \sfI
    &\le 
    \bigg(1 + \frac{|x_o-y_o|}{r}\bigg)^{\frac{n}{p-1}+s p'}
    \bigg[r^{sp}
    \int_{\mathbb{R}^n\setminus B_R(y_o)}
        \frac{|u-(u)_{y_o,R}|^{p-1}}{|x-y_o|^{n+sp}}\,\mathrm dx
    \bigg]^{\frac{1}{p-1}} \\
    &=
    \bigg(1 + \frac{|x_o-y_o|}{r}\bigg)^{\frac{n}{p-1}+s p'}
    \bigg(\frac{r}{R}\bigg)^{s p'}
    \Tail\big(u-(u)_{y_o,R}; B_R(y_o)\big).
\end{align*}
To estimate $\sfI\sfI$, note that $|x-x_o|\ge r$ for every $x\in B_R(y_o)\setminus B_r(x_o)$.  
Applying H\"older's inequality yields
\begin{align*}
    \sfI\sfI
    &\le
    \bigg[
        r^{-n}
        \int_{B_R(y_o)}
            \big|u-(u)_{{y_o;R}}\big|^{p-1}\,\mathrm dx
    \bigg]^{\frac{1}{p-1}} \\
    &\le c
    \bigg[
        \Big(\frac{R}{r}\Big)^{n}
        \mint_{B_R(y_o)}
            \big|u-(u)_{{y_o;R}}\big|^{p-1}\,\mathrm dx
    \bigg]^{\frac{1}{p-1}} \\
    &\le
    c \Big(\frac{R}{r}\Big)^{\frac{n}{p-1}}
    \bigg[
        \mint_{B_R(y_o)}
            \big|u-(u)_{{y_o;R}}\big|^{p}\,\mathrm dx
    \bigg]^{\frac{1}{p}} .
\end{align*}
Applying H\"older's inequality again and enlarging the domain of integration gives
\begin{align*}
    \sfI\sfI\sfI
    &\le
    \mint_{B_r(x_o)}
        \big|u-(u)_{{y_o;R}}\big|\,\mathrm dx \\
    &\le
    \bigg[
        \mint_{B_r(x_o)}
            \big|u-(u)_{{y_o;R}}\big|^{p}\,\mathrm dx
    \bigg]^{\frac{1}{p}} \\
    &\le
    \Big(\frac{R}{r}\Big)^{\frac{n}{p}}
    \bigg[
        \mint_{B_R(y_o)}
            \big|u-(u)_{{y_o;R}}\big|^{p}\,\mathrm dx
    \bigg]^{\frac{1}{p}} .
\end{align*}
The preceding estimates for $\sfI$, $\sfI\sfI$, and $\sfI\sfI\sfI$ then establish the asserted inequalities.
\end{proof}

A quantitative estimate for the tail of the mean–subtracted function over dyadic enlargements of a ball is as follows.

\begin{lemma}\label{lem:tail-est}
Let $p>1$ and $0<s<1$. 
Then there exists a constant $c=c(n,s,p)>0$ such that for every 
$u\in L^{p}_{\mathrm{loc}}(\Omega)\cap L^{p-1}_{sp}(\mathbb{R}^n)$ 
and every $i\in\mathbb{N}$ with $B_{2^i\rho}(x_o)\Subset\Omega$, the following estimates hold. 
If $p\ge2$, then
\begin{align*}
    \Tail \big(u-(u)_{x_o,\rho}; B_\rho(x_o)\big)
    &\le
    2^{-isp'}
    \Tail \big(u-(u)_{x_o,2^i\rho}; B_{2^i\rho}(x_o)\big) \\
    &\quad +
    c \sum_{k=1}^{i} 2^{-ksp'}
    \bigg[
        \bint_{B_{2^k\rho}(x_o)}
        |u-(u)_{x_o,2^k\rho}|^p\, \d x
    \bigg]^\frac1p .
\end{align*} 
If $1<p<2$, then
\begin{align*}
    \Tail \big(u-(u)_{x_o,\rho}; B_\rho(x_o)\big)^{p-1} 
    &\le
    2^{-isp}
    \Tail \big(u-(u)_{x_o,2^i\rho}; B_{2^i\rho}(x_o)\big)^{p-1} \\
    &\phantom{\le\,} +
    c \sum_{k=1}^{i} 2^{-ksp} \bigg[ \bint_{B_{2^k\rho}(x_o)} \big|u-(u)_{x_o,2^k\rho}\big|^{p} \, \d x \bigg]^\frac{p-1}{p} .
\end{align*}
\end{lemma}
\begin{proof}
We apply Lemma~\ref{lem:tail} iteratively on the dyadic sequence of balls
\(B_\rho(x_o), B_{2\rho}(x_o), \dots, B_{2^{i-1}\rho}(x_o)\).
If \(p\ge 2\), this yields
\begin{align*}
    \Tail \big(u&-(u)_{x_o,\rho}; B_\rho(x_o)\big) \\
    &\le 
    2^{-sp'} \Tail \big(u-(u)_{x_o,2\rho}; B_{2\rho}(x_o)\big) + 
    c \bigg[ \bint_{B_{2\rho}(x_o)} \big|u-(u)_{x_o,2\rho}\big|^{p} \, \d x \bigg]^\frac{1}{p} \\
    &\le 
    2^{-2sp'} \Tail \big(u-(u)_{x_o,4\rho}; B_{4\rho}(x_o)\big) + 
    c\,2^{-sp'}
    \bigg[ \bint_{B_{4\rho}(x_o)} \big|u-(u)_{x_o,4\rho}\big|^{p} \, \d x \bigg]^\frac{1}{p} \\
    &\phantom{\le\,} +
    c \bigg[ \bint_{B_{2\rho}(x_o)} \big|u-(u)_{x_o,2\rho}\big|^{p} \, \d x \bigg]^\frac{1}{p} \\
    &\le 
    \dots \\
    &\le 
    2^{-isp'} \Tail \big(u-(u)_{x_o,2^i\rho}; B_{2^i\rho}(x_o)\big) \\
    &\phantom{\le\,} +  
    c \sum_{k=1}^{i} 2^{-(k-1)sp'} 
    \bigg[ \bint_{B_{2^k\rho}(x_o)} \big|u-(u)_{x_o,2^k\rho}\big|^{p} \, \d x \bigg]^\frac{1}{p} \\
    &=
    2^{-isp'} \Tail \big(u-(u)_{x_o,2^i\rho}; B_{2^i\rho}(x_o)\big) \\
    &\phantom{\le\,} +  
    c\,2^{sp'} \sum_{k=1}^{i} 2^{-ksp'} \bigg[ \bint_{x_o,2^k\rho} \big|u-(u)_{x_o,2^k\rho}\big|^{p} \, \d x \bigg]^\frac{1}{p}
\end{align*}
In the case \(1<p<2\), the same iterative argument applies with the exponent \(p-1\). 
Specifically, we obtain
\begin{align*}
    \Tail \big(u&-(u)_{x_o,\rho}; B_\rho(x_o)\big)^{p-1} \\
    &\le 
    2^{-sp} \Tail \big(u-(u)_{x_o,2\rho}; B_{2\rho}(x_o)\big)^{p-1} + 
    c
    \bigg[ \bint_{B_{2\rho}(x_o)} \big|u-(u)_{x_o,2\rho}\big|^{p} \, \d x \bigg]^\frac{p-1}{p} \\
    &\le 
    \cdots \\
    &\le 
    2^{-isp} \Tail \big(u-(u)_{x_o,2^i\rho}; B_{2^i\rho}(x_o)\big)^{p-1} \\
    &\phantom{\le\,} +  
    c\,2^{sp} \sum_{k=1}^{i} 2^{-ksp} \bigg[ \bint_{B_{2^k\rho}(x_o)} \big|u-(u)_{x_o,2^k\rho}\big|^{p} \, \d x \bigg]^\frac{p-1}{p}.
\end{align*}
\end{proof}

Finally, if two functions coincide outside a certain ball, then their mean-subtracted tails are comparable up to the $L^p$-difference.

\begin{lemma}\label{lem:tv-u}
Let \(p\in(1,\infty)\) and \(s\in(0,1)\).
There exists a constant \(c=c(n,s,p)\) such that the following holds.
Let \(B_r(x_o)\subset B_R(x_o)\Subset\Omega\) be concentric balls and let
\(u,v\in L^p_{\mathrm{loc}}(\Omega)\cap L^{p-1}_{sp}(\mathbb{R}^n)\) satisfy
\(v\equiv u\) on \(\mathbb{R}^n\setminus B_R(x_o)\).
Then
\begin{align*}
    \Tail \big(v&-(v)_{x_o,r}; B_r(x_o)
    \big)\\
    &\le 
    c\Tail \big(u-(u)_{x_o,r}; B_r(x_o)\big) 
    +
        c \Big(\frac{R}{r}\Big)^{\frac{n}{p-1}}
    \bigg[
        \mint_{B_R(x_o)}
            |u-v|^{p}\,\mathrm dx
    \bigg]^{\frac{1}{p}} .
\end{align*}
\end{lemma}

\begin{proof}
By adding and subtracting \(u-(u)_{x_o,r}\) and using Minkowski’s inequality
when \(p\ge 2\), and the quasi-subadditivity of the tail functional
(up to a multiplicative constant) when \(1<p<2\), we obtain
\begin{align*}
    \Tail \big(v-(v)_{x_o,r}; B_r(x_o)\big)
    &\le 
    c\Tail \big(v-u; B_r(x_o)\big) \\
    &\phantom{\le\,} +
    c\Tail \big(u-(u)_{x_o,r}; B_r(x_o)\big) \\
    &\phantom{\le\,} +
    c\Tail \big((u)_{x_o,r}-(v)_{x_o,r}; B_r(x_o)\big) .
\end{align*}
It remains to estimate the first and the third term on the right-hand side.
Using the assumption \(v\equiv u\) on \(\R^n\setminus B_R(x_o)\), the radial
dependence of the kernel, and Hölder’s inequality, we infer
\begin{align*}
    \Tail \big(v-u; B_r(x_o)\big)
    &=
    \bigg[r^{sp}
    \int_{B_R(x_o)\setminus B_r(x_o)}
        \frac{|v-u|^{p-1}}{|x-x_o|^{n+sp}}\,\mathrm dx \bigg]^{\frac{1}{p-1}} \\
    &\le c
    \bigg[
        \Big(\frac{R}{r}\Big)^{n}
        \mint_{B_R(x_o)}
            |u-v|^{p-1}\,\mathrm dx
    \bigg]^{\frac{1}{p-1}} \\
    &\le
    c \Big(\frac{R}{r}\Big)^{\frac{n}{p-1}}
    \bigg[
        \mint_{B_R(x_o)}
            |u-v|^{p}\,\mathrm dx
    \bigg]^{\frac{1}{p}} .
\end{align*}
The third term can be treated analogously. Indeed, 
\begin{align*}
    \Tail \big((u)_{x_o,r}-(v)_{x_o,r}; B_r(x_o)\big)
    &\le 
    c\, |(u)_{x_o,r}-(v)_{x_o,r}| \\
    &\le 
    c\, \bint_{B_r(x_o)} |u-v| \,\dx   \\
    &\le 
    c \Big(\frac{R}{r}\Big)^{\frac{n}{p}}
    \bigg[
        \mint_{B_R(x_o)}
            |u-v|^{p}\,\mathrm dx
    \bigg]^{\frac{1}{p}} .
\end{align*}
Combining the above estimates yields the claim.
\end{proof}

\subsection{The homogeneous problem}
In this section we summarize the regularity statements of local weak solutions to the homogeneous fractional $p$-Laplace equation 
\begin{equation} \label{eq:frac-p-lap-homo}
    \left( - \Delta_p\right)^s v = 0
    \quad \mbox{in $\Omega$.}
\end{equation}
We start with the $L^\infty$-estimate from~\cite[Theorem 1.1]{DKP} (see also \cite[Theorem~6.2]{Coz17b}), which remains valid in the setting of variable coefficients.

\begin{theorem}\label{thm:sup-est}
Let $p \in (1,\infty)$ and $s \in (0,1)$. Assume that $v \in W^{s,p}_{\loc}(\Omega) \cap L^{p-1}_{sp}(\R^n)$ is a local weak solution to the homogeneous problem~\eqref{eq:frac-p-lap} in the sense of Definition~\ref{def:weak-sol}, with coefficients $a$ satisfying~\eqref{eq:a-condition}$_1$ and $f \equiv 0$. Then, $v$ is locally bounded in $\Omega$, i.e.~$v\in L^\infty_{\loc}(\Omega)$. Moreover, there exists a constant $c=c(n,p,s,C_o,C_1)$ such that for any ball $B_{2R} = B_{2R}(x_o) \Subset \Omega$ we have 
\begin{align*}
    \|v\|_{L^\infty(B_R)} 
    \le 
    \frac{c}{R^{\frac{n}{p}}} \|v\|_{L^p(B_{2R})} + c\, \Tail(v;B_R).
\end{align*}
\end{theorem}

In addition, the weak gradient of a local weak solution to~\eqref{eq:frac-p-lap-homo} exists and is locally in $L^q$ for every $q < \infty$, see~\cite[Theorem 1.4]{BDLMS-1} and \cite[Theorem~1.2]{BDLMS-2}.

\begin{theorem} \label{thm:Lq-gradient}
Let $p\in(1,\infty)$ and $s \in \big(\frac{(p-2)_+}{p},1\big)$. Then, for any local weak solution $v \in W^{s,p}_{\loc}(\Omega) \cap L^{p-1}_{sp}(\R^n)$ to the homogeneous problem~\eqref{eq:frac-p-lap-homo} in the sense of Definition~\ref{def:weak-sol}, we have 
$$
    v \in W^{1,q}_{\loc}(\Omega) \quad 
    \text{ for any } q \in [p,\infty).
$$
Moreover, there exists a constant $c =  c(n,p,s,q)$ such that for any ball $B_{2R} \equiv B_{2R}(x_o) \Subset \Omega$ we have
$$
    \| \nabla v \|_{L^q(B_{R})} 
    \leq 
    c\, R^{\frac{n}{q}-1} \left[ R^{s- \frac{n}{p}} [v]_{W^{s,p}(B_{2R})} + \Tail \big(v-(v)_{{2R}};B_{2R}\big) \right].
$$
\end{theorem}

\begin{proof}
By Theorem~\ref{thm:sup-est}, $v$ is locally bounded. Setting
$w:=v-(v)_{2R}$, we note that $w$ is also a locally bounded local weak
solution to~\eqref{eq:frac-p-lap-homo}. Hence, by
\cite[Theorem~1.4]{BDLMS-1} if $p\ge2$, respectively
\cite[Theorem~1.2]{BDLMS-2} if $p\in(1,2)$, it follows that
$v\in W^{1,q}_{\loc}(\Omega)$ for every $q\in[p,\infty)$. Moreover, there
exists a constant $c=c(n,p,s,q)$ such that for any ball
$B_{2R}\equiv B_{2R}(x_o)\Subset\Omega$,
\begin{align*}
    \|\nabla v\|_{L^q(B_R)}
    &\le
    c\,R^{\frac{n}{q}-1}
    \Big[
        R^{s-\frac{n}{p}} [v]_{W^{s,p}(B_{\frac32 R})}
        + \|w\|_{L^\infty(B_{\frac32 R})}
        + \Tail\big(w;B_{\frac32 R}\big)
    \Big].
\end{align*}
By Theorem~\ref{thm:sup-est} and the fractional Poincar\'e inequality from
Lemma~\ref{lem:poin}, we obtain
\begin{align*}
    \|v-(v)_{2R}\|_{L^\infty(B_{\frac32 R})}
    &\le
     \frac{c}{R^{\frac{n}{p}}} \|v-(v)_{2R}\|_{L^p(B_{2R})}
     + c\,\Tail\big(w;B_{\frac32 R}\big)\\
     &\le
     c\, R^{s-\frac{n}{p}} [v]_{W^{s,p}(B_{2R})}
     + c\,\Tail\big(w;B_{\frac32 R}\big).
\end{align*}
It remains to estimate the tail term. Applying Lemma~\ref{lem:t} to
$B_{\frac32 R}$ and $B_R\equiv B_{2R}$, and using again
Lemma~\ref{lem:poin}, we obtain
\begin{align*}
    \Tail\big(w;B_{\frac32 R}\big)
    &\le 
    c(n)\Big[
        \Tail \big(w; B_{2R}\big)
        + 
        R^{-\frac{n}{p}}\| w\|_{L^p(B_{2R}(x_o))}
    \Big] \\
    &\le 
    c(n)\Big[
        \Tail \big(w; B_{2R}\big)
        + 
        R^{s-\frac{n}{p}} [v]_{W^{s,p}(B_{2R})}
    \Big].
\end{align*}
Collecting the above estimates yields the claimed local $L^q$ gradient estimate.
\end{proof}

Finally, local weak solutions to~\eqref{eq:frac-p-lap-homo} admit higher differentiability, which is stated here in terms of Besov spaces. 

\begin{lemma}\label{lem:2nd-diffquot-homo}
Let $p>1$ and $s\in\big(\frac{(p-2)_+}{p},1\big)$. There exists a constant
$c=c(n,p,s)>0$ such that, if
$v\in W^{s,p}_{\loc}(\Omega)\cap L^{p-1}_{sp}(\R^n)$ is a locally bounded
local weak solution to~\eqref{eq:frac-p-lap-homo} in the sense of Definition~\ref{def:weak-sol}  and
$B_{4R}\equiv B_{4R}(x_o)\Subset\Omega$, then
\begin{align*}
    \int_{B_R} \big|\boldsymbol{\tau}_h^2 v\big|^p \,\dx
    &\le
    c\Big(\frac{|h|}{R}\Big)^{\theta}
    \Big[
        R^{sp}[v]_{W^{s,p}(B_{4R})}^p
        + R^{n}\Tail\big(v-(v)_{4R};B_{4R}\big)^p
    \Big]
\end{align*}
for every $h\in\R^n\setminus\{0\}$ with $|h|\le \tfrac17 R$, where
\begin{equation}\label{def:theta}
    \theta :=
    \left\{
    \begin{array}{cl}
        p+\frac12 sp, & \text{if } p\in(1,2),\\[3pt]
        sp+2, & \text{if } p\ge2.
    \end{array}
    \right.
\end{equation}
\end{lemma}

\begin{proof}
Set
\[
    w:=v-(v)_{4R} .
\]
Then $w$ is also a local weak solution to~\eqref{eq:frac-p-lap-homo}.
If $p\ge2$, applying \cite[Lemma~5.7]{BDLMS-1} to $w$ with
$(q,r,R)=(p,R,2R)$ yields, for any $0<|h|\le \tfrac17 R$,
\begin{align*}
    \int_{B_{R}} |\boldsymbol{\tau}_h^2 v|^p \,\dx
    &=
    \int_{B_{R}} |\boldsymbol{\tau}_h^2 w|^p \,\dx
    \le
    c \Big(\frac{|h|}{R}\Big)^{sp+2}\sfK^p .
\end{align*}
If $p\in(1,2)$, \cite[Lemma~5.1]{BDLMS-2} implies that, for any
$\epsilon\in(0,1-s)$,
\begin{align*}
    \int_{B_{R}} |\boldsymbol{\tau}_h^2 v|^p \,\dx
    &=
    \int_{B_{R}} |\boldsymbol{\tau}_h^2 w|^p \,\dx
    \le
    c \Big(\frac{|h|}{R}\Big)^{p+sp \frac{p}{2}-\epsilon p(1-\frac{p}{2})}\sfK^p .
\end{align*}
Here
\[
    \sfK^p
    :=
    R^p \int_{B_{2R}} |\nabla w|^p \,\dx
    +
    R^n \|w\|_{L^\infty(B_{2R})}^p
    +
    R^n \Tail\big(w;B_{2R}\big)^p .
\]
Choosing $\epsilon\in(0,1-s)$ sufficiently small, we have
\[
    p+ sp\tfrac{p}{2}-\epsilon p\Big(1-\tfrac{p}{2}\Big)
    \ge
    p+\tfrac12 sp .
\]
It remains to estimate $\sfK$.
We estimate the contributions separately, starting with the gradient term.
By Theorem~\ref{thm:Lq-gradient},
\begin{align*}
    R^p \int_{B_{2R}} |\nabla w|^p \,\dx
    \le
    c\Big[
        R^{sp}[w]_{W^{s,p}(B_{4R})}^p
        +
        R^{n}\Tail\big(w;B_{4R}\big)^p
    \Big].
\end{align*}
By Theorem~\ref{thm:sup-est} and the fractional Poincar\'e inequality
(Lemma~\ref{lem:poin}), we obtain
\begin{align*}
    \|v-(v)_{4R}\|_{L^\infty(B_{2 R})}
    &\le
    \frac{c}{R^{\frac{n}{p}}} \|v-(v)_{4R}\|_{L^p(B_{4R})}
    + c\,\Tail\big(w;B_{2R}\big)\\
    &\le
    c\, R^{s-\frac{n}{p}} [v]_{W^{s,p}(B_{4R})}
    + c\,\Tail\big(w;B_{2 R}\big).
\end{align*}
By Lemma~\ref{lem:t} and Lemma~\ref{lem:poin},
\begin{align*}
    \Tail\big(w; B_{2R}\big)
    &\le 
    c(n)\Big[
        \Tail \big(w; B_{4R}\big)
        + R^{-\frac{n}{p}}\| w\|_{L^p(B_{4R})}
    \Big] \\
    &\le 
    c(n)\Big[
        \Tail \big(w; B_{4R}\big)
        + R^{s-\frac{n}{p}} [v]_{W^{s,p}(B_{4R})} 
    \Big].
\end{align*}
Using the above inequalities and the identity $[w]_{W^{s,p}}=[v]_{W^{s,p}}$, 
\[
    \sfK^p
    \le 
    c \Big[
        R^{sp} [v]_{W^{s,p}(B_{4R})}^p
        + R^n \Tail\big(v-(v)_{4R}; B_{4R}\big)^p
    \Big].
\]
\end{proof}

\section{Higher differentiability for fractional $p$-Poisson equation}\label{S:3}

In this section we consider inhomogeneous fractional $p$-Laplace equations of the type  
\begin{equation} \label{eq:frac-p}
        \left( - \Delta_p\right)^s u = f \quad \text{ in } \Omega,
\end{equation}
Our goal is to prove higher differentiability of local weak solutions of order $W^{1+\gamma,p}$ for some $\gamma>0$ as stated in Theorem~\ref{thm:higher-diff} whose abridged version has been stated in Theorem~\ref{Thm:2}. The proof relies on the results for the homogeneous problem and a comparison argument.

\subsection{Definition and Role of the Parameters}
Due to the definition of $r_*$ from Theorem~\ref{Thm:2}, we shall
distinguish between the cases
\begin{equation*}
s \le s_o
\quad\text{and}\quad
s > s_o
\quad\mbox{where } s_o:=\frac{1}{p'}\bigg[1+\frac{n}{p(p-1)}\bigg].
\end{equation*}
More precisely, in the latter regime
the assumption \(f\in L_{\loc}^{1}(\Omega)\) already suffices to guarantee
\(\nabla u\in L^{p}_{\mathrm{loc}}(\Omega)\).
In contrast, when
\(s \le s_o\),
one needs to impose the stronger condition
\[
f\in L_{\loc}^{\frac{np}{(p-1)[\,n+p(sp'-1)\,]}}(\Omega)
\]
in order to ensure the same local integrability of the gradient. To unify the two regimes, we introduce the parameter
\begin{align}\label{def:alpha}
    \alpha 
    &:=
    \max\bigg\{\frac{n}{(p-1)\,[\,n+p(sp'-1)\,]}\,,\, \frac{1}{p}\bigg\} \nonumber \\[6pt]
    &=\left\{
    \begin{array}{cl}
    \displaystyle
        \frac{n}{(p-1)[n+p(sp'-1)]}, & \mbox{if $\displaystyle s\in(\tfrac{p-1}{p},s_o]$,}\\[12pt]
        \displaystyle\frac{1}{p}, &  \mbox{if $\displaystyle s\in(s_o,1)$.}
    \end{array}
    \right.
\end{align}
Note that \(\alpha\in[\tfrac{1}{p},\,\tfrac{1}{p-1})\) in either case.
This necessitates an improved integrability assumption on the right-hand side $f$.
Accordingly, we shall frequently work with a parameter
$\tilde\alpha\in[\alpha,\tfrac{1}{p-1}]$ and introduce
\begin{align}\label{def:eps}
    \varepsilon
    &:=
    sp' - 1 - \frac{n}{p}\bigg[\frac{1}{\tilde\alpha (p-1)}-1\bigg]
    \in [0,\, sp'-1].
\end{align}
To see that $\varepsilon\ge 0$, we replace $\tilde\alpha$ with $\alpha$ in
\eqref{def:eps} and invoke the case distinction in the definition of $\alpha$
in \eqref{def:alpha}. 

The following lemma illustrates the interplay of the relevant parameters in controlling the integral of $f$ against a function $w$ supported in a smaller ball $B_R$. The parameter $\tilde\alpha$ is chosen at least $\alpha$ to guarantee sufficient integrability of $f$, while the requirement $\varepsilon>0$ quantifies a small gain in the scaling with respect to the radius $R$ (see~\eqref{def:eps}).

\begin{lemma}\label{lem:fw}
Let $p>1$ and $s\in(0,1)$ with $sp'>1$, and let $\tilde\alpha\in[\alpha,\frac{1}{p-1}]$ be such that $\epsilon>0$, where $\alpha$ and $\epsilon$ are defined in~\eqref{def:alpha} and~\eqref{def:eps}. 
Then, there exists a constant
$c=c(n,s,p,\tilde\alpha)>0$ such that for every ball
$B_{2R}\equiv B_{2R}(x_o)\Subset\Omega$, every function
$w\in W^{s,p}(B_{2R})$ with $\operatorname{spt} (w)\subset B_R$,
every $f\in L^{\tilde\alpha p}(B_R)$, and every $\tilde\delta\in(0,1]$,
the estimate
\begin{align*}
    \int_{B_R} f w \,\mathrm{d}x
    \le
    \tilde\delta\,[w]_{W^{s,p}(B_{2R})}^p
    +
    c\,\tilde\delta^{-\frac{1}{p-1}} R^{(1-s+\varepsilon)p}
    \|f\|_{L^{\tilde\alpha p}(B_R)}^{p'}
\end{align*}
holds.
\end{lemma}

\begin{proof}
We introduce the rescaled functions
\[
    \widetilde w(x):= R^{-s} w(Rx),
    \qquad
    \widetilde f(x):= R^{s} f(Rx).
\]
With this choice we have $\widetilde w\in W^{s,p}(B_2)$,
$\operatorname{spt}(\widetilde w)\subset B_1$, and
$\widetilde f\in L^{\tilde\alpha p}(B_1)$.
Since $\tilde\alpha p>\alpha p\ge1$, Hölder’s inequality yields
\[
    \int_{B_1} \widetilde f\,\widetilde w\,\mathrm{d}x
    \le
    \|\widetilde w\|_{L^{\frac{\tilde\alpha p}{\tilde\alpha p-1}}(B_1)}
    \|\widetilde f\|_{L^{\tilde\alpha p}(B_1)}.
\]
Then the fractional Sobolev embedding Lemma~\ref{lem:Sobolev} yields
\begin{equation}\label{sob}
    \|\widetilde w\|_{L^{\kappa}(B_1)}
    \le 
    c\, \|\widetilde w\|_{W^{s,p}(B_2)},
\end{equation}
where $\kappa$ is as in \eqref{def:kappa} with $\gamma=s$ and $c=c(n,s,p,\kappa)$ and 
Note that the dependence of the embedding constant on the exponent $\kappa$ appears only in the borderline case $sp=n$.
To combine the preceding two estimates, it suffices to check that
\[
    \frac{\tilde\alpha p}{\tilde\alpha p-1} \le \kappa,
\]
which indeed holds in all relevant cases. For $sp\ge n$ the claim is immediate, since $\tilde\alpha p\in(1,\infty)$ and hence $\frac{\tilde\alpha p}{\tilde\alpha p-1}\in(1,\infty)$. In the case $sp<n$, one first observes that $p(sp'-1)<sp'$, and hence
\[
    \tilde\alpha > \alpha \ge \frac{n}{(p-1)[n+p(sp'-1)]} > \frac{n}{(p-1)[n+sp']} = \frac{n}{n(p-1)+sp}.
\]
It follows that
\[
    \frac{\tilde\alpha p}{\tilde\alpha p-1} = 1 + \frac{1}{\tilde\alpha p - 1} < 1 + \frac{n(p-1)+sp}{n-sp} = \frac{np}{n-sp} = \kappa,
\]
as desired. Hence, we may choose $\kappa = \frac{\tilde\alpha p}{\tilde\alpha p-1}$ in~\eqref{sob}, yielding
\[
    \int_{B_1} \widetilde f \,\widetilde w\,\dx
    \le 
    c\, \|\widetilde w\|_{W^{s,p}(B_2)} \|\widetilde f\|_{L^{\tilde\alpha p}(B_1)}
    \le 
    c \,[\widetilde w]_{W^{s,p}(B_2)} \|\widetilde f\|_{L^{\tilde\alpha p}(B_1)},
\]
where in the second inequality we used Lemma~\ref{lem:coz17b-4.7} together with the fact that $\widetilde w = 0$ almost everywhere on $B_2\setminus B_1$.
Scaling back to $(B_R,B_{2R})$, we note that
\[
    [\widetilde w]_{W^{s,p}(B_1)} = R^{-s}[w]_{W^{s,p}(B_R)},\qquad 
    \|\widetilde f\|_{L^{\tilde\alpha p}(B_1)} = R^s \|f\|_{L^{\tilde\alpha p}(B_R)}.
\]
Hence, the preceding inequality becomes
\begin{align*}
    \int_{B_R} f w \,\dx
    &= R^n \int_{B_1} \widetilde f\, \widetilde w\,\dx \\
    &\le c\, R^n [\widetilde w]_{W^{s,p}(B_2)} \|\widetilde f\|_{L^{\tilde\alpha p}(B_1)} \\
    &= c\, R^{n-s} [w]_{W^{s,p}(B_{2R})} \, R^{s-\frac{n}{\tilde\alpha p}} \|f\|_{L^{\tilde\alpha p}(B_R)} \\
    &= c\, R^{n - \frac{n}{p} - \frac{n}{\tilde\alpha p} + s} [w]_{W^{s,p}(B_{2R})} \|f\|_{L^{\tilde\alpha p}(B_R)}.
\end{align*}
Using the definition of $\epsilon$ in~\eqref{def:eps}, we rewrite the exponent of $R$ as
\begin{align*}
    n - \frac{n}{p} + s - \frac{n}{\tilde\alpha p} 
    &= \frac{n(p-1)}{p} + s - (p-1)(sp'-1) - \frac{n(p-1)}{p} + \epsilon(p-1) \\
    &= s - (sp+p-1) + \epsilon(p-1) \\
    &= (1-s+\epsilon)(p-1).
\end{align*}
 Hence, the preceding inequality yields
\begin{align*}
    \int_{B_R} f w \,\dx
    &\le 
    c \,R^{(1-s+\epsilon)(p-1)} [w]_{W^{s,p}(B_{2R})} \|f\|_{L^{\tilde\alpha p}(B_R)} \\
    &\le 
    \tilde\delta \,[w]_{W^{s,p}(B_{2R})}^p
    + c\, \tilde\delta^{-\frac{1}{p-1}} R^{(1-s+\epsilon)p} \|f\|_{L^{\tilde\alpha p}(B_R)}^{p'},
\end{align*}
where Young's inequality was used in the second step. 
\end{proof}

\subsection{Comparison estimates}\label{S:comp-est} 
The following lemma provides a quantitative comparison between a local weak
solution $u$ of \eqref{eq:frac-p} and the solution $v$ of the associated
homogeneous Dirichlet problem. In the superquadratic case $p\ge 2$ the
difference can be controlled solely in terms of the inhomogeneity. In contrast,
for $1<p<2$ the estimate necessarily involves a small multiple of the
fractional energy of $u$, reflecting the subquadratic nature of the problem.
The precise setup is as follows. Let $B_R(x_o)\Subset\Omega$ be a fixed ball.
We denote by
\[
v\in W^{s,p}(B_R(x_o))\cap L^{p-1}_{sp}(\R^n)
\]
the unique weak solution to the homogeneous Dirichlet problem
\begin{equation}\label{CD-homo} 
\left\{
    \begin{array}{cl}
      (-  \Delta_p)^s v  = 0 & \mbox{in $B_R(x_o)$,} \\[6pt]
      v = u & \mbox{a.e.~in $\R^n \setminus B_R(x_o)$.}
    \end{array}
\right.
\end{equation}
\begin{lemma}[Comparison estimate]\label{lem:fract-level-comparison}
Let $p>1$ and $s\in(0,1)$ with $sp'>1$, and let $\tilde\alpha\in[\alpha,\frac{1}{p-1}]$ be such that $\epsilon>0$, where $\alpha$ and $\epsilon$ are defined in~\eqref{def:alpha} and~\eqref{def:eps}.
Suppose that $u$ is a local weak solution to~\eqref{eq:frac-p} in the sense of Definition~\ref{def:weak-sol}, and let $B_{2R} \equiv B_{2R}(x_o) \Subset \Omega$. Denote by $v$ the unique weak solution to the homogeneous Dirichlet problem~\eqref{CD-homo} on $B_R(x_o)$ in the sense of Definition~\ref{def:weak-sol-D}, and set
\[
w := u-v.
\]
Then, there exists a constant $c = c(n,p,s,\tilde \alpha) > 0$ such that the following estimates hold.

\medskip
\noindent
\emph{(i) Superquadratic case.} 
If $p\in[2,\infty)$, then
\[
R^{-sp}\|w\|_{L^p(B_R)}^p + [w]_{W^{s,p}(\R^n)}^p
\le c \, R^{(1-s+\epsilon)p} \, \|f\|_{L^{\tilde\alpha p}(B_R)}^{p'}.
\]

\medskip
\noindent
\emph{(ii) Subquadratic case.} 
If $p\in(1,2]$,  then for every $\delta\in(0,1]$ we have
\[
R^{-sp}\|w\|_{L^p(B_R)}^p + [w]^p_{W^{s,p}(B_{2R})}
\le \delta \, [u]^p_{W^{s,p}(B_{2R})}
+ \frac{c}{\delta^{\frac{2-p}{p-1}}} \, R^{(1-s+\epsilon)p} \, \|f\|_{L^{\tilde\alpha p}(B_R)}^{p'}.
\]
\end{lemma}

\begin{proof}
The argument follows closely the proof of~\cite[Lemma~3.2]{BDLM}. The only modification concerns the treatment of the term involving the inhomogeneity $f$ in~\cite[(3.4)]{BDLM}. 
Instead of using~\cite[(3.6)]{BDLM}, we estimate the integral
\[
\int_{B_R} f w \,\mathrm{d}x
\] 
by applying Lemma~\ref{lem:fw}. This gives, for any $\tilde\delta \in (0,1]$,
\[
\int_{B_R} f w \,\mathrm{d}x
\le
\tilde\delta \,[w]_{W^{s,p}(B_{2R})}^p
+ c\,\tilde\delta^{-\frac{1}{p-1}} R^{(1-s+\epsilon)p} \|f\|_{L^{\tilde\alpha p}(B_R)}^{p'}.
\]
Using this estimate, one can control the $W^{s,p}$-seminorm of $w$ exactly as in~\cite[Lemma~3.2]{BDLM}. The corresponding bound for the $L^p$-norm of $w$ then follows from Sobolev's embedding, as in~\cite[Corollary~3.3]{BDLM}.
\end{proof}

\subsection{Higher differentiability of the gradient}
Let $u$ be a local weak solution to~\eqref{eq:frac-p}.
We denote by $v$ the unique weak solution of the Dirichlet problem
\begin{equation}\label{CD-v}
\left\{
    \begin{array}{cl}
      (-\Delta_p)^s v = 0 & \text{in } B_{8R}(x_o), \\[6pt]
      v = u & \text{a.e.\ in } \mathbb{R}^n\setminus B_{8R}(x_o).
    \end{array}
\right.
\end{equation}
The function $v$ serves as a fractional $p$-harmonic replacement of $u$ and allows us to
separate the regularity of the homogeneous operator from the
contribution of the inhomogeneity $f$.
The following lemma provides a quantitative estimate for second-order
finite differences of $v$.

\begin{lemma}\label{lem:vi-1}
Let $p>1$ and $s\in(0,1)$ with $sp'>1$, and let $\tilde\alpha\in[\alpha,\frac{1}{p-1}]$ be such that $\epsilon>0$, where $\alpha$ and $\epsilon$ are defined in~\eqref{def:alpha} and~\eqref{def:eps}. 
Then there exists a constant $c=c(n,p,s,\tilde\alpha)$ such that the following holds.
Whenever 
\(u \in W^{s,p}_{\rm loc}(\Omega)\cap L^{p-1}_{sp}(\mathbb{R}^n)\)
is a local weak solution to~\eqref{eq:frac-p} in the sense of Definition~\ref{def:weak-sol},
$B_{8R}(x_o)\Subset\Omega$, and 
\(v \in W^{s,p}(B_{8R}(x_o))\cap L^{p-1}_{sp}(\mathbb{R}^n)\)
is the unique weak solution to \eqref{CD-v} in the sense of Definition~\ref{def:weak-sol-D},
then for any step size \(h\in\mathbb{R}^n\) with \(0<|h|\le \tfrac17 R\) there holds
\begin{align*}
    \int_{B_R(x_o)} \big|\boldsymbol{\tau}_h^2 v\big|^p \,\dx  
    &\leq 
    c \Big(\frac{|h|}{R}\Big)^{\theta} 
    \Bigg[ R^{s p} [u]_{W^{s,p}(B_{4R}(x_o))}^p  \\
    &\qquad\qquad\qquad\quad + 
    R^n \Tail \big(u - (u)_{x_o,8R}; B_{8R}(x_o)\big)^p \\
    &\qquad\qquad\qquad\quad + 
    R^{(1+\epsilon)p} \|f\|_{L^{\tilde\alpha p}(B_{8R}(x_o))}^{p'}\Bigg] .
\end{align*}
Here $\theta$ is given by~\eqref{def:theta}.
\end{lemma}

\begin{proof}
In the following, we omit $x_o$ from the notation for simplicity.
We apply Lemma~\ref{lem:2nd-diffquot-homo}  to $v$ with the choice $(q,r,R)=(p,R,2R)$, yielding that
\begin{align*}
    \int_{B_{R}} |\boldsymbol{\tau}_h^2 v|^p \, \dx 
    \leq 
    c \Big(\frac{|h|}{R}\Big)^{\theta} 
    \sfK^p, 
\end{align*} 
for any $0<|h|\le \frac17 R$, where
$$
    \sfK^p
    :=
    R^{s p} [v]_{W^{s,p}(B_{4R})}^p +
    R^{n} \Tail\big(v-(v)_{4R}; B_{4R}\big)^p .
$$
We estimate the two terms as follows. 
Applying the comparison estimate from Lemma~\ref{lem:fract-level-comparison} 
to $v-u$ on $B_{8R}(x_o)$ with the choice $\delta=1$, we obtain
\begin{align*}
    R^{s p} [v]_{W^{s,p}(B_{4R})}^p 
    &\le 
    R^{s p} [v]_{W^{s,p}(B_{8R})}^p \\
    &\le c\Big[
    R^{s p} [u]_{W^{s,p}(B_{8R})}^p+
    R^{s p} [u-v]_{W^{s,p}(B_{8R})}^p
    \Big]\\
    &
    \le c\Big[
    R^{s p} [u]_{W^{s,p}(B_{8R})}^p
    +
    R^{(1+\epsilon)p} \|f\|_{L^{\tilde\alpha p}(B_{8R})}^{p'}
    \Big].
\end{align*}
It remains to replace the function $v$ in the tail term by $u$.
To this end, we first note that $v=u$ in $\mathbb{R}^n\setminus B_{8R}$, so that Lemma~\ref{lem:tv-u} is applicable.
Consequently, the tail of $v$ can be controlled in terms of the tail of $u$. This is followed by Lemma~\ref{lem:tail}, the fractional Poincar\'e inequality from Lemma~\ref{lem:poin} and the comparison estimate from
Lemma~\ref{lem:fract-level-comparison} with the choice $\delta=1$. Hence,
 \begin{align*}
    R^{n} & \Tail\big(v-(v)_{4R}; B_{4R}\big)^p \\
    &\le 
    c\,R^{n} \Tail\big(u-(u)_{4R}; B_{4R}\big)^p +
    c\,\|u-v\|^p_{L^p(B_{8R})} \Big] \\
    &=
    c\Big[ R^{n} \Tail\big(v-(u)_{8R}; B_{8R}\big)^p +
    \|u-(u)_{8R}\|^p_{L^p(B_{8R})} + 
    \|u-v\|^p_{L^p(B_{8R})} \Big]\\
    &\le 
    c\Big[ R^n\Tail\big(u- (u)_{8R}; B_{8R}\big)^p +
    R^{s p} [u]_{W^{s,p}(B_{8R})}^p +
    R^{(1+\epsilon)p} \|f\|_{L^{\tilde\alpha p}(B_{8R})}^{p'} \Big] .
\end{align*}
By combining the two preceding estimates with the bound for the second-order differences, and noting that $sp+sp'=spp'$, the claimed inequality follows. 
\end{proof}

We denote 
$$
    \sfE\big(u;B_R(x_o)\big) 
    := 
    \bigg[ \bint_{B_R(x_o)} |u - (u)_{x_o;R}|^p \, \d x \bigg]^\frac{1}{p} + 
    \Tail \big( u - (u)_{x_o;R} ; B_R(x_o) \big).
$$

\begin{lemma}\label{lem:vi}
Let $p>1$ and $s\in(0,1)$ with $sp'>1$, and let $\tilde\alpha\in[\alpha,\frac{1}{p-1}]$ be such that $\epsilon>0$, where $\alpha$ and $\epsilon$ are defined in~\eqref{def:alpha} and~\eqref{def:eps}. 
Then there exists a constant $c=c(n,p,s,\tilde\alpha)$ such that the following holds.
Whenever
\(u \in W^{s,p}(B_\rho(y_o))\cap L^{p-1}_{sp}(\mathbb{R}^n)\)
is a weak solution in \(B_\rho(y_o)\) satisfying 
\(u \in W^{\sigma,p}(B_\rho(y_o))\) for some \(\sigma\in[s,1)\), 
\(x_o\in B_{\frac12 \rho}(y_o)\),  \(R\in(0,\frac1{16}\rho]\),
and
\(v \in W^{s,p}(B_{8R}(x_o))\cap L^{p-1}_{sp}(\mathbb{R}^n)\)
is the unique weak solution to \eqref{CD-v},
then for any \(m_o\in\mathbb{N}_0\) such that
\(\tfrac{1}{32}\rho<2^{m_o}R\le\tfrac{1}{16}\rho\) and any
step size \(h\in\mathbb{R}^n\) with \(0<|h|\le \tfrac17 R\), there holds
\begin{align*}
    \int_{B_R(x_o)} \big|\boldsymbol{\tau}_h^2 v\big|^p \,\dx
    &\le
    c \Big(\frac{|h|}{R}\Big)^{\theta}
    \Bigg[
        R^{\sigma p}
        \Bigg(
            \sum_{k=0}^{m_o}
            2^{k(\sigma-sp' - \frac{n}{p})}
            [u]_{W^{\sigma,p}(B_{2^k 8R}(x_o))}
        \Bigg)^p \\
    &\qquad\qquad\qquad
        + \Big(\frac{R}{\rho}\Big)^{spp'} R^n\sfE\big(u;B_\rho(y_o)\big)^p \\
    &\qquad\qquad\qquad + 
        R^{(1+\epsilon)p}
        \|f\|_{L^{\tilde\alpha p}(B_{8R}(x_o))}^{p'}
    \Bigg],
\end{align*}
where $\theta$ is given in~\eqref{def:theta}.
\end{lemma}

\begin{proof}
We apply Lemma~\ref{lem:vi-1} and then use the fractional embedding~\eqref{eq:fract-embedding-s-sigma}
to replace the term $R^{s p} [u]_{W^{s,p}(B_{4R}(x_o))}^p$ by
$R^{\sigma p} [u]_{W^{\sigma,p}(B_{8R}(x_o))}^p$.
This yields
\begin{align*}
    \int_{B_{R}(x_o)} |\boldsymbol{\tau}_h^2 v|^p \, \dx 
    &\leq
    c \Big(\frac{|h|}{R}\Big)^{\theta} 
    \Big[ R^{\sigma p} [u]_{W^{\sigma,p}(B_{8R}(x_o))}^p \\
    &\qquad\qquad\qquad\quad + 
    R^n \Tail \big(u - (u)_{x_o;8R}; B_{8R}(x_o)\big)^p \\
    &\qquad\qquad\qquad\quad + 
    R^{(1+\epsilon)p} \|f\|_{L^{\tilde\alpha p}(B_{8R}(x_o))}^{p'} \Big].
\end{align*}
At this stage, it remains to estimate the tail contribution on the right-hand side. 
Let $m_o\in\mathbb{N}$ be the largest integer such that 
$B_{2^{m_o}8R}(x_o)\Subset B_\rho(y_o)$. 
From the tail estimate in Lemma~\ref{lem:tail-est} in the case $p\ge2$, it follows that
\begin{align} \label{eq:tail-rhoi-estimate}
    \Tail \big(u-(u)_{x_o;8R}; B_{8R}(x_o)\big) 
    &\leq 
    \sfI+ c\,\sfI\sfI,
\end{align}
where 
$$
    \sfI
    :=
    2^{-m_o sp'} \Tail \big(u-(u)_{x_o;2^{m_o} 8R};B_{2^{m_o} 8R(x_o)}\big),
$$
and
$$
    \sfI\sfI
    :=
    \sum_{k=1}^{m_o} 2^{-ksp'} 
    \bigg[ \bint_{B_{2^k 8R}(x_o)} \big|u-(u)_{x_o;2^k 8R}\big|^p \, \d x \bigg]^\frac{1}{p}.
$$
In the case $1<p<2$ we have 
\begin{align}\label{eq:tail-rhoi-estimate-}
    \Tail \big(u-(u)_{x_o;8R}; B_{8R}(x_o)\big)^{p-1} 
    &\leq 
    \sfI^{p-1} + c\,\sfI\sfI\sfI^{p-1},
\end{align}
where 
\begin{align*}
    \sfI\sfI\sfI^{p-1}
    :=
    \sum_{k=1}^{m_o} 2^{-ksp} \bigg[ \bint_{B_{2^k8R}(x_o)} \big|u-(u)_{x_o,2^k8R}\big|^{p} \, \d x \bigg]^\frac{p-1}{p} .
\end{align*}
Let us now consider the common term $\sfI$. 
We apply Lemma~\ref{lem:tail} with $(r,R)$ replaced by $(2^{m_o}8R,\rho)$ (note that in the case $1<p<2$ we take both sides to the power $\frac{1}{p-1}$) and obtain
\begin{align*}
    \sfI
    &\le 
    c\bigg[1 + \frac{|x_o-y_o|}{r}\bigg]^{\frac{n}{p-1}+sp'} 
    \Big(\frac{8R}{\rho}\Big)^{sp'} \Tail \big(u-(u)_{y_o,\rho}; B_\rho(y_o)\big) \\
    &\phantom{\le\,} + 
    c\, 2^{-m_o sp'}
    \Big(\frac{\rho}{2^{m_o}8R}\Big)^{\frac{n}{p-1}}
    \bigg[ \bint_{B_{\rho}(y_o)} \big|u-(u)_{y_o,\rho}\big|^{p} \, \d x \bigg]^\frac{1}{p} \\
    &\le 
    c \Big(\frac{R}{\rho}\Big)^{sp'} \Bigg[
    \Tail \big(u-(u)_{y_o,\rho}; B_\rho(y_o)\big) +
    \bigg[ \bint_{B_{\rho}(y_o)} \big|u-(u)_{y_o,\rho}\big|^{p} \, \d x \bigg]^\frac{1}{p} \Bigg].
\end{align*}
Here we also used that fact that $\frac{|x_o-y_o|}{r}\le \frac{\rho}{2^{m_o}8R}\le 4$ by the choice of $m_o$. 

For the second term on the right-hand side of~\eqref{eq:tail-rhoi-estimate}, 
we apply the fractional Poincaré inequality from Lemma~\ref{lem:poin}, 
with $\sigma$ in place of $s$. 
This yields in the case $p\ge 2$ that
\begin{align*}
    \sfI\sfI
    &\leq 
    c\, R^{\sigma- \frac{n}{p}} 
    \sum_{k=1}^{m_o} 2^{k(\sigma-sp' - \frac{n}{p})}  [u]_{W^{\sigma,p}(B_{2^k 8R}(x_o))}.
\end{align*}
and in the case $1<p<2$ that 
\begin{align*}
    \sfI\sfI\sfI^{p-1}
    \leq 
    c\, R^{(\sigma- \frac{n}{p})(p-1)} 
    \sum_{k=1}^{m_o} 2^{k(\sigma-sp' - \frac{n}{p})(p-1)}  [u]^{p-1}_{W^{\sigma,p}(B_{2^k 8R}(x_o))}.
\end{align*}
Inserting the estimates for $\sfI$ and $\sfI\sfI$, respectively $\sfI\sfI\sfI$ into~\eqref{eq:tail-rhoi-estimate}, respectively~\eqref{eq:tail-rhoi-estimate-}
and then substituting the resulting bound into the inequality for the second-order differences of $u$, 
we obtain the asserted estimate.
\end{proof}

The following theorem establishes quantitative higher differentiability for local weak solutions to the fractional $p$-Poisson equation~\eqref{eq:frac-p-lap}. 
Under suitable integrability assumptions on the right-hand side $f$, 
it shows that the solution itself belongs to $W^{1+\gamma,p}_{\mathrm{loc}}(\Omega)$ for some $\gamma>0$, 
providing an explicit local estimate for the gradient in terms of the data and the nonlocal tail of $u$.

\begin{theorem}[Higher differentiability]\label{thm:higher-diff}
Let $p>1$ and $s\in(0,1)$ with $sp'>1$, and let $\tilde\alpha\in[\alpha,\frac{1}{p-1}]$ be such that $\epsilon>0$, where $\alpha$ and $\epsilon$ are defined in~\eqref{def:alpha} and~\eqref{def:eps}. Assume in addition that
$f \in L^{\tilde{\alpha}p}_{\mathrm{loc}}(\Omega)$. Then any local weak solution $u$ to~\eqref{eq:frac-p}
belongs to
\begin{equation*}
u \in W^{1+\gamma,p}_{\mathrm{loc}}(\Omega)
\quad \text{for every } \gamma \in (0,\gamma_o),
\qquad
\gamma_o := \frac{\theta-p}{\theta-sp+\bar\epsilon p}\,\bar\epsilon,
\end{equation*}
where $\bar\epsilon:=\epsilon \min\{1,p-1\}$ and $\theta$ is given by~\eqref{def:theta}.
Moreover, for every $\gamma\in(0,\gamma_o)$ there exists a constant
$c = c(n,p,s,\tilde{\alpha},\gamma)>0$ such that for every ball
$B_R\Subset\Omega$ the quantitative estimate
\begin{align*}
    &\|\nabla u\|_{L^p(B_{\frac12 R})}
    + R^\gamma [\nabla u]_{W^{\gamma,p}(B_{\frac12 R})}\\
    &\qquad\qquad\le
    \frac{c}{R}
    \bigg[
        R^s [u]_{W^{s,p}(B_R)}
        + R^{\frac{n}{p}}
        \Tail\big(u-(u)_{R}; B_R\big)  
        + R^{1+\epsilon}
        \|f\|_{L^{\tilde{\alpha}p}(B_R)}^{\frac{1}{p-1}}
    \bigg]
\end{align*}
holds. Finally, in the limit $\epsilon\downarrow0$ one has
$\gamma\downarrow 0$ and $c\uparrow\infty$.
\end{theorem}

\begin{proof}
The proof is divided into several steps. 
We first establish quantitative estimates for second–order difference quotients.
These bounds are then used to prove almost $W^{1,p}$–differentiability.
Combining the resulting fractional differentiability with the estimates for the difference quotients yields differentiability beyond first order.
Finally, collecting the previous arguments, we conclude the desired estimate.

\smallskip
\noindent
\textit{Step~1: Estimate for second-order difference quotients.}
Let $B_\rho(y_o)\subset \Omega$ and \(\sigma \in [s,1)\), and assume that \(u \in W^{\sigma,p}(B_\rho(y_o))\).
Moreover, fix \(\beta \in (0,1)\) and let \(h \in \mathbb{R}^n\) satisfy
\[
    0<|h| \le h_o\rho, 
    \qquad\mbox{where }
    h_o
    := 
    \min\Big\{ \big(\tfrac{1}{16}\big)^{\frac{1}{\beta}},
\big(\tfrac{1}{7}\big)^{\frac{1}{1-\beta}} \Big\}.
\]
Applying Lemma~\ref{lem:covering} with
\((y_o,\tfrac12\rho,|h|^\beta\rho^{1-\beta})\) in place of \((x_o,R,r)\),
we infer the existence of a finite index set \(I\) and a family of points
\(\{x_i\}_{i\in I}\subset B_{\frac12 \rho}(y_o)\) such that
\begin{equation}\label{cov-1}
        B_{\frac12 \rho}(y_o)
    \subset
    \bigcup_{i\in I} B_{|h|^\beta\rho^{1-\beta}}(x_i),
   \qquad
    |I|
    \le
    c(n)\Big(\frac{\rho}{|h|}\Big)^{\beta n} ,
\end{equation}
and
\begin{equation}\label{cov-2}
     \sup_{x\in\mathbb{R}^n}
    \sum_{i\in I}
    \chi_{B_{2^k|h|^\beta\rho^{1-\beta}}(x_i)}(x)
    \le
    c(n)\,2^{nk}\qquad\forall\, k\in\N.
\end{equation}
In particular, \(\big\{B_i\equiv B_{|h|^\beta\rho^{1-\beta}}(x_i)\big\}_{i\in I}\) is a covering of
\(B_{\frac12 \rho}(y_o)\) with \(x_i\in B_{\frac12 \rho}(y_o)\) and satisfying
$16 B_i\subset B_\rho(y_o)$
for every $i\in I$.
For each \(i\in I\), let \[v_i\in W^{s,p}(8B_{i})\cap L_{sp}^{p-1}(\R^n)\] denote the unique weak solution to
\[
\left\{
\begin{array}{cl}
(-\Delta_p)^s v_i = 0,
& \text{in } 8 B_i, \\[6pt]
v_i = u,
& \text{a.e.\ in } \mathbb{R}^n \setminus 8 B_i.
\end{array}
\right.
\]
Choose \(m_o \in \mathbb{N}_0\) such that
\[
\tfrac{1}{32}\rho < 2^{m_o}|h|^\beta\rho^{1-\beta} \le \tfrac{1}{16}\rho.
\]
Then, by splitting we have
\begin{align}\label{Eq:tau-h-B-i}
    \int_{B_i} |\boldsymbol\tau_h^2 u|^p \,\d x
    &\le
    2^{p-1}\int_{B_i} |\boldsymbol\tau_h^2 (u - v_i)|^p \,\d x
    + 2^{p-1}\int_{B_i} |\boldsymbol\tau_h^2 v_i|^p \,\d x 
\end{align}
and we treat both contributions separately. The first one will be controlled in terms of 
\begin{equation}\label{def:Ii}
    \sfI_i
    :=
    \Big(\frac{|h|}{\rho}\Big)^{sp+(\sigma-s+\bar\epsilon)\beta p} 
    \Big[ \rho^{\sigma p} [u]_{W^{\sigma,p}(16B_i)}^p +
    \rho^{(1+\epsilon)p}
    \|f\|_{L^{\tilde\alpha p}(8B_i)}^{p'}\Big]. 
\end{equation}
In fact, the first integral of second order difference on the right-hand side is vacuously controlled by an integral of first order difference over a larger domain. Then, we estimate it by combining
Lemma~\ref{lem:N-FS} with the comparison estimate from
Lemma~\ref{lem:fract-level-comparison}.
In the case $p\ge 2$ we infer that  
\begin{align}\label{change-1}
    \int_{B_i} \big|\boldsymbol\tau_h^2 (u - v_i)\big|^p \,\d x
    &\le
    2^{p-1}\int_{\mathbb{R}^n} \big|\boldsymbol\tau_h (u - v_i)\big|^p \,\d x \nonumber\\
    &\le
    c\, |h|^{sp} [u-v_i]_{W^{s,p}(\mathbb{R}^n)}^p \nonumber\\
    &\le
    c\, |h|^{sp} \big(|h|^\beta\rho^{1-\beta}\big)^{(1-s+\epsilon)p}
    \|f\|_{L^{\tilde\alpha p}(8B_i)}^{p'} \nonumber\\
    &\le 
    c\, \Big(\frac{|h|}{\rho}\Big)^{sp+(\sigma-s+\bar\epsilon)\beta p} \rho^{(1+\epsilon)p}
    \|f\|_{L^{\tilde\alpha p}(8B_i)}^{p'} \nonumber\\
    &\le  
    c\,\sfI_i ,
\end{align}
since $\sigma<1$ and $\bar\varepsilon=\varepsilon$ in the present case.
In the same spirit, the case $p\in(1,2)$ yields that, for any $\delta\in(0,1]$,
\begin{align*}
    \int_{B_i} & \big|\boldsymbol\tau_h^2 (u - v_i)\big|^p \,\d x\\
    &\le
    2^{p-1}\int_{2B_i} \big|\boldsymbol\tau_h (u - v_i)\big|^p \,\d x \nonumber\\
    &\le
    c\, |h|^{sp} [u-v_i]_{W^{s,p}(16 B_i)}^p\nonumber\\
    &\le
    c\, |h|^{sp} \Big[\delta[u]_{W^{s,p}(16B_i)}^p +
    \delta^{-\frac{2-p}{p-1}} \big(|h|^\beta\rho^{1-\beta}\big)^{(1-s+\epsilon)p}
    \|f\|_{L^{\tilde\alpha p}(8B_i)}^{p'}\Big] \nonumber\\
    &\le
    c\, |h|^{sp} \Big[\delta\big(|h|^\beta\rho^{1-\beta}\big)^{(\sigma-s)p} [u]_{W^{\sigma,p}(16B_i)}^p\\
    &\qquad\qquad\qquad
    +
    \delta^{-\frac{2-p}{p-1}} \big(|h|^\beta\rho^{1-\beta}\big)^{(1-s+\epsilon)p}
    \|f\|_{L^{\tilde\alpha p}(8B_i)}^{p'}\Big] \nonumber\\
    &=
    c\, \Big(\frac{|h|}{\rho}\Big)^{sp} 
    \bigg[\delta\Big(\frac{|h|}{\rho}\Big)^{(\sigma-s)\beta p} \rho^{\sigma p} [u]_{W^{\sigma,p}(16B_i)}^p \\
    &\qquad\qquad\qquad +
    \delta^{-\frac{2-p}{p-1}} \Big(\frac{|h|}{\rho}\Big)^{(1-s+\epsilon)\beta p}\rho^{(1+\epsilon)p}
    \|f\|_{L^{\tilde\alpha p}(8B_i)}^{p'}\bigg] .
\end{align*}
Here, to estimate the $\delta$-term we also invoked \eqref{eq:fract-embedding-s-sigma}.
We now balance the powers of $\frac{|h|}{\rho}$ by choosing the parameter $\delta$ in the form
\[
    \delta := \Big(\frac{|h|}{\rho}\Big)^{(1-\sigma+\varepsilon)\beta p(p-1)} .
\]
With this choice, the two appearing powers of $\frac{|h|}{\rho}$ inside
the brackets are unified to be
$(\sigma-s)\beta p + (1-\sigma+\varepsilon)\beta p(p-1)
$. A direct computation shows that 
$$
    (\sigma-s)\beta p + (1-\sigma+\epsilon)\beta p(p-1)
    \ge 
    \big(\sigma-s+\epsilon(p-1)\big)\beta p
    =
    \big(\sigma-s+\bar\epsilon\big)\beta p,
$$
so that 
\begin{align}\label{change-1<}
    \int_{B_i} \big|\boldsymbol\tau_h^2 (u - v_i)\big|^p \,\d x 
    &\le 
    c\,\sfI_i.
\end{align}  
Hence, in either case the first term on the right-hand side of \eqref{Eq:tau-h-B-i} is controlled all the same in terms of $\sfI_i$.

We next turn to the estimation of the second term on the right-hand side of \eqref{Eq:tau-h-B-i}. Since \(u\in W^{\sigma,p}(B_\rho(y_o))\), Lemma~\ref{lem:vi} can be applied
on each ball \(B_i=B_{|h|^\beta\rho^{1-\beta}}(x_i)\) with \(R=|h|^\beta\rho^{1-\beta}\).
In particular, we have \(\frac{|h|}{R}=(\frac{|h|}{\rho})^{1-\beta}\) and
\(|h|\le \tfrac17 R = \tfrac17 |h|^\beta\rho^{1-\beta}\), the latter being guaranteed by
the choice of \(h_o\).
Hence, Lemma~\ref{lem:vi} yields
\begin{align*}
    \int_{B_i} \big|\boldsymbol
    \tau_h^2 v_i\big|^p \,\dx
    &\le
    c\, \big[ \sfI\sfI_i + \sfI\sfI\sfI_i + \sfI\sfV_i \big],
\end{align*}
where
\begin{align*}
    \sfI\sfI_i
    &:=
    \Big(\frac{|h|}{\rho}\Big)^{\theta(1-\beta)} \big(|h|^\beta\rho^{1-\beta}\big)^{\sigma p}
    \Bigg[
        \sum_{k=0}^{m_o}
        2^{k(\sigma-sp' - \frac{n}{p})}
        [u]_{W^{\sigma,p}(2^k 8B_i)}
    \Bigg]^p\\
    \sfI\sfI\sfI_i
    &:=
    \Big(\frac{|h|}{\rho}\Big)^{\theta(1-\beta)+spp'\beta} \big(|h|^{\beta}\rho^{1-\beta}\big)^n
    \sfE\big(u;B_\rho(y_o)\big)^p,
\end{align*}
and
\begin{align*}
     \sfI\sfV_i
    &:=
    \Big(\frac{|h|}{\rho}\Big)^{\theta(1-\beta)} \big(|h|^{\beta}\rho^{1-\beta}\big)^{(1+\epsilon)p}
    \|f\|_{L^{\tilde\alpha p}(8B_i)}^{p'}
\end{align*}
and $c=c(n,p,s,\tilde\alpha)$. 
Here, $\theta$ is defined in~\eqref{def:theta}.
We re-write
\begin{align*}
    \sfI\sfV_i
    &=
    \Big(\frac{|h|}{\rho}\Big)^{\theta(1-\beta)+(1+\epsilon)p\beta} 
    \rho^{(1+\epsilon)p}
    \|f\|_{L^{\tilde\alpha p}(8B_i)}^{p'}  .
\end{align*}
We now examine the exponent of $\frac{|h|}{\rho}$. In the case $p\ge2$ we have
$\theta = sp+2$, and therefore
\begin{align*}
    \theta(1-\beta)+(1+\epsilon)p\beta
    &=
    sp + (1-s+\epsilon)p\beta + 2(1-\beta)\ge 
    sp + (\sigma-s+\bar\epsilon)p\beta,
\end{align*}
where in the last step we used that $\sigma\le 1$ and $\bar\varepsilon =\varepsilon$.
In the case $p\in(1,2)$ we have $\theta = p+\tfrac12 sp$, and hence
\begin{align*}
    \theta(1-\beta)+(1+\epsilon)p\beta
    &=
    sp + (1-s+\epsilon)p\beta + (p-\tfrac12 sp)(1-\beta) \\
    &\ge 
    sp + (\sigma-s+\bar\epsilon)p\beta, 
\end{align*}
where we used that $p-\tfrac12 sp>0$, $\sigma\le1$, and $\bar\varepsilon\le\varepsilon$. 
This allows us to reduce the exponent of $\frac{|h|}{\rho}$ from 
$\theta(1-\beta)+(1+\varepsilon)p\beta$ to $sp + (\sigma-s+\bar\varepsilon)p\beta$, 
so that 
\[
    \sfI\sfV_i \le \sfI_i.
\] 
As a result, the sum $\sfI_i + \sfI\sfI_i + \sfI\sfI\sfI_i$
effectively controls the right-hand side of \eqref{Eq:tau-h-B-i}, 
which we now sum over the balls of the covering. Since $B_{\frac12\rho}(y_o) \subset \bigcup_{i\in I} B_i$, we obtain
\begin{align*}
    \int_{B_{\frac12 \rho}(y_o)} \big|\boldsymbol\tau_h^2 u\big|^p \,\dx
    &\le
    \sum_{i\in I}
    \int_{B_i} \big|\boldsymbol\tau_h^2 u\big|^p \,\dx 
    \le
    c \sum_{i\in I}
    \big[ \sfI_i+\sfI\sfI_i + \sfI\sfI\sfI_i  \big].
\end{align*}
We next estimate the sum of the terms \(\sfI\sfI_i\).
By Minkowski’s inequality~\eqref{eq:sum-minkowski}, we obtain
\begin{align*}
        \sum_{i\in I}\Bigg[
        \sum_{k=0}^{m_o}
        2^{k(\sigma-sp' - \frac{n}{p})}
        [u]_{W^{\sigma,p}(2^k 8B_i)}
    \Bigg]^p 
    &\le \Bigg[\sum_{k=0}^{m_o}\bigg[\sum_{i\in I}2^{pk(\sigma-sp' - \frac{n}{p})}
        [u]^p_{W^{\sigma,p}(2^k 8B_i)}\bigg]^{\frac1p}\Bigg]^p\\
        &= \Bigg[\sum_{k=0}^{m_o}2^{k(\sigma-sp' - \frac{n}{p})}\bigg[\sum_{i\in I}
        [u]^p_{W^{\sigma,p}(2^k 8B_i)}\bigg]^{\frac1p}\Bigg]^p\\
        &\le c\Bigg[\sum_{k=0}^{m_o}2^{k(\sigma-sp')} 
        [u]_{W^{\sigma,p}(B_\rho(y_o))} \Bigg]^p
\end{align*}
for some $c=c(n,p)$. To get the last line we also invoked \eqref{cov-2}, so that
\begin{align*}
\sum_{i\in I}[u]^p_{W^{\sigma,p}(2^k 8B_i)}
&= \sum_{i\in I}\iint_{B_\rho(y_o)\times B_\rho(y_o)}\frac{|u(x)-u(y)|^p}{|x-y|^{n+sp}}\chi_{2^k 8B_i}(x)\chi_{2^k 8B_i}(y)\d x \d y\\
&\le
\sum_{i\in I}\iint_{B_\rho(y_o)\times B_\rho(y_o)}\frac{|u(x)-u(y)|^p}{|x-y|^{n+sp}}\chi_{2^k 8B_i}(x)\d x \d y\\
&\le c(n)2^{nk}[u]^p_{W^{\sigma,p}(B_\rho(y_o))}.
\end{align*}
Therefore, we estimate
\begin{align*}
    \sum_{i\in I} \sfI\sfI_i
    &\le
    c \,\Big(\frac{|h|}{\rho}\Big)^{\theta(1-\beta)} \big(|h|^\beta\rho^{1-\beta}\big)^{\sigma p}
    \Bigg[
        \sum_{k=0}^{m_o}
        2^{k(\sigma-sp')}
        [u]_{W^{\sigma,p}(B_\rho(y_o))}
    \Bigg]^p \\
    &\le c\,
    \Big(\frac{|h|}{\rho}\Big)^{\theta(1-\beta)+\sigma p\beta} \rho^{\sigma p}
    \bigg( \frac{1}{1-2^{\sigma-sp'}} \bigg)^p
    [u]_{W^{\sigma,p}(B_\rho(y_o))}^p,
\end{align*}
since \(\sigma<1<sp'\). 
We next estimate the contribution of the terms $\sfI\sfI\sfI_i$.
Recalling \eqref{cov-1}$_2$, i.e., $|I|\le c(n)\,(\rho/|h|)^{\beta n}$, we infer that
\begin{align*}
    \sum_{i\in I}& \sfI\sfI\sfI_i\\
    &\le
    c\, |I| \Big(\frac{|h|}{\rho}\Big)^{\theta(1-\beta)+spp'\beta} \big(|h|^{\beta}\rho^{1-\beta}\big)^n
    \sfE\big(u;B_\rho(y_o)\big)^p \\
    &\le
    c \,\Big(\frac{|h|}{\rho}\Big)^{\theta(1-\beta)+spp'\beta} \rho^n
    \sfE\big(u;B_\rho(y_o)\big)^p \\
    &\le
    c \,\Big(\frac{|h|}{\rho}\Big)^{\theta(1-\beta)+\sigma p\beta} 
    \Big[\rho^{\sigma p}[u]_{W^{\sigma,p}(B_\rho(y_o))}^p
        + \rho^{n}\Tail\big(u-(u)_{y_o,\rho};B_\rho(y_o)\big)^p\Big],
\end{align*}
where we also used the fractional Poincar\'e inequality from Lemma~\ref{lem:poin} and the facts that $\sigma<sp'$ and $|h|/\rho\le 1$.
Taking into account the assumption
\(\tilde\alpha \le \frac{1}{p-1}\), which ensures that
\(\frac{1}{\tilde\alpha(p-1)} \ge 1\), and invoking again
the covering property~\eqref{cov-2}, we infer that
\begin{align*}
    \sum_{i\in I} \sfI_i
    &\le
    c\, \Big(\frac{|h|}{\rho}\Big)^{sp+(\sigma-s+\bar\epsilon)p\beta} 
    \sum_{i\in I}\Big[ \rho^{\sigma p}
    [u]_{W^{\sigma,p}(16B_i)}^p +
    \rho^{(1+\epsilon)p}
    \|f\|_{L^{\tilde\alpha p}(8B_i)}^{p'} \Big] \\
    &\le 
    c\, \Big(\frac{|h|}{\rho}\Big)^{sp+(\sigma-s+\bar\epsilon)p\beta}\\
    &\qquad\cdot
    \Bigg[
    \rho^{\sigma p}
    [u]_{W^{\sigma,p}(B_\rho(y_o))}^p
    +
    \rho^{(1+\epsilon)p}\bigg[
    \sum_{i\in I}\int_{B_\rho (y_o)}|f|^{\tilde\alpha p}\chi_{8B_i}\dx
    \bigg]^\frac1{\tilde\alpha (p-1)}
    \Bigg]\\
    &\le
    c\, \Big(\frac{|h|}{\rho}\Big)^{sp+(\sigma-s+\bar\epsilon)p\beta}
    \Big[
    \rho^{\sigma p}
    [u]_{W^{\sigma,p}(B_\rho(y_o))}^p
    +
    \rho^{(1+\epsilon)p}
    \|f\|_{L^{\tilde\alpha p}(B_{\rho (y_o))}}^{p'} \Big].
\end{align*}
The bound for the \(W^{\sigma,p}\)-seminorm on \(16B_i\) is a special case
of the estimate on \(2^k 8B_i\) obtained earlier, corresponding to the
choice \(k=1\). Hence, collecting the preceding estimates, we conclude that
\begin{align*}
    \int_{B_{\frac12 \rho}(y_o)} \big|\boldsymbol\tau_h^2 u\big|^p \,\dx
    &\le
    c \bigg[ \Big(\frac{|h|}{\rho}\Big)^{\theta(1-\beta)+\sigma p\beta} + \Big(\frac{|h|}{\rho}\Big)^{sp + (\sigma-s+\bar\varepsilon)\beta p} \bigg] \sfK^p,
\end{align*}
where $c=c(n,p,s,\tilde\alpha,\sigma)$ and 
\begin{equation*}
    \sfK^p
    :=
    \rho^{\sigma p}[u]_{W^{\sigma,p}(B_{\rho}(y_o))}^p + 
    \rho^{n}\Tail\big(u-(u)_{y_o,\rho};B_\rho(y_o)\big)^p + \rho^{(1+\epsilon) p}\|f\|_{L^{\tilde\alpha p}(B_\rho(y_o))}^{p'}.
\end{equation*}
We now choose
\[
    \beta 
    :=
    \frac{\theta-sp}{\theta-sp+\bar\epsilon p}.
\]
This choice is dictated by the requirement that the two exponents of the factor
\(\frac{|h|}{\rho}\) appearing on the right-hand side coincide. Indeed, a
straightforward computation shows that with this value of \(\beta\) one has
\[
    \theta(1-\beta)+\sigma p\beta
    =
    sp+(\sigma-s+\bar\epsilon)\beta p,
\]
and that their common value is given by
\[
    sp+(\sigma-s+\bar\epsilon)\beta p
    =
    \frac{\sigma(\theta-sp)+\bar\epsilon\theta}{\theta-sp+\bar\epsilon p}\, p.
\]
Consequently, the previous estimate simplifies to
\begin{align}\label{sstep1}
    \int_{B_{\frac12 \rho}(y_o)} \big|\boldsymbol\tau_h^2 u\big|^p \,\dx
    \le
    c\, \Big(\frac{|h|}{\rho}\Big)^{\frac{\sigma(\theta-sp)+\bar\epsilon\theta}{\theta-sp+\bar\epsilon p}\, p}
    \sfK^p,
\end{align}
where \(c=c(n,p,s,\tilde\alpha,\sigma)\).

\smallskip

\noindent
\textit{Step 2: Almost \(W^{1,p}\)-differentiability.} 
As in Step~1, let \(B_\rho(y_o)\subset\Omega\) and assume that
\(u\in W^{\sigma,p}(B_\rho(y_o))\) for some \(\sigma\in[s,1)\).
We first note that \(\theta>p\).
Indeed, by the definition in~\eqref{def:theta}, if \(p\in(1,2)\) then
\(\theta=p+\tfrac12 sp>p\), whereas if \(p\ge2\) we have
\(\theta=sp+2>(p-1)+2>p\).
Consequently, for the exponent of \(\frac{|h|}{\rho}\) appearing
in~\eqref{sstep1} we infer that
\[
    \frac{\sigma(\theta-sp)+\bar\epsilon\theta}{\theta-sp+\bar\epsilon p}\, p
    >
    \frac{\sigma(\theta-sp)+\bar\epsilon p}{\theta-sp+\bar\epsilon p}\, p .
\]
From~\eqref{sstep1} we therefore obtain
\begin{align*}
    &\int_{B_{\frac12\rho}(y_o)} \big|\boldsymbol\tau_h^2 \big(u-(u)_{y_o,\rho}\big)\big|^p \, \d x \leq 
    c^p\sfK^p  \Big(\frac{|h|}{\rho}\Big)^{\frac{\sigma(\theta-sp)+\bar\epsilon p}{\theta-sp+\bar\epsilon p} p},
\end{align*}
where we used that first-order difference quotients annihilate constants, i.e.,
\(
\boldsymbol\tau_h (u-(u)_{y_o,\rho})=\boldsymbol\tau_h u
\),
and hence the same holds for second-order differences. Since \(u\in W^{\sigma,p}(B_\rho(y_o))\), the right-hand side is finite.
At this stage, we first apply Lemma~\ref{lem:Domokos}\,(i) to reduce
second-order finite differences to first-order ones, and then invoke
Lemma~\ref{lem:FS-N} to translate the resulting estimate into a bound for
the fractional seminorm.
The precise argument is as follows. We apply Lemma~\ref{lem:Domokos}\,(i) with
$(w,q,\gamma,M,r,R,d)$ replaced by
\[
\bigg( u-(u)_{y_o,\rho}, \,p,\, 
\frac{\sigma(\theta-sp)+\bar\epsilon p}{\theta-sp+\bar\epsilon p}
,\,c\,\sfK\rho^{-\frac{\sigma(\theta-sp)+\bar\epsilon p}{\theta-sp+\bar\epsilon p}},\,\tfrac12\rho, \,\rho, \,h_o\rho\bigg).
\]
This yields
\begin{align*}
    \int_{B_{\frac12\rho}(y_o)} &\big|\boldsymbol\tau_h (u-(u)_{y_o,\rho})\big|^p \, \d x \\
    &\leq c\,|h|^{\frac{\sigma(\theta-sp)+\bar\epsilon p}{\theta-sp+\bar\epsilon p} p} \Bigg[
    \bigg(\frac{\sfK}{\rho^{\frac{\sigma(\theta-sp)+\bar\epsilon p}{\theta-sp+\bar\epsilon p}}}\bigg)^p +
    \frac{ \|u-(u)_{y_o,\rho}\|_{L^p(B_\rho(y_o))}^p}{(h_o\rho)^{\frac{\sigma(\theta-sp)+\bar\epsilon p}{\theta-sp+\bar\epsilon p}p}}
    \Bigg] \\
    &\leq 
    c\, \Big(\frac{|h|}{\rho}\Big)^{\frac{\sigma(\theta-sp)+\bar\epsilon p}{\theta-sp+\bar\epsilon p} p}
    \sfK^p ,
\end{align*}
for any $0<|h|\le \frac12 h_o\rho$, where the constant \(c\) depends only on
\(n,p,s,\sigma, \tilde \alpha\).  Since \(s>\tfrac{p-1}{p}\), the exponent of 
\(\frac{|h|}{\rho}\) admits the lower bound 
\[
    \frac{\sigma(\theta-sp)+\bar\epsilon p}{\theta-sp+\bar\epsilon p} \,p  
    =
    \bigg(1-(1-\sigma)\frac{\theta-sp}{\theta-sp+\bar\epsilon p}\bigg)p
    >
    \bigg(1-(1-\sigma)\frac{\theta-sp}{\theta-sp+\bar\epsilon}\bigg)p .
\]
This allows us to apply Lemma~\ref{lem:FS-N} with
\((w,q,M,\gamma,\beta,R,d)\) replaced by
\[
\bigg(u-(u)_{y_o,\rho},\,p,\,c\,\sfK\rho^{-\frac{\sigma(\theta-sp)+\bar\epsilon p}{\theta-sp+\bar\epsilon p}}, \,\frac{\sigma(\theta-sp)+\bar\epsilon p}{\theta-sp+\bar\epsilon p}, \,1-\frac{(1-\sigma)(\theta - sp)}{\theta-sp+\bar\epsilon}, \,\tfrac12\rho,\, \tfrac12 h_o\rho\bigg). 
\]
We thus conclude that $u\in W^{1-\frac{(1-\sigma)(\theta - sp)}{\theta-sp+\bar\epsilon},p}(B_{\frac12\rho}(y_o))$ and the following quantitative estimate holds:
\begin{align}\label{start-iteration}\nonumber
    [u]_{W^{1-\frac{(1-\sigma)(\theta - sp)}{\theta-sp+\bar\epsilon},p}(B_{\frac12\rho}(y_o))}^p & \\
    &\hspace{-80pt}\leq 
    c 
    \Bigg[ \frac{(h_o\rho)^{\left[\frac{\sigma(\theta-sp)+\bar\epsilon p}{\theta-sp+\bar\epsilon p}-(1-\frac{(1-\sigma)(\theta - sp)}{\theta-sp+\bar\epsilon})\right]p} }{\rho^{\frac{\sigma(\theta-sp)+\bar\epsilon p}{\theta-sp+\bar\epsilon p}p}}\sfK^p + 
    \frac{\|u-(u)_{y_o,\rho}\|^p_{L^p(B_{\frac12 \rho}(y_o))}}{(h_o\rho)^{(1-\frac{(1-\sigma)(\theta - sp)}{\theta-sp+\bar\epsilon})p}}
    \Bigg] \nonumber\\
    &\hspace{-80pt}\le 
    \frac{c}{\rho^{(1-\frac{(1-\sigma)(\theta - sp)}{\theta-sp+\bar\epsilon})p}}\, \sfK^p.
\end{align}
Moreover, the constant \(c\) depends only on \(n,p,s,\sigma , \tilde \alpha\). To obtain the last inequality, we used the definition of \(h_o\) together
with the definition of \(\sfE\), which allows us to absorb the
$\sfE$-term into \(\sfK^p\). The estimate shows that the fractional differentiability improves from
\(W^{\sigma,p}\) to \(W^{1-\frac{1-\sigma}{\theta-sp+\bar\epsilon p},p}\).
Since the inequality is stable under rescaling, this improvement can be
iterated to obtain higher fractional regularity.
For the concrete iteration of \eqref{start-iteration}, we define sequences
\[
    \sigma_o := s, \qquad
    \sigma_i := 1-\frac{(1-\sigma_{i-1})(\theta-sp)}{\theta-sp+\bar\epsilon}
    = 1-(1-s)\Big(\frac{\theta-sp}{\theta-sp+\bar\epsilon}\Big)^i, \qquad i\in\mathbb{N},
\]
and 
\[
    \rho_i:=\frac{\rho}{2^i}
    \qquad i\in\mathbb{N}_0.
\]
Then estimate~\eqref{start-iteration} applies with \(\sigma=\sigma_o=s\)
and yields that the \(W^{\sigma_1,p}\)-seminorm of \(u\) on
\(B_{\rho_1}(y_o)\) is finite.
This allows us to apply~\eqref{start-iteration} with
\((\rho,\sigma)\) replaced by \((\rho_1,\sigma_1)\), and to deduce
that the \(W^{\sigma_2,p}\)-seminorm of \(u\) on \(B_{\rho_2}(y_o)\) is
finite.
Using at each step the corresponding
rescaled version of \(\sfK\), we obtain
\begin{align*}
    [u]_{W^{\sigma_i,p}(B_{\rho_i}(y_o))}^p
    &\le
    \frac{c}{\rho_{i-1}^{\sigma_i p}}
    \Big[
        \rho_{i-1}^{\sigma_{i-1} p}
        [u]_{W^{\sigma_{i-1},p}(B_{\rho_{i-1}}(y_o))}^p \\
    &\qquad\qquad + 
        \rho_{i-1}^n
        \Tail\big(u-(u)_{y_o,\rho_{i-1}};B_{\rho_{i-1}}(y_o)\big)^p \\
    &\qquad\qquad + 
    \rho_{i-1}^{(1+\varepsilon)p}
        \|f\|_{L^{\alpha p}(B_{\rho_{i-1}}(y_o))}^{p'}
    \Big] \\
    &\le
    \frac{c}{\rho^{\sigma_i p}}
    \Big[
        \rho^{\sigma_{i-1} p}
        [u]_{W^{\sigma_{i-1},p}(B_{\rho_{i-1}}(y_o))}^p \\
    &\qquad\qquad + 
        \rho^n \Tail\big(u-(u)_{y_o,\rho_{i-1}};B_{\rho_{i-1}}(y_o)\big)^p \\
    &\qquad\qquad + 
        \rho^{(1+\varepsilon)p}
        \|f\|_{L^{\alpha p}(B_{\rho}(y_o))}^{p'}
    \Big].
\end{align*}
Here, the constant \(c\) depends on \(n,p,s\) and \(i\). 
For the tail contribution we invoke
Lemma~\ref{lem:tail}, which yields
\begin{align}\label{est-1}
    \Tail&\big(u - (u)_{y_o,\rho_{i-1}};B_{\rho_{i-1}}(y_o)\big) \nonumber\\
    &\leq 
    c\, \frac{\Tail \big(u-(u)_{y_o,\rho};B_{\rho}(y_o)\big)}{2^{(i-1)sp'}}
    +
    c\, 2^{\frac{n}{p}(i-1)} \bigg[ \bint_{B_{\rho}(y_o)} |u-(u)_{y_o,\rho}|^p \, \d x \bigg]^\frac{1}{p} \nonumber\\
    &\leq 
    c\, 2^{-(i-1)sp'}
    \Tail \big(u-(u)_{y_o,\rho};B_{\rho}(y_o)\big) +
    c\, 2^{\frac{n}{p}(i-1)}\rho^{s-\frac{n}{p}} [u]_{W^{s,p}(B_{\rho}(y_o))} ,
\end{align}
where $c = c(s,n,p)$. 
Inserting this estimate into the previous inequality, we obtain
\begin{align*}
      [u]_{W^{\sigma_i,p}(B_{\rho_i}(y_o))}^p
    &\le
    \frac{c}{\rho^{\sigma_i p}} 
    \Big[
        \rho^{\sigma_{i-1}p}[u]_{W^{\sigma_{i-1},p}(B_{\rho_{i-1}(y_o)})}^p +
        \rho^{sp} [u]_{W^{s,p}(B_{\rho}(y_o))}^p \\
        &\phantom{\le\, \frac{c}{\rho^{\sigma_i p}} 
    \Big[}
        + \rho^n\Tail\big(u-(u)_{y_o,\rho};B_\rho(y_o)\big)^p
        + \rho^{(1+\epsilon)p}\|f\|_{L^{\tilde \alpha p}(B_{\rho}(y_o))}^{p'}
    \Big],
\end{align*}
where $c=c(n,p,s,i)$. 
Iterating this inequality for \(i=1,\dots,j\), we conclude that for every
\(j\in\mathbb{N}\)
\begin{align}\label{est-2}
    [u]_{W^{\sigma_j,p}(B_{\rho_j}(y_o))}^p &\leq 
    \frac{c}{\rho^{\sigma_jp}} 
    \Big[ \rho^{sp}[u]_{W^{s,p}(B_{\rho}(y_o))}^p + 
    \rho^n\Tail\big(u-(u)_{y_o,\rho};B_\rho(y_o)\big)^p \nonumber\\
    &\phantom{\le\,\frac{c}{\rho^{\sigma_jp}} 
    \Big[ } + \rho^{(1+\epsilon)p}\|f\|_{L^{\tilde\alpha p}(B_\rho(y_o))}^{p'}\Big],
\end{align}
where the constant \(c\) depends only on \(n,p,s\) and \(j\). 
Since \(\sigma_j \uparrow 1\) as \(j\to\infty\) and $B_\rho(y_o)\subset\Omega$ was arbitrary, a standard covering argument yields
\[
    u \in W^{\sigma,p}_{\mathrm{loc}}(\Omega)
    \quad \text{for every } \sigma\in(0,1).
\]

\smallskip
\noindent
\textit{Step 3: Differentiability beyond first order.}
Let $\tilde{\alpha}$ be as in the statement of the theorem, that is,
$\tilde{\alpha}\in [\alpha,\frac{1}{p-1}]$, with $\alpha$ given by~\eqref{def:alpha}. For $j \in \N$ to be chosen later, we apply inequality \eqref{sstep1} with $(\rho,\sigma)$ replaced by $(\rho_j,\sigma_j)$  and obtain 
\begin{align*}
    \int_{B_{\rho_{j+1}}(y_o)} & |\boldsymbol{\tau}_h^2 u|^p \,\dx \\
    &\leq
    c \Big(\frac{|h|}{\rho}\Big)^{\theta_j p}
    \Big[
        \rho^{\sigma_j p}
        [u]_{W^{\sigma_j,p}(B_{\rho_j}(y_o))}^p
        + \rho^{n}\,
        \Tail\big(u-(u)_{y_o,\rho_j};B_{\rho_j}(y_o)\big)^p \\
    &\qquad\qquad\qquad
        + \rho^{(1+\varepsilon)p}
        \|f\|_{L^{\tilde \alpha p}(B_{\rho_j}(y_o))}^{p'}
    \Big],
\end{align*}
for any $h\in\R^n$ with $0<|h|<h_o\rho_j$, 
where $c=c(n,p,s,j)$ and 
\[
\theta_j
:=
\frac{\sigma_j(\theta-sp)+\bar\epsilon\theta}{\theta-sp+\bar\epsilon p}.
\]
Using \eqref{est-1} and \eqref{est-2} to control the terms on the right-hand
side, we further estimate
\begin{align}\label{second-diff}
    \int_{B_{\rho_{j+1}}(y_o)} |\boldsymbol{\tau}_h^2 u|^p \,\dx
    \le
    c^p \Big(\frac{|h|}{\rho}\Big)^{\theta_j p}
    \widetilde{\sfK}^p,
\end{align}
where
\begin{align*}
    \widetilde{\sfK}^p
    :=
    \rho^{s p}\,[u]_{W^{s,p}(B_{\rho}(y_o))}^p
    + \rho^{n}\,\Tail\big(u-(u)_{y_o,\rho};B_{\rho}(y_o)\big)^p
    + \rho^{(1+\varepsilon)p}
    \|f\|_{L^{\tilde\alpha p}(B_{\rho}(y_o))}^{p'} .
\end{align*}
In order for the above estimate to be effective, it is crucial that the
exponents \(\theta_j\) exceed \(1\) for large \(j\).
To this end, we compute
\begin{align*}
    \lim_{j\to\infty}\theta_j
    &=
    \frac{\theta-sp+\bar\epsilon\theta}{\theta-sp+\bar\epsilon p} \quad 
    =
    1+\frac{\theta-p}{\theta-sp+\bar\epsilon p}\,\bar\epsilon =
    1+\gamma_o
    >
    1,
\end{align*}
since \(\theta>p\). In particular, this yields a genuine gain beyond first-order
differentiability. Indeed, for any \(\gamma\in(0,\gamma_o)\) there exists \(j_o\in\mathbb{N}\)
such that \(\theta_{j_o}>1+\gamma\). Observe that $j_o$ depends only on $n,p,s, \tilde \alpha,\gamma$.  Next, we apply Lemma~\ref{lem:Domokos}, i.e.~\eqref{est-1st-diffquot>1}, with $(w,q,M,\gamma,r,R,d)$ replaced by 
\[
\Big(u-(u)_{y_o,\rho_{j_o}}, p, c\,\rho^{-\theta_{j_o}}\widetilde\sfK, \theta_{j_o}, \rho_{j_o+1}, \rho_{j_o},  h_o\rho_{j_o}\Big).
\]
Then, for any $h\in\R^n$ with $0<|h|\le \frac12 h_o\rho_{j_o}$, we obtain
\begin{align*}
    \int_{B_{\rho_{j_o+1}}(y_o)} |\boldsymbol{\tau}_h u|^p \,\dx
    &\le
    c\, |h|^p \Bigg[\frac{\widetilde\sfK^p}{\rho^{\theta_{j_o} p}} (h_o\rho_{j_o})^{(\theta_{j_o}-1)p} + \frac{\|u-(u)_{y_o,\rho_{j_o}}\|^p_{L^p(B_{\rho_{j_o}}(y_o))}}{(h_o\rho_{j_o})^p}  \Bigg] \\ 
    &\qquad\le
    c\, \frac{|h|^p}{\rho^p} \widetilde\sfK^p .
\end{align*}
To obtain the last line, we applied the fractional Poincar\'e inequality from Lemma~\ref{lem:poin}.
The constant $c$ may depend on $j_o$.
By a standard property of difference quotients, we may therefore conclude
that the weak gradient \(\nabla u\) exists and satisfies
\(\nabla u \in L^p(B_{\rho_{j_o+1}}(y_o);\mathbb{R}^n)\).
Moreover, the following quantitative estimate holds:
\begin{align}\label{grad-est}
    \|\nabla u\|_{L^p(B_{\rho_{j_o+1}}(y_o))}^p
    \le
    \frac{c}{\rho^p}\,\widetilde{\sfK}^p .
\end{align}
Having established the existence of the weak gradient, we may apply
inequality~\eqref{second-diff} to conclude fractional differentiability
of \(\nabla u\). More precisely, Lemma~\ref{lem:2nd-Ni-FS} applied with $(w,q,M,\gamma,\beta,R,d)$ replaced by 
\[
    \Big(u, p, c\,\rho^{-\theta_{j_o} p}\widetilde\sfK, \theta_{j_o}-1, \gamma, \rho_{j_o+2}, \tfrac12 h_o\rho_{j_o}\Big)
\]
ensures that $u\in W^{1+ \gamma, p}(B_{\rho_{j_o+2}}(y_o))$ and that
\begin{align}\label{grad-diff-est}\nonumber
    &[\nabla u]_{W^{1+ \gamma, p}(B_{\rho_{j_o+2}}(y_o))}^p \\
    &\qquad \leq 
    c\, \big(h_o\rho_{j_o}\big)^{(\theta_{j_o}-1-\gamma)p} \bigg[ \frac{\widetilde\sfK^p}{\rho_{j_o}^{\theta_{j_o} p}} + \frac{\rho^p}{(h_o\rho_{j_o})^{\theta_{j_o} p}} \int_{B_{\rho_{j_o+1}}(y_o)} |\nabla u|^p \,\dx \bigg] \nonumber\\
    &\qquad \leq 
    \frac{c}{\rho^{(1+\gamma)p}} \widetilde\sfK^p .
\end{align}
Here, in the last inequality we used~\eqref{grad-est}, and the constant
\(c\) depends only on \(n,p,s\) and \(j_o\).
As the choice of \(B_\rho(y_o)\subset\Omega\) was arbitrary, we conclude
that \(u\in W^{1+\gamma,p}_{\mathrm{loc}}(\Omega)\).
\smallskip

\noindent
\textit{Step 4: Covering argument and final estimate.}
Let \(\rho=\tfrac12 R\).
We choose a finite covering of \(B_{\frac12 R}\) by balls
\(\{B_{\rho_{j_o+2}}(y_\ell)\}_{\ell=1}^M\) such that the enlarged balls
\(B_\rho(y_\ell)\) have uniformly bounded overlap, namely each point of
\(\mathbb{R}^n\) belongs to at most \(c(n,j_o)\) of them, and
\(B_\rho(y_\ell)\subset B_R\) for every \(\ell=1,\dots,M\).
We then apply inequalities~\eqref{grad-est} and~\eqref{grad-diff-est} on
each ball \(B_\rho(y_\ell)\) and sum over \(\ell=1,\dots,M\). In this way, we obtain 
\begin{align*}
    \|\nabla u&\|_{L^p(B_{\frac12 R})}^p +
    R^{\gamma p} [\nabla u]_{W^{1+ \gamma, p}(B_{\frac12 R})}^p \\
    &\leq 
    \sum_{\ell=1}^M 
    \Big[ \|\nabla u\|^p_{L^p(B_{\rho_{j_o+1}}(y_\ell))} + 
    R^{\gamma p} [\nabla u]_{W^{1+ \gamma, p}(B_{\rho_{j_o+2}}(y_\ell))}^p \Big]\\
    &\leq 
    \frac{c}{R^{p}} 
    \sum_{\ell=1}^M
    \bigg[ \rho^{sp} [u]^p_{W^{s, p}(B_{\rho}(y_\ell))} + 
    \rho^n \Tail\big(u-(u)_{y_\ell,\rho}; B_{\rho}(y_\ell)\big)^p\\
    &\qquad\qquad\qquad+ 
    \rho^{(1+\epsilon)p}\|f\|_{L^{\tilde \alpha p}(B_{\rho}(y_\ell))}^{p'} \bigg] \\
    &\leq 
    \frac{c}{R^{p}} 
    \bigg[ R^{sp} [u]^p_{W^{s, p}(B_{R})} + 
    R^n \Tail\big(u-(u)_R; B_{R}\big)^p + 
    R^{(1+\epsilon)p}\|f\|_{L^{\tilde \alpha p}(B_{R})}^{p'} \bigg].
\end{align*}
To obtain the last line we used Lemma~\ref{lem:tail} and the fractional Poincar\'e inequality from Lemma~\ref{lem:poin}.
\end{proof}

\section{Higher differentiability for variable coefficients}\label{S:4}
We study nonlocal operators $(-a\Delta_p)^s$ with variable coefficients 
$a\colon\mathbb{R}^n \times \mathbb{R}^n \to \mathbb{R}_+$ satisfying the ellipticity, boundedness, and continuity conditions in \eqref{eq:a-condition}.

\subsection{Comparison estimate}
In order to compare the local weak solution $u$ of the nonhomogeneous problem~\eqref{eq:frac-p-lap} with a simpler reference, 
we first freeze the coefficient $a(x,y)$ inside the ball $B_R(x_o)$ and consider the associated homogeneous Dirichlet problem 
\eqref{CD-homo-a}. The function $v$ is then the unique weak solution to this frozen-coefficient problem. 

The precise setup is as follows. Let $B_R(x_o)\subset\Omega$, and define the frozen coefficient
\begin{equation}\label{def:frozen} 
a_{x_o,R}(x,y)=
\left\{
    \begin{array}{cl}
      (a)_{x_o,R}, & \mbox{if $(x,y)\in B_R(x_o)\times B_R(x_o)$,} \\[6pt]
      a(x,y), & \mbox{if $(x,y)\notin B_R(x_o)\times B_R(x_o)$;}
    \end{array}
\right.
\end{equation}
Denote by $v \in W^{s,p}(B_R(x_o)) \cap L^{p-1}_{sp}(\R^n)$ the unique weak solution to the corresponding homogeneous Dirichlet problem
\begin{equation}\label{CD-homo-a} 
\left\{
    \begin{array}{cl}
      (- a_{x_o,R}\Delta_p)^s v  = 0, & \mbox{in $B_R(x_o)$,} \\[6pt]
      v = u, & \mbox{a.e.~in $\R^n \setminus B_R(x_o)$.}
    \end{array}
\right.
\end{equation}
and set $w := u-v$. The forthcoming Lemma~\ref{lem:compar-var-coeff} provides precise control of $w$, 
capturing the effect of the inhomogeneity $f$ as well as the frozen coefficients, 
for both the superquadratic ($p\ge 2$) and subquadratic ($p\in(1,2)$) cases.

\begin{lemma}[Comparison estimate with variable coefficients]\label{lem:compar-var-coeff}
Let $p>1$ and $s\in(0,1)$ with $sp'>1$, and let $\tilde\alpha\in[\alpha,\frac{1}{p-1}]$ be such that $\epsilon>0$, where $\alpha$ and $\epsilon$ are defined in~\eqref{def:alpha} and~\eqref{def:eps}.
Then there exists a constant 
$c=c(n,p,s,C_o,C_1,\tilde\alpha)>0$ with the following property. 
Whenever $u$ is a local weak solution to~\eqref{eq:frac-p-lap} in the sense of Definition~\ref{def:weak-sol} with coefficients $a$ satisfying~\eqref{eq:a-condition}, $f\in L^{\tilde\alpha p}(\Omega)$, $B_{2R}\equiv B_{2R}(x_o)\Subset\Omega$ with $R \leq R_o$ where $R_o$ is from~\eqref{eq:a-condition}, and $v\in W^{s,p}(B_R(x_o)) \cap L^{p-1}_{sp}(\R^n)$ denotes the unique weak solution to the homogeneous Dirichlet problem~\eqref{CD-homo-a} with frozen coefficients $a_{x_o,R}$ in the sense of Definition~\ref{def:weak-sol-D}, the function
\[
w := u-v
\]
satisfies the following estimates:

\medskip
\noindent
\emph{(i) Superquadratic case.} 
If $p\in[2,\infty)$, then
\begin{align*}
    R^{-sp}\|w\|_{L^p(B_R)}^p + [w]_{W^{s,p}(\R^n)}^p
    &\leq 
    c\,R^{\frac{\chi p}{p-1}}[u]^p_{W^{s,p}(B_R)} +
    c\, R^{(1-s+\epsilon)p} \|f\|_{L^{\tilde\alpha p}(B_R)}^{p'}.
\end{align*}

\medskip
\noindent
\emph{(ii) Subquadratic case.} 
If $p\in(1,2]$,  then for every $\delta\in(0,1]$ we have
\begin{align*}
    R^{-sp}\|w\|_{L^p(B_R)}^p &+ [w]^p_{W^{s,p}(B_{2R})}\\
    &\le 
    \bigg[\delta 
     +
     \frac{R^{\frac{\chi p}{p-1}}}{\delta^{\frac{2-p}{p-1}p'}} \bigg]
     [u]_{W^{s,p}(B_{2R})}^p 
     + 
     \frac{c}{\delta^{\frac{2-p}{p-1}p'}} 
     R^{(1-s+\epsilon)p}
    \|f\|_{L^{\tilde\alpha p}(B_R)}^{p'} .
\end{align*}
\end{lemma}

\begin{proof}
Testing the equations for $u$ and $v$ with $w := u-v$ and then taking the difference, we obtain
\begin{align*}
    \iint_{\R^n \times \R^n} &
    \frac{\big[a(x,y)\boldsymbol{(u(x) - u(y))}^{p-1} - a_{x_o,R}(x,y)\boldsymbol{(v(x) - v(y))}^{p-1}\big] }{|x-y|^{n+sp}}\\
    & \qquad\qquad\qquad\cdot[w(x) - w(y)]\, \d x \d y\\
    &= 
    \int_{B_R} f w \, \d x ,
\end{align*}
where for $a\in\R$ we write
$$
    \boldsymbol{a}^{p-1}
    :=
    |a|^{p-2} a
    \quad\mbox{if $a\not=0$}
    \qquad\mbox{and}\qquad 
    \boldsymbol{a}^{p-1}:=0
    \quad\mbox{if $a=0$}.
$$
Properly splitting the integrand on the left-hand side, we get
\begin{equation*}
    \sfI=\sfI\sfI+
    \sfI\sfI\sfI,
\end{equation*}
where
\begin{align*}
    \sfI&:=
    \iint_{\R^n \times \R^n}
    a_{x_o,R}(x,y)\frac{\sfD (x,y)}{|x-y|^{n+sp}} \, \d x \d y,
\end{align*}
\begin{align*}
    \sfI\sfI&:=\iint_{\R^n \times \R^n} [a_{x_o,R}(x,y)-a(x,y)]
    \frac{\boldsymbol{(u(x) - u(y))}^{p-1}  (w(x) - w(y))}{|x-y|^{n+sp}} \, \d x \d y,
\end{align*}
and
\begin{align*}
    \sfI\sfI\sfI
    &:=\int_{B_R} f w \, \d x.
\end{align*}
Here, we defined
\begin{equation*}
    \sfD (x,y)
    :=\big(\boldsymbol{(u(x)-u(y))}^{p-1}-
    \boldsymbol{(v(x)-v(y))}^{p-1}\big)
    (w(x)-w(y)).
\end{equation*}
We emphasize that $\sfD(x,y)\ge 0$ for a.e.~$(x,y)\in\R^n\times\R^n$,
which follows from the monotonicity of the map
$\boldsymbol\xi \mapsto\boldsymbol \xi^{p-1}$
and 
\begin{equation}
\label{wx-y}
    w(x) - w(y)
    =
    (u(x)-u(y)) - (v(x)-v(y)).
\end{equation} 
When $p\ge 2$, we have the pointwise inequality
\begin{equation*}
    \big(\boldsymbol{\eta}^{p-1}- \boldsymbol{\xi}^{p-1}\big)(\eta-\xi) \geq \tfrac{1}{c(p)} |\eta-\xi|^p
    \quad\mbox{for any $\xi,\eta\in \R$.}
\end{equation*}
Using \eqref{wx-y} and the ellipticity bound $a_{x_o,R} \ge C_o > 0$, we obtain a lower bound for the left-hand side:
\begin{equation*}
    \frac{C_o}{c(p)} [w]_{W^{s,p}(B_{2R})}^p
    \le \iint_{B_{2R}\times B_{2R} }
    a_{x_o,R}(x,y)\frac{\sfD (x,y)}{|x-y|^{n+sp}} \, \d x \d y
 \le\sfI\sfI+\sfI
    \sfI\sfI.
\end{equation*}
In the sub-quadratic case $p \in (1,2)$, a lower bound for $\sfI$ follows from ~\cite[Lemma~2.2]{BDLMS-1}; cf.~\cite[Lemma~2.1]{Acerbi-Fusco}, \cite[Lemma~2.2]{GiaquintaModica:1986}. 
More precisely, we have
\begin{align*}
    \sfD(x,y) 
    &\ge 
    \frac{1}{c(p)} \big(|u(x)-u(y)| + |v(x)-v(y)|\big)^{p-2} |w(x)-w(y)|^2.
\end{align*}
Applying Young's inequality with exponents $(\frac{2}{2-p},\frac{2}{p})$ and using the previous display, we get
\begin{align*}
    |w(x)-w(y)|^p
    &= 
    \delta^{\frac{2-p}{2}} \big(|u(x)-u(y)|+|v(x)-v(y)|\big)^{\frac{(2-p)p}{2}} \\
    &\quad \cdot 
    \delta^{-\frac{2-p}{2}} \big(|u(x)-u(y)|+|v(x)-v(y)|\big)^{\frac{(p-2)p}{2}} |w(x)-w(y)|^p \\
    &\le 
    \delta \big(|u(x)-u(y)|+|v(x)-v(y)|\big)^p \\
    &\quad + 
    \delta^{-\frac{2-p}{p}} \big(|u(x)-u(y)|+|v(x)-v(y)|\big)^{p-2} |w(x)-w(y)|^2 \\
    &\le 
    8 \delta |u(x)-u(y)|^p + 2 \delta |w(x)-w(y)|^p 
    + \frac{c(p)}{\delta^{\frac{2-p}{p}}} \, \sfD(x,y).
\end{align*}
Choosing $\delta\in\bigl(0,\tfrac14\bigr]$ in the preceding inequality,
the second term on the right-hand side can be absorbed into the left-hand side.
As a consequence, we obtain
\begin{align*}
    |w(x)-w(y)|^p
    \le 
    16\,\delta\,|u(x)-u(y)|^p
    + \frac{c(p)}{\delta^{\frac{2-p}{p}}}\,\sfD (x,y).
\end{align*}
Multiplying both sides by $a_{x_o,R}(x,y)\,|x-y|^{-n-sp}$ and integrating over
$B_{2R}\times B_{2R}$, we infer, using the ellipticity bounds
$C_o\le a_{x_o,R}\le C_1$ and the nonnegativity of the integrand in $\sfI$, that
\begin{align*}
    \frac{C_o}{c(p)}&[w]_{W^{s,p}(B_{2R})}^p \\
    &\leq 
    \frac{16C_1}{c(p)}\delta [u]_{W^{s,p}(B_{2R})}^p +
    \frac{1}{\delta^{\frac{2-p}{p}}} \iint_{B_{2R}\times B_{2R} }
    a_{x_o,R}(x,y)\frac{\sfD (x,y)}{|x-y|^{n+sp}} \, \d x \d y \\
    &\leq \frac{16C_1}{c(p)}\delta [u]_{W^{s,p}(B_{2R})}^p +
    \frac{1}{\delta^{\frac{2-p}{p}}}\sfI\\
    &=
    \frac{16C_1}{c(p)}\delta [u]_{W^{s,p}(B_{2R})}^p +
    \frac{1}{\delta^{\frac{2-p}{p}}} 
    \big[\sfI \sfI+\sfI\sfI\sfI\big].
\end{align*}
The term $\sfI\sfI$ on the right-hand side can be estimated by exploiting
the Hölder continuity of the coefficient, Hölder's inequality, and Young's
inequality. Indeed, we have
\begin{align*}
    \frac{\sfI\sfI}{\delta^{\frac{2-p}{p}}}
    &=
    \frac{1}{\delta^{\frac{2-p}{p}}}\iint_{B_R\times B_R}
    \big[(a)_{x_o,R}-a(x,y)\big]
    \frac{\boldsymbol{(u(x)-u(y))}^{p-1}\,(w(x)-w(y))}{|x-y|^{n+sp}}
    \,\d x \d y \\
    &\le
    \frac{c\,R^{\chi}}{\delta^\frac{2-p}p}
    [u]_{W^{s,p}(B_R)}^{p-1}
    [w]_{W^{s,p}(B_R)} \\
    &\le
    \frac{C_o}{2c(p)}[w]_{W^{s,p}(B_R)}^p
    +
    \frac{c\, R^{\frac{\chi p}{p-1}}}{\delta^\frac{2-p}{p-1}}
    [u]_{W^{s,p}(B_R)}^p\\
    &=
    \frac{C_o}{2c(p)}[w]_{W^{s,p}(B_R)}^p
    +
    \frac{c\, R^{\frac{\chi p}{p-1}}}{\delta^\frac{1}{p-1}}\delta
    [u]_{W^{s,p}(B_R)}^p.
\end{align*}
Inserting this estimate into the previous inequality and absorbing the first
term on the right-hand side into the left-hand side, we arrive at
\begin{align*}
     [w]_{W^{s,p}(B_{2R})}^p
     &\le
     c\,\delta \bigg[1
     +
     \Big(\frac{R^{\chi p}}{\delta}\Big)^\frac1{p-1}\bigg]
     [u]_{W^{s,p}(B_{2R})}^p
     +c\, \frac{\sfI\sfI\sfI}{\delta^\frac{2-p}{p-1}}.
\end{align*}
It remains to estimate the term involving the inhomogeneity $f$.
By Lemma~\ref{lem:fw}, we have for every $\tilde\delta\in(0,1)$ that 
\begin{align*}
    \sfI\sfI\sfI
    \le
    \tilde\delta\,[w]_{W^{s,p}(B_{2R})}^p
    +
    c\tilde\delta^{-\frac{1}{p-1}} R^{(1-s+\epsilon)p}\,
    \|f\|_{L^{\tilde\alpha p}(B_R)}^{p'} .
\end{align*}
Choosing $\tilde\delta=\frac{\delta^\frac{2-p}{p-1}}{2c}$, we can re-absorb $\frac12[w]_{W^{s,p}(B_{2R})}^p$ into the left-hand side with the result
\begin{align*}
     [w]_{W^{s,p}(B_{2R})}^p
     \le
     c\,\delta \bigg[1
     +
     \Big(\frac{R^{\chi p}}{\delta}\Big)^\frac1{p-1}\bigg]
     [u]_{W^{s,p}(B_{2R})}^p
     + \frac{c}{\delta^{\frac{2-p}{p-1}p'}} 
     R^{(1-s+\epsilon)p}\,
    \|f\|_{L^{\tilde\alpha p}(B_R)}^{p'}.
\end{align*}
Substituting $c\,\delta$ by $\delta$ we obtain the comparison estimate for the $W^{s,p}$-norm of $w$ in the remaining case $p\in(1,2)$. The inequality for its $L^p$-norm now follows from Sobolev's inequality, as in~\cite[Corollary~3.3]{BDLM}.
\end{proof}

\subsection{Higher differentiability of the gradient}
The following lemma provides higher differentiability and gradient estimates for solutions
to the fractional $p$-Laplace equation with frozen coefficients, and parallels Lemma~\ref{lem:2nd-diffquot-homo}.

\begin{lemma}[Higher differentiability for frozen-coefficient solutions]\label{lem:vi-a}
Let $p>1$, $s\in(0,1)$ with $sp'>1$, and let $B_R\equiv B_R(x_o)\subset\R^n$. Whenever $v\in W^{s,p}_{\rm loc}(B_R)\cap L^{p-1}_{sp}(\R^n)$ is a local weak solution to 
\[
(- a_{x_o,R}\Delta_p)^s v = 0 \quad \text{in } B_R,
\]
then 
\begin{equation*}
    v \in W^{1+\gamma,p}_{\rm loc}(B_R) 
    \quad \text{for every } \gamma \in (0,\gamma_o),
    \qquad 
    \gamma_o := \frac{\theta - p}{\theta - sp + \bar\epsilon p}\,\bar\epsilon,
\end{equation*}
where $\bar\epsilon := \min\{1,p-1\}\,(sp'-1)$ and $\theta$ is defined in~\eqref{def:theta}. 
Moreover, there exists a constant 
$c=c(n,p,s,C_o,C_1,\gamma)>0$ such that for every such $v$ one has
\begin{align*}
    \|\nabla v\|_{L^p(B_{\frac14R})}
    &+   R^\gamma[\nabla v]_{W^{\gamma,p}(B_{\frac14R})} \\
    &\leq
    \frac{c}{R} \Big[ R^{s}[v]_{W^{s,p}(B_{\frac12R})} +
    R^\frac{n}{p} \Tail\big(v-(v)_{\frac12R};B_{\frac12 R}\big)\Big],
\end{align*}
and, for every $h\in\R^n\setminus\{0\}$ with $|h|\le \frac1{32}R$,
\begin{align*}
    \int_{B_{\frac14 R}} |\boldsymbol{\tau}_h^2 v|^p \,\dx
    &\le 
    c \Big(\frac{|h|}{R}\Big)^{(1+\gamma)p} 
    \Big[ R^{sp} [v]_{W^{s,p}(B_{\frac12 R})}^p + R^n \Tail(v-(v)_{\frac12 R}; B_{\frac12 R})^p \Big].
\end{align*}
Moreover, for every $q\in(1,\infty)$, we have $v \in W^{1,pq}_{\rm loc}(B_R)$, and there exists a constant 
$c=c(n,p,s,C_o,C_1,q)>0$ such that
\begin{align*}
    \|\nabla v\|_{L^{pq}(B_{\frac14 R})} 
    &\le 
    c \, R^{-\frac{n}{p}\left(1-\frac{1}{q}\right)-1} 
    \Big[ R^s [v]_{W^{s,p}(B_R)} + 
    R^{\frac{n}{p}} \Tail\big(v-(v)_{\frac12 R}; B_{\frac12 R}\big) \Big].
\end{align*}
\end{lemma}

\begin{proof}
Throughout the proof we omit the reference to the center $x_o$. 
Let $\varphi \in W^{s,p}(\R^n)$ with $\operatorname{spt}(\varphi) \subset B_{\frac12 R}$. 
By the weak formulation of $v$ and a straightforward algebraic manipulation, we have
\begin{align*}
    \iint_{\R^n\times\R^n}\frac{\boldsymbol{(v(x)-v(y))}^{p-1}(\varphi(x)-\varphi(y))}{|x-y|^{n+sp}} \,\dx\dy
    &=\sfI,
\end{align*}
where
\begin{align*}
    \sfI&:=
    \iint_{\R^n\times\R^n}\frac{(a)_{R}-a_{R}(x,y)}{(a)_{R}}\frac{\boldsymbol{(v(x)-v(y))}^{p-1}(\varphi(x)-\varphi(y))}{|x-y|^{n+sp}} \,\dx\dy
\end{align*}
Since $(a)_{R} = a_{x_o,R}(x,y)$ in $B_R \times B_R$ and $\operatorname{spt}(\varphi) \subset B_{R/2}$, the right-hand side reduces to
\[
    \sfI=
    2\iint_{B_{\frac12 R}\times(\R^n\setminus B_{R})} \frac{(a)_{R}-a_{R}(x,y)}{(a)_{R}}\frac{\boldsymbol{(v(x)-v(y))}^{p-1}\varphi(x)}{|x-y|^{n+sp}}\,\dx\dy.
\]
Hence, $v$ satisfies
\[
    (-\Delta_p)^s v = g \quad \text{in } B_{\frac12 R},
\]
where the inhomogeneity $g$ is given by
\[
    g(x) := 2 \int_{\R^n \setminus B_R} 
    \frac{(a)_R - a_{x_o,R}(x,y)}{(a)_R} \frac{
    |v(x)-v(y)|^{p-2} (v(x)-v(y))}{|x-y|^{n+sp}} \dy.
\]
Since $v \in L^p_{\rm loc}(B_R)\cap L^{p-1}_{sp}(\R^n)$ and $|x-y|\ge \frac12 R$ for $x\in B_{\frac12 R}$ and $y\in \R^n \setminus B_R$, we have $g \in L^{p'}(B_{\frac12 R})$. Hence, we may apply Theorem~\ref{thm:higher-diff} (see also~\cite[Theorem~1.2]{DKLN-higherdiff}) with $\tilde\alpha = \frac{1}{p-1}$ to conclude that
\[
    v \in W^{1+\gamma,p}_{\rm loc}(B_{\frac12 R}) \quad \text{for any } \gamma \in (0,\gamma_o),
\]
Moreover, 
Theorem~\ref{thm:higher-diff} yields the quantitative estimate
\begin{align}\label{est:v*}\nonumber
    \|\nabla &v\|_{L^p(B_{\frac14 R})}
    +  R^\gamma[\nabla v]_{W^{\gamma,p}(B_{\frac14R})}\\
    &\leq
    \frac{c}{R} \Big[
         R^s[v]_{W^{s,p}(B_{\frac12 R})}
        +  
        R^{\frac{n}{p}}\Tail\big(v-(v)_{\frac12 R}; B_{R/2}\big) +
        R^{sp'}\|g\|_{L^{p'}(B_{\frac12 R})}^{\frac{1}{p-1}}
    \Big],
\end{align}
where $c=c(n,p,s,\gamma)$. Note that for our choice $\tilde\alpha = \frac{1}{p-1}$, one has $
    \epsilon = sp' - 1$,
see~\eqref{def:eps}. 
In the preceding estimates it remains to estimate the $L^{p'}$-norm of $g$.
Using the assumptions~\eqref{eq:a-condition} on the coefficient $a$, we have 
\begin{align}\label{est-f}
    |g(x)|
    &\le 
    c\, \bigg|\int_{\R^n\setminus B_{R}} \frac{|v(x)-v(y)|^{p-1}}{|x-y|^{n+sp}}\,\dy \bigg| \nonumber\\
    &\le 
    c\, \bigg|\int_{\R^n\setminus B_{R}} \frac{|v(x)-(v)_{\frac12 R}|^{p-1}}{|x-y|^{n+sp}}\,\dy \bigg| +
    c\, \bigg|\int_{\R^n\setminus B_{R}} \frac{|v(y)-(v)_{\frac12 R}|^{p-1}}{|x-y|^{n+sp}}\,\dy \bigg|  \nonumber\\
    &\le 
    \frac{c}{R^{sp}} |v(x)-(v)_{\frac12 R}|^{p-1} +
    c\, \bigg|\int_{\R^n\setminus B_{\frac12 R}} \frac{|v(y)-(v)_{\frac12 R}|^{p-1}}{|y-x_o|^{n+sp}}\,\dy \bigg|  \nonumber\\
    &\le 
    \frac{c}{R^{sp}} \Big[ |v(x)-(v)_{\frac12 R}|^{p-1} +
    \Tail\big(v-(v)_{\frac12 R};B_{\frac12 R}\big)^{p-1}\Big]  ,
\end{align}
for a.e. $x\in B_{\frac12 R}$.
Taking both sides to the power $p'$, integrating over $B_{\frac12 R}$ and using Lemma~\ref{lem:poin} yields
\begin{align*}
    \|g\|_{L^{p'}(B_{\frac12 R})}^{p'}
    &\le 
    \frac{c}{R^{spp'}}\bigg[ \int_{B_{\frac12 R}} |v(x)-(v)_{\frac12 R}|^{p} \,\dx +
    c\,R^n \Tail\big(v-(v)_{\frac12 R};B_{\frac12 R}\big)^p \bigg]\\
    &\le 
    \frac{c}{R^{spp'}} \Big[ R^{sp}[v]_{W^{s,p}(B_{\frac12R})}^p +
    c\,R^n \Tail\big(v-(v)_{\frac12 R};B_{\frac12 R}\big)^p\Big] ,
\end{align*}
where  $c=c(n,s,p,C_o,C_1)$. 
Plugging this estimate for $\|g\|_{L^{p'}(B_{\frac12 R})}$ into \eqref{est:v*} gives the first asserted bound for $v$.
The second claimed estimate for the second-order finite differences of $v$ then follows directly from the first one by Lemma~\ref{lem:W-1+gamma}.

Previously, we only exploited that $g \in L^{p'}(B_{\frac12 R})$. However, in view of Theorem~\ref{thm:sup-est} we know that actually $g \in L^\infty(B_{\frac12 R})$. This improvement allows us to invoke~\cite[Theorem~1.1 and Remark~1.2]{BDLM}, yielding $\nabla v \in L^{pq}_{\mathrm{loc}}(B_{\frac12 R})$ for any $q>1$, together with the estimate
\begin{align*}
    \bigg[ \bint_{B_{\frac14R}} & |\nabla v|^{pq} \, \d x \bigg]^\frac{1}{q} \\
    &\leq 
    \frac{c}{R^{(1-s)p}}
    \mint_{B_{\frac12R}}\int_{B_{\frac12R}}
    \frac{|v(x)-v(y)|^p}{|x-y|^{n+sp}}\,\d x\d y +
    \frac{c}{R^p}
    \mathrm{Tail}\big(v-(v)_{\frac12R}; B_{\frac12R}\big)^p \\
    &\phantom{\le\,} +
    \frac{c}{R^{(1-s)p}}\bigg[\mint_{B_{\frac12R}}
    |R^sg|^{\frac{pq}{p-1}}\,\d x\bigg]^\frac1q,
\end{align*}
where $c=c(n,s,p,q)$. 
At this point, it remains to estimate the $L^{pq'}$-norm of $g$. 
Applying~\eqref{est-f}, Theorem~\ref{thm:sup-est}, and the fractional Poincar\'e inequality from Lemma~\ref{lem:poin}, we obtain
 \begin{align*}
    \bigg[\mint_{B_{\frac12R}} &
    |R^sg|^{\frac{pq}{p-1}}\,\d x\bigg]^\frac1q \\
    &\le 
    c\,R^{sp'} \|g\|_{L^{\infty}(B_{\frac12 R})}^{p'} \\
    &\le 
    c\, R^{sp'-spp'}\Big[\|v-(v)_{\frac12 R}\|_{L^{\infty}(B_{\frac12 R})}^p +
    \Tail\big(v-(v)_{\frac12 R};B_{\frac12 R}\big)^p \Big]\\
    &\le 
    c\, R^{-sp}\Big[ R^{-n}\|v-(v)_{\frac12 R}\|^p_{L^{p}(B_{R})} +
    \Tail\big(v-(v)_{\frac12 R};B_{\frac12 R}\big)^p \Big]\\
    &\le 
    c\, R^{-sp}\Big[R^{-n} \|v-(v)_{R}\|^p_{L^{p}(B_{R})} +
    \Tail\big(v-(v)_{\frac12 R};B_{\frac12 R}\big)^p \Big]\\
    &\le 
    c\, R^{-sp} \Big[ R^{sp-n}[v]_{W^{s,p}(B_{R})}^p +
    \Tail\big(v-(v)_{\frac12 R};B_{\frac12 R}\big)^p\Big] .
\end{align*}
Substituting the preceding estimate into the chain of inequalities above 
yields the third inequality and thus completes the proof of the lemma.
\end{proof}

The following statement is the analogue of Lemma~\ref{lem:vi} for the frozen-coefficient operator.

\begin{lemma}\label{lem:wi}
Let $p>1$ and $s\in(0,1)$ with $sp'>1$, let $0<C_o\le C_1$, and let $\tilde\alpha\in[\alpha,\frac{1}{p-1}]$ be such that $\epsilon>0$, where $\alpha$ and $\epsilon$ are defined in~\eqref{def:alpha} and~\eqref{def:eps}.
Let $\theta=\theta(s,p)$ be given by~\eqref{def:theta}
and let $\varepsilon$ be as in~\eqref{def:eps}.
Then there exists a constant
$c=c(n,p,s,C_o,C_1,\tilde\alpha)>0$ with the following property.
Assume that
$u\in W^{\sigma,p}(B_\rho(y_o))\cap L^{p-1}_{sp}(\mathbb{R}^n)$,
with $\sigma\in[s,1)$, is a weak solution to~\eqref{eq:frac-p-lap}
in $B_\rho(y_o)$ with coefficients $a$ satisfying~\eqref{eq:a-condition}.
Suppose that $\rho\le1$, $f\in L^{\tilde\alpha p}(B_\rho(y_o))$,
$x_o\in B_{\frac12 \rho}(y_o)$, and $R\in(0,\frac1{16}\rho]$.
Let
$v\in W^{s,p}(B_{8R}(x_o))\cap L^{p-1}_{sp}(\mathbb{R}^n)$
denote the unique weak solution to~\eqref{CD-homo-a}
in $B_{8R}(x_o)$ with frozen coefficients $a_{x_o,8R}$.
Then, for any $m_o\in\mathbb{N}_0$ such that
$
\frac{1}{32}\rho < 2^{m_o}R \le \frac{1}{16}\rho
$
and for any step size $h\in\mathbb{R}^n$ with $0<|h|\le \frac17R$, there holds
\begin{align*}
    \int_{B_R(x_o)} \big|\boldsymbol{\tau}_h^2 v\big|^p \, \d x
    &\le
    c \Big(\frac{|h|}{R}\Big)^{\theta}
    \Bigg[
        R^{\sigma p}
        \Bigg(
            \sum_{k=0}^{m_o}
            2^{k(\sigma-sp' - \frac{n}{p})}
            [u]_{W^{\sigma,p}(B_{2^k 8R}(x_o))}
        \Bigg)^p \\
    &\qquad\qquad\quad
        + \Big(\frac{R}{\rho}\Big)^{spp'} R^n
        \sfE \big(u;B_\rho(y_o)\big)^p \\
    &\qquad\qquad\quad
        + R^{(1+\varepsilon)p}
        \|f\|_{L^{\tilde\alpha p}(B_{8R}(x_o))}^{p'}
    \Bigg].
\end{align*}
\end{lemma}

\begin{proof}
The proof follows the same scheme as that of Lemma~\ref{lem:vi}.
The only difference is that Lemmas~\ref{lem:compar-var-coeff}
and~\ref{lem:vi-a} are employed in place of
Lemmas~\ref{lem:fract-level-comparison} and~\ref{lem:vi-1},
respectively. More precisely, one first establishes the analogue of Lemma~\ref{lem:vi-1}.
To this end, Lemma~\ref{lem:2nd-diffquot-homo} is replaced by
Lemma~\ref{lem:vi-a}, applied with $\gamma=\frac12\gamma_o$.
As a consequence, one may choose
\[
\theta = p\Bigl(1+\frac12\gamma_o\Bigr) > p,
\]
where the constant $\theta$ depends only on $s$ and $p$.
The resulting estimate is then combined with the comparison estimate
from Lemma~\ref{lem:compar-var-coeff}.
Since $\rho\le1$, the right-hand side in the latter comparison estimate
reduces to that of Lemma~\ref{lem:fract-level-comparison} with the choice
$\delta=1$.
This yields the desired analogue of Lemma~\ref{lem:vi-1}.
Once this preliminary result is available, the proof of
Lemma~\ref{lem:wi} proceeds verbatim as in the proof of
Lemma~\ref{lem:vi}. 
\end{proof}

\begin{theorem}[Higher differentiability]\label{thm:higher-diff-2}
Let $p>1$ and $s\in(0,1)$ with $sp'>1$, let $0<C_o\le C_1$, and let $\tilde\alpha\in[\alpha,\frac{1}{p-1}]$ be such that $\epsilon>0$, where $\alpha$ and $\epsilon$ are defined in~\eqref{def:alpha} and~\eqref{def:eps}.
Assume that $f\in L^{\tilde{\alpha}p}_{\mathrm{loc}}(\Omega)$ and that
$u$ is a local weak solution to~\eqref{eq:frac-p-lap} with coefficients
$a$ satisfying~\eqref{eq:a-condition}.
Then there exists
\(
\gamma=\gamma(p,s,\chi,\tilde\alpha)>0
\)
such that
\[
u\in W^{1+\gamma,p}_{\mathrm{loc}}(\Omega).
\]
Moreover, there exists a constant
$c=c(n,p,s,C_o,C_1,\chi,\tilde\alpha)>1$ such that for every ball
$B_R\Subset\Omega$ with $R\le R_o$,  where $R_o$ is from~\eqref{eq:a-condition}, the estimate
\begin{align*}
    &\|\nabla u\|_{L^p(B_{\frac12 R})}
    + R^\gamma [\nabla u]_{W^{\gamma,p}(B_{\frac12 R})} \\
    &\quad\le
    \frac{c}{R}
    \bigg[
        R^s [u]_{W^{s,p}(B_R)}
        + R^{\frac{n}{p}}
        \Tail\big(u-(u)_{R}; B_R\big)
        + R^{1+\varepsilon}
        \|f\|_{L^{\tilde{\alpha}p}(B_R)}^{\frac{1}{p-1}}
    \bigg]
\end{align*}
holds.
Finally, in the limit $\epsilon\downarrow0$ one has
$\gamma\downarrow0$ and $c\uparrow\infty$.
\end{theorem}

\begin{remark}\label{rem:go}
Let $\gamma_o$ be the parameter from Lemma~\ref{lem:vi-a}. 
Since the proof of Theorem~\ref{thm:higher-diff-2} relies on the a priori estimate from Lemma~\ref{lem:vi-a}, it is evident that $\gamma < \gamma_o$.
\end{remark}

\begin{proof}
The proof follows the same strategy as the one for the fractional Poisson equation
with constant coefficients $a\equiv1$, cf.\ Theorem~\ref{thm:higher-diff}.
We therefore only indicate the modifications that are required in the present setting.
Throughout the proof, Steps~1--4 refer to the corresponding steps in the proof of
Theorem~\ref{thm:higher-diff}.

\smallskip
\noindent
\textit{Step~1: Estimate for second-order difference quotients.}
We proceed as in the proof of Theorem~\ref{thm:higher-diff} up to
inequality~\eqref{Eq:tau-h-B-i}. 
It remains to estimate the first term on the right-hand side of
\eqref{Eq:tau-h-B-i}, namely
\[
     \int_{B_i} \big|\boldsymbol\tau_h^2 (u - v_i)\big|^p \,\d x,
\]
in terms of the quantity
\[
    \sfI_i
    :=
    \Big(\frac{|h|}{\rho}\Big)^{sp+(\sigma-s+\bar\epsilon)\beta p} 
    \Big[
        \rho^{\sigma p} [u]_{W^{\sigma,p}(16B_i)}^p
        +
        \rho^{(1+\epsilon)p}
        \|f\|_{L^{\tilde\alpha p}(8B_i)}^{p'}
    \Big],
\]
where
\begin{equation}\label{def:bareps}
    \bar\epsilon
    :=
    \min\{1,(p-1)^2\}
    \min\bigg\{\epsilon,\frac{\chi}{p-1}\bigg\}.
\end{equation}
Observe that the quantity $\sfI_i$ coincides with the one defined
in~\eqref{def:Ii}, up to the different definition of $\bar\epsilon$
in Theorem~\ref{thm:higher-diff} and in~\eqref{def:bareps}.
For this estimate, we proceed as in~\eqref{change-1}, applying
the comparison estimate for variable coefficients from
Lemma~\ref{lem:compar-var-coeff}.
In the case $p\ge2$ one has
\begin{align*}
    \int_{B_i} &
    \big|\boldsymbol\tau_h^2 (u - v_i)\big|^p \,\d x\\
    &\le
    c |h|^{sp} [u-v_i]_{W^{s,p}(\mathbb{R}^n)} \\
    &\le
    c |h|^{sp} \Big[
    \big(|h|^\beta\rho^{1-\beta}\big)^{\frac{\chi p}{p-1}}[u]^p_{W^{s,p}(8B_i)} 
    +
    \big(|h|^\beta\rho^{1-\beta}\big)^{(1-s+\epsilon)p}
    \|f\|_{L^{\tilde\alpha p}(8B_i)}^{p'} \Big] \\
    &\le
    c |h|^{sp} \Big[
    \big(|h|^\beta\rho^{1-\beta}\big)^{\frac{\chi p}{p-1}+(\sigma-s)p}[u]^p_{W^{\sigma,p}(8B_i)}  +
    \big(|h|^\beta\rho^{1-\beta}\big)^{(1-s+\epsilon)p}
    \|f\|_{L^{\tilde\alpha p}(8B_i)}^{p'} \Big] \\
    &=
    c \Big(\frac{|h|}{\rho}\Big)^{sp+(\sigma-s)\beta p} 
    \bigg[ 
    \Big(\frac{|h|}{\rho}\Big)^{\frac{\chi}{p-1}\beta p} \rho^{(\sigma+\frac{\chi}{p-1})p} [u]^p_{W^{\sigma,p}(8B_i)} \\
    &\qquad\qquad\qquad\qquad\quad +
    \Big(\frac{|h|}{\rho}\Big)^{(1-\sigma+\epsilon)\beta p} \rho^{(1+\epsilon)p}
    \|f\|_{L^{\tilde\alpha p}(8B_i)}^{p'} \bigg] \\
    &\le
    c \Big(\frac{|h|}{\rho}\Big)^{sp+(\sigma-s)\beta p} 
    \bigg[ 
    \Big(\frac{|h|}{\rho}\Big)^{\frac{\chi}{p-1}\beta p} \rho^{\sigma p} [u]^p_{W^{\sigma,p}(8B_i)} 
    +
    \Big(\frac{|h|}{\rho}\Big)^{\epsilon\beta p} \rho^{(1+\epsilon)p}
    \|f\|_{L^{\tilde\alpha p}(8B_i)}^{p'} \bigg]\\
    &\le 
    c\, \sfI_i,
\end{align*}
where the last inequality follows from the definition of $\bar\epsilon$ in~\eqref{def:bareps} and the assumption $\rho \le R \le 1$. This provides the analogue of~\eqref{change-1} for the present setting.
In the case $p\in(1,2)$, one obtains for any $\delta \in (0,1]$ that
\begin{align*}
    \int_{B_i} \big|\boldsymbol\tau_h^2 (u - v_i)\big|^p \,\d x
    &\le
    c |h|^{sp} [u-v_i]_{W^{s,p}(16B_i)} \\
    &\le
    c |h|^{sp} \Bigg[
    \bigg[\delta +
    \frac{\big(|h|^\beta\rho^{1-\beta}\big)^{\frac{\chi p}{p-1}}}{\delta^{\frac{2-p}{p-1}p'}}\bigg][u]^p_{W^{s,p}(16B_i)} \\
    &\qquad\qquad\quad +
    \frac{\big(|h|^\beta\rho^{1-\beta}\big)^{(1-s+\epsilon)p}}{\delta^{\frac{2-p}{p-1}p'}} 
    \|f\|_{L^{\tilde\alpha p}(8B_i)}^{p'}\Bigg] \nonumber\\
    &\le
    c |h|^{sp} \Bigg[
    \bigg[\delta +
    \frac{\big(|h|^\beta\rho^{1-\beta}\big)^{\frac{\chi p}{p-1}}}{\delta^{\frac{2-p}{p-1}p'}}\bigg]\big(|h|^\beta\rho^{1-\beta}\big)^{(\sigma-s)p} [u]_{W^{\sigma,p}(16B_i)}^p \\
    &\qquad\qquad\quad + 
    \frac{\big(|h|^\beta\rho^{1-\beta}\big)^{(1-s+\epsilon)p}}{\delta^{\frac{2-p}{p-1}p'}}
    \|f\|_{L^{\tilde\alpha p}(8B_i)}^{p'} \Bigg] \\
    &=
    c \Big(\frac{|h|}{\rho}\Big)^{sp+(\sigma-s)\beta p} \Bigg[
    \bigg[\delta +
    \Big(\frac{|h|}{\rho}\Big)^{\frac{\chi}{p-1}\beta p} \frac{\rho^{\frac{\chi p}{p-1}}}{\delta^{\frac{2-p}{p-1}p'}}\bigg]
    \rho^{\sigma p}[u]_{W^{\sigma,p}(16B_i)}^p \\
    &\qquad\qquad\quad + 
    \Big(\frac{|h|}{\rho}\Big)^{(1-\sigma+\epsilon)\beta p} \frac{\rho^{(1+\epsilon)p}}{\delta^{\frac{2-p}{p-1}p'}}
    \|f\|_{L^{\tilde\alpha p}(8B_i)}^{p'} \Bigg].
\end{align*}
We now set 
\[
\delta=\Big(\frac{|h|}{\rho}\Big)^{\bar\epsilon\beta p}.
\]
Then
\begin{align*}
    \delta +
    \Big(\frac{|h|}{\rho}\Big)^{\frac{\chi}{p-1}\beta p} \frac{\rho^{\frac{\chi p}{p-1}}}{\delta^{\frac{2-p}{p-1}p'}}
    =
    \Big(\frac{|h|}{\rho}\Big)^{\bar\epsilon\beta p} +
    \rho^{\frac{\chi p}{p-1}} \Big(\frac{|h|}{\rho}\Big)^{[\frac{\chi}{p-1}-\frac{2-p}{p-1}p'\bar\epsilon]\beta p}
    \le 
    2\Big(\frac{|h|}{\rho}\Big)^{\bar\epsilon\beta p},
\end{align*}
where we used that \(\rho \le R \le 1\) and that
$$
    \frac{\chi}{p-1}-\frac{2-p}{p-1}p'\bar\epsilon
    \ge 
    \frac{\bar\epsilon}{(p-1)^2}-\frac{2-p}{p-1}p'\bar\epsilon
    =
    \frac{\bar\epsilon}{(p-1)^2} \big[1-(2-p)p\big]
    =
    \bar\epsilon
$$
by the definition of \(\bar\epsilon\) in~\eqref{def:bareps}. Moreover, we have 
$$
    \delta^{-\frac{2-p}{p-1}p'}\Big(\frac{|h|}{\rho}\Big)^{(1-\sigma+\epsilon)\beta p}
    \le  
    \Big(\frac{|h|}{\rho}\Big)^{-\frac{2-p}{p-1}p'\bar\epsilon\beta p} \Big(\frac{|h|}{\rho}\Big)^{\epsilon\beta p}
    \le 
    \Big(\frac{|h|}{\rho}\Big)^{\bar\epsilon\beta p},
$$
since \(\sigma<1\) and 
$$
    \epsilon-\frac{2-p}{p-1}p'\bar\epsilon
    \ge 
    \frac{\bar\epsilon}{(p-1)^2}-\frac{2-p}{p-1}p'\bar\epsilon
    =
    \frac{\bar\epsilon}{(p-1)^2} \big[1-(2-p)p\big]
    =
    \bar\epsilon
$$
by the definition of \(\bar\epsilon\) in~\eqref{def:bareps}.
In conclusion, we obtain
\begin{align*}
    \int_{B_i} \big|\boldsymbol\tau_h^2 (u - v_i)\big|^p \,\d x
    &\le
    c \Big(\frac{|h|}{\rho}\Big)^{sp+(\sigma-s+\bar\epsilon)\beta p} 
    \Big[ \rho^{\sigma p}[u]_{W^{\sigma,p}(16B_i)}^p +
    \rho^{(1+\epsilon)p}
    \|f\|_{L^{\tilde\alpha p}(8B_i)}^{p'} \Big] \\
    &=
    c\, \sfI_i .
\end{align*}
Note that this inequality is precisely the analogue of~\eqref{change-1<}.

The estimates for the remaining terms remain unchanged, apart from the choice of $\theta$, 
since Lemma~\ref{lem:vi-a} exactly parallels Lemma~\ref{lem:vi}. 
Hence, we arrive at the analogue of~\eqref{sstep1}, namely 
\begin{align*}
    \int_{B_{\frac12 \rho}(y_o)} \big|\boldsymbol\tau_h^2 u\big|^p \,\dx
    \le
    c \Big(\frac{|h|}{\rho}\Big)^{\frac{\sigma(\theta-sp)+\bar\epsilon\theta}{\theta-sp+\bar\epsilon p}\, p}
    \sfK^p,
\end{align*}
where $c=c(n,p,s,C_o,C_1,\chi,\tilde\alpha,\sigma)$, $\theta=\theta(s,p)$, and
\begin{equation*}
    \sfK^p
    :=
    \rho^{\sigma p}[u]_{W^{\sigma,p}(B_{\rho}(y_o))}^p + 
    \rho^{n}\Tail\big(u-(u)_{y_o,\rho};B_\rho(y_o)\big)^p + \rho^{(1+\epsilon) p}\|f\|_{L^{\tilde\alpha p}(B_\rho(y_o))}^{p'}.
\end{equation*}

\smallskip

\noindent
\textit{Steps 2--4.} 
From this point onward, the proof proceeds exactly as in Theorem~\ref{thm:higher-diff}, 
with the only modification being the definition of $\bar\epsilon$ in~\eqref{def:bareps}. 
Since no further adjustments are needed, we omit the details and refer the reader to Steps 2--4 in the proof of Theorem~\ref{thm:higher-diff}. 
\end{proof}

\section{Comparison estimate at the gradient level.}

We next establish a comparison estimate at the gradient level.
\begin{proposition}[Comparison at the gradient level]\label{prop:comparison}
Let $p>1$ and $s\in(0,1)$ with $sp'>1$, let $0<C_o\le C_1$, and let $\tilde\alpha\in[\alpha,\frac{1}{p-1}]$ be such that $\epsilon>0$, where $\alpha$ and $\epsilon$ are defined in~\eqref{def:alpha} and~\eqref{def:eps}.
Let $\delta\in(0,1]$. Then there exist constants 
$c=c(n,p,s,C_o,C_1,\chi,\tilde\alpha)>0$, 
$\gamma=\gamma(p,s,\chi,\tilde\alpha)>0$, 
and a radius $\widetilde R_o=\widetilde R_o(\delta , n,p,s, C_o,C_1 ,\chi,\tilde\alpha, R_o)\in(0,1]$, where $R_o$ is from~\eqref{eq:a-condition}, 
such that the following holds. 
Whenever $u$ is a local weak solution to~\eqref{eq:frac-p-lap} in the sense of Definition~\ref{def:weak-sol} with coefficients $a$ satisfying~\eqref{eq:a-condition}, $f\in L^{\tilde\alpha p}(\Omega)$, 
$B_{2R}\equiv B_{2R}(x_o)\Subset\Omega$ with $0<R\le\widetilde R_o$, 
and $v\in W^{s,p}(B_R(x_o)) \cap L^{p-1}_{sp}(\R^n)$ denotes the unique weak solution to the homogeneous Dirichlet problem~\eqref{CD-homo-a} with frozen coefficients $a_{x_o,R}$ in the sense of Definition~\ref{def:weak-sol-D}, the following estimate holds:
\begin{align*}
    \|\nabla u &- \nabla v\|_{L^p(B_{\frac14 R})}\\ 
    &\leq 
    \delta  
    \Big[ 
    \| \nabla u \|_{L^{p}(B_{R})} +  
    R^{\frac{n}{p}-1} \Tail \big(u-(u)_{{R}}; B_{R}\big)\Big] 
    +
    \frac{c}{\delta^{\frac1\gamma}} R^{\epsilon} \| f \|_{L^{\tilde \alpha p}(B_{R})}^\frac{1}{p-1}.
\end{align*}
\end{proposition}

\begin{proof}
Let $\gamma_o$ be the constant from Lemma~\ref{lem:vi-a}, and let 
$\gamma = \gamma(p,s,\chi,\tilde{\alpha})>0$ be the one from Theorem~\ref{thm:higher-diff-2}. 
Then $\gamma < \gamma_o$; see Remark~\ref{rem:go}.  
Applying the interpolation inequality from Lemma~\ref{lem:GN}, we obtain
\begin{align*}
    \|\nabla u &- \nabla v\|_{L^p(B_{\frac14 R})} \\
    &\leq 
    \frac{c}{R^{\frac{\gamma}{1+\gamma}}}
    \|u-v\|_{L^p(B_{{\frac14R}})}^{\frac{\gamma}{1+\gamma}} 
    \bigg[ 
    \frac{1}{R}\,\|u-v\|_{L^p(B_{{\frac14R}})}
    + 
    R^\gamma [\nabla u-\nabla v]_{W^{\gamma,p}(B_{{\frac14R}})}\bigg]^{\frac{1}{1+\gamma}}.
\end{align*}
By Theorem~\ref{thm:higher-diff-2} together with the embedding from Lemma~\ref{lem:FS-S}, we deduce
\begin{align*}
    & [\nabla u]_{W^{\gamma,p}(B_{{\frac14R}})}\\
    &\quad\leq 
    \frac{c}{R^{1+\gamma}} \Big[ R \| \nabla u \|_{L^{p}(B_{R})} +  R^\frac{n}{p} \Tail \big(u-(u)_{R}; B_{R}\big) + R^{1+\epsilon} \| f \|_{L^{\tilde \alpha p}(B_{R})}^\frac{1}{p-1} \Big],
\end{align*}
where $c=c(n,p,s,C_o,C_1,\chi,\tilde\alpha)>1$ and $R \leq R_o$. Furthermore, by Lemmas~\ref{lem:vi-a}, \ref{lem:tail}, and \ref{lem:FS-S}, we have
\begin{align*}
    &[\nabla v]_{W^{\gamma,p}(B_{\frac14 R})} \\
    &\quad\leq
    \frac{c}{R^{1+\gamma}} \Big[ R^{s}[v]_{W^{s,p}(B_{{\frac12R}})} +
    R^\frac{n}{p} \Tail\big(v-(v)_{{\frac12R}};B_{{\frac12R}}\big)\Big] \\
    &\quad\leq
    \frac{c}{R^{1+\gamma}} \Big[ 
    R^{s}[u-v]_{W^{s,p}(B_{R})} +
    R^{s}[u]_{W^{s,p}(B_{R})} \\
    &\qquad\qquad\quad + 
    R^\frac{n}{p} \Tail\big(u-(u)_{R};B_{R}\big) +
    \|u-(u)_R\|_{L^p(B_R)} \Big] \\
    &\quad \leq 
    \frac{c}{R^{1+\gamma}}\Big[ 
    R^s[u-v]_{W^{s,p}(B_{R})} +
    R\|\nabla u\|_{L^p(B_R)} +
    R^{\frac{n}{p}} \Tail \big(u-(u)_{R}; B_R\big)\Big],
\end{align*}
where
$c=c(n,p,s,C_o,C_1,\chi)>0$.
Inserting these estimates above, we obtain 
\begin{align}\label{grad-comp-1}\nonumber
    \|\nabla u &- \nabla v\|_{L^p(B_{\frac14 R})} \\
    &\leq 
    \frac{c}{R^{\frac{\gamma}{1+\gamma}}}
    \|u-v\|_{L^p(B_{R})}^{\frac{\gamma}{1+\gamma}} \nonumber\\
    &\quad \cdot
    \bigg[ 
    \frac{1}{R}\,\|u-v\|_{L^p(B_{R})} +
    \frac{1}{R^{1-s}}[u-v]_{W^{s,p}(B_{R})} + 
    \| \nabla u \|_{L^{p}(B_R)} \nonumber\\
    &\qquad +  
    R^{\frac{n}{p}-1} \Tail \big(u-(u)_{R}; B_R\big) + 
    R^{\epsilon} \| f \|_{L^{\tilde \alpha p}(B_R)}^\frac{1}{p-1}\bigg]^{\frac{1}{1+\gamma}}  ,
\end{align}
where $c=c(n,p,s,C_o,C_1,\chi,\tilde\alpha)>1$. 

Denote by $\tilde \delta \in (0,1]$ be the parameter from Lemma~\ref{lem:compar-var-coeff} (ii). Note that assuming 
\begin{equation*}
     R \le \tilde \delta^{\frac{1}{\chi p'}[1 + \frac{(2-p)_+}{p-1} p']}  
\end{equation*}
we have
$$
\frac{R^{\frac{\chi p}{p-1}}}{\tilde\delta^{\frac{(2-p)_+}{p-1}p'}}
    \le 
    \tilde\delta,
$$
for all $p \in (1,\infty)$. Hence, combining Lemma~\ref{lem:compar-var-coeff} (i) and (ii) together with Lemma~\ref{lem:FS-S} gives
\begin{align*}
    \|u-v\|_{L^p(B_{R})} & + R^s[u-v]_{W^{s,p}(B_{R})} \\
    &\le 
    c\tilde\delta^{\frac{1}{p}} R^{s}
     [u]_{W^{s,p}(B_{2R})} + 
     \frac{c}{\tilde\delta^{\frac{(2-p)_+}{(p-1)^2}}} 
     R^{1+\epsilon}
    \|f\|_{L^{\tilde\alpha p}(B_R)}^{\frac{1}{p-1}}\\
    &\le 
    c\tilde\delta^{\frac{1}{p}} R
     \| \nabla u \|_{L^{p}(B_{2R})} + 
     \frac{c}{\tilde\delta^{\frac{(2-p)_+}{(p-1)^2}}} 
     R^{1+\epsilon}
    \|f\|_{L^{\tilde\alpha p}(B_R)}^{\frac{1}{p-1}}.
\end{align*}
In~\eqref{grad-comp-1}, we apply the preceding estimate twice: first with $\tilde\delta\in(0,1]$ to control the term $\|u-v\|_{L^p(B_R)}^{\frac{\gamma}{1+\gamma}}$, and second with $\tilde\delta=1$ to bound 
\[
\frac{1}{R}\,\|u-v\|_{L^p(B_R)} + \frac{1}{R^{1-s}} [u-v]_{W^{s,p}(B_R)}.
\]
Proceeding in this way, we obtain
\begin{align*}
    \|\nabla u &- \nabla v\|_{L^p(B_{\frac12 R})} \\
    &\leq 
    c
    \bigg[\tilde\delta^{\frac{1}{p}} 
     \| \nabla u \|_{L^{p}(B_{2R})} + 
     \frac{R^\epsilon}{\tilde\delta^{\frac{(2-p)_+}{(p-1)^2}}} 
    \|f\|_{L^{\tilde\alpha p}(B_R)}^{\frac{1}{p-1}} \bigg]^{\frac{\gamma}{1+\gamma}} \\
    &\phantom{\le\,}\cdot
    \bigg[ 
    \| \nabla u \|_{L^{p}(B_R)} +  
    R^{\frac{n}{p}-1} \Tail \big(u-(u)_{R}; B_R\big) + 
    R^{\epsilon} \| f \|_{L^{\tilde \alpha p}(B_R)}^\frac{1}{p-1}\bigg]^{\frac{1}{1+\gamma}} \\
    &\le 
    \tfrac12\delta \bigg[ 
    \| \nabla u \|_{L^{p}(B_R)} +  
    R^{\frac{n}{p}-1} \Tail \big(u-(u)_{R}; B_R\big) + 
    R^{\epsilon} \| f \|_{L^{\tilde \alpha p}(B_R)}^\frac{1}{p-1}\bigg] \\
    &\phantom{\le\,} +
    c\delta^{-\frac{1}{\gamma}}
    \bigg[\tilde\delta^{\frac{1}{p}} 
     \| \nabla u \|_{L^{p}(B_{2R})} + 
     \frac{R^\epsilon}{\tilde\delta^{\frac{(2-p)_+}{(p-1)^2}}} 
    \|f\|_{L^{\tilde\alpha p}(B_R)}^{\frac{1}{p-1}} \bigg] .
\end{align*}
We now set
\[
\tilde\delta :=  \frac{\delta^{\frac{p(1+\gamma)}{\gamma}} }{(2c)^{p}},
\]
so that
$
c\, \tilde\delta^{\frac{1}{p}} = \tfrac12 \, \delta^{1+\frac{1}{\gamma}}$. 
Note that this choice of $\tilde\delta$ together with $R \leq R_o$ requires
\[
R \le \min\Bigg\{R_o,
\bigg(\frac{\delta^{\frac{1+\gamma}{\gamma}}}{2c}\bigg)^{\frac{p-1}{\chi} \left[ 1 + \frac{(2-p)_+}{p-1}p' \right]} \Bigg\}
=: \widetilde R_o,
\]
so that the preceding estimates are valid. Moreover, with this choice of $\tilde\delta$, the preceding inequality becomes
\begin{align*}
    \|\nabla u &- \nabla v\|_{L^p(B_{\frac12 R})} \\
    &\le 
    \delta \Big[ 
    \| \nabla u \|_{L^{p}(B_R)} +  
    R^{\frac{n}{p}-1} \Tail \big(u-(u)_{R}; B_R\big) \Big] 
    +
     \frac{c\,R^\epsilon}{\delta^{\frac{1}{\gamma} + \frac{(\gamma + 1)(2-p)_+}{\gamma(p-1)}p'}}
    \|f\|_{L^{\tilde\alpha p}(B_R)}^{\frac{1}{p-1}}.
\end{align*}
Noting that the exponent
\[
\frac{1}{\gamma} + \frac{(\gamma + 1)(2-p)_+}{\gamma(p-1)}p'
\]
depends only on $p$, $s$, $\chi$, and $\tilde\alpha$, we thus obtain the comparison estimate with this exponent.
Finally, by a suitable redefinition of the parameter $\gamma$, the comparison estimate holds uniformly for all $p>1$.
\end{proof}

\section{Proof of the gradient estimate}\label{sec:proof}

\subsection{Choice of Parameters}\label{sec:parameters}
Let $s\in(0,1)$ with $sp'>1$ and 
\begin{equation*}
    r
    \in
    \bigg(\alpha p\,,\, \frac{n}{(p-1)(sp^{\prime}-1)}\bigg),
\end{equation*}
where $\alpha$ is defined in~\eqref{def:alpha}. Next, we set
$$
    q=\frac{rn(p-1)}{n-r(p-1)(sp^{\prime}-1)}
    > p
    \quad\mbox{and}\quad 
    \mu=\frac{r}{q}
    =
    \frac{n}{(p-1)\,[\,n+q(sp'-1)\,]},
$$
where the inequality $q>p$ follows from the fact that $q$ is monotonically increasing with respect to $r$ and the fact that $r>\alpha p$. 
We now specify the choice of the parameter $\tilde\alpha$.
In the case $s\in(\frac{1}{p'},s_o]$ we have
\[
    \alpha
    =
    \frac{n}{(p-1)\,[\,n+p(sp'-1)\,]}
    <
    \min\bigg\{\frac{1}{p-1}\,,\, \frac{n}{p(p-1)(sp'-1)} \bigg\}.
\]
Since the mapping
\[
(1,\infty)\ni q \longmapsto \frac{nq}{(p-1)[\,n+q(sp'-1)\,]}
\]
is increasing, the assumption $q>p$ implies
\[
\mu q>\alpha p\ge 1.
\]
In the case $s\in(s_o,1)$ we have 
$$
    \alpha
    =
    \frac{1}{p}
    <
    \min\bigg\{\frac{1}{p-1}\,,\, \frac{n}{p(p-1)(sp'-1)} \bigg\}.
$$
We additionally assume that
\[
    q
    > 
    \frac{n(p-1)}{n-(p-1)(sp'-1)}
    =
    \frac{n(p-1)}{n-1+p(1-s)},
\]
which guarantees that $\mu q > 1=\alpha p$.

Hence, in any case we may choose
\begin{equation}\label{restrict:tilde-alpha}
        \tilde\alpha
    \in
    \bigg(
        \alpha\,,\,
        \min\bigg\{\frac{1}{p-1},\frac{n}{p(p-1)(sp'-1)}\bigg\}
    \bigg)
\quad\text{such that}\quad
\tilde\alpha p< \mu q,
\end{equation}
and define $\epsilon$ according to~\eqref{def:eps}.
Note that this choice guarantees $\epsilon>0$. In fact, in the case $s\in(\frac{1}{p'},s_o]$
we have 
$$
    \epsilon
    >
    sp' - 1 - \frac{n}{p}\Big(\frac{1}{\alpha(p-1)}-1\Big)
    =
    0,
$$
while in the case $s\in(s_o,1)$
we have 
$$
    \epsilon
    >
    sp' - 1 - \frac{n}{p}\Big(\frac{p}{p-1}-1\Big)
    =
    sp' - 1 - \frac{n}{p(p-1)}
    >
    0.
$$
Moreover, the condition $\tilde\alpha<\frac{n}{p(p-1)(sp'-1)}$ implies that
\[
    \epsilon 
    <
    sp' - 1 - \frac{n}{p}\bigg[\frac{p(sp'-1)}{n}-1\bigg]
    =
    \frac{n}{p}.
\]

Recall that $r=\mu q$. Then, we can show that for every $f\in L^{\mu q}(B_R)$ it holds
\begin{align}\label{crazy-holder}
    \|f\|_{L^{\tilde\alpha p}(B_R)}^{\frac{1}{p-1}}
    \le
    \|f\|_{L^{\frac{n\mu p}{\,n-\epsilon p\,}}(B_R)}^{\mu}
    \|f\|_{L^{\mu q}(B_R)}^{\frac{\mu q(sp'-1)}{n}} .
\end{align}
The estimate follows from Hölder's inequality and can be interpreted
as an interpolation between the spaces
$L^{\frac{n\mu p}{n-\epsilon p}}(B_R)$ and $L^{\mu q}(B_R)$. Since the exponents involved are rather intricate, we spell out the
argument, including the application of Hölder's inequality and the
explicit computation of the corresponding exponents. By Hölder's inequality applied with the conjugate exponents
\(\frac{1}{1-\theta}\) and \(\frac{1}{\theta}\), where \(\theta\in(0,1)\)
will be chosen later, we obtain
\begin{align}\label{crazy-holder-}\nonumber
    \bigg[\int_{B_R} |f|^{\tilde\alpha p} \,\dx\dt \bigg]^{\frac{1}{\tilde\alpha p}} 
    &=
    \bigg[\int_{B_R} |f|^{\tilde\alpha p-\theta\mu q}|f|^{\theta\mu q} \,\dx\dt \bigg]^{\frac{1}{\tilde\alpha p}} \\
    &\le 
    \bigg[\int_{B_R} |f|^{\frac{\tilde\alpha p-\theta\mu q}{1-\theta}} \,\dx\dt \bigg]^{\frac{1-\theta}{\tilde\alpha p}} 
    \bigg[\int_{B_R} |f|^{\mu q} \,\dx\dt \bigg]^{\frac{\theta}{\tilde\alpha p}}.
\end{align}
We now choose
$$
    \theta := \frac{\tilde \alpha p (p-1) (sp'-1)}{n}.
$$
Note that \(\theta\in(0,1)\) by the choice of $\tilde\alpha$ and the assumption $sp'>1$. We compute 
\begin{align*}
    \frac{1-\theta}{\tilde\alpha(p-1)}
    =
    \frac{1}{\tilde\alpha(p-1)} - \frac{p (sp'-1)}{n}
    =
    1-\frac{\epsilon p}{n}
\end{align*}
and hence 
\begin{equation*}
    1-\theta
    =
    \frac{\tilde\alpha(p-1)(n-\varepsilon p)}{n}.
\end{equation*}
In order to proceed, we need to compute the integrability exponent of $|f|$
appearing in the first integral on the right-hand side. We obtain
\begin{align*}
    \frac{\tilde\alpha p-\theta\mu q}{1-\theta}
    &=
    \frac{n}{\tilde\alpha(p-1)(n-\epsilon p)}
    \bigg[\tilde\alpha p-\frac{\tilde\alpha p(p-1)(sp'-1)\mu q}{n}\bigg] \\
    &=
    \frac{p}{n-\epsilon p}
    \frac{n-(p-1)(sp'-1)\mu q}{p-1} \\
    &=
    \frac{\mu p}{n-\epsilon p}
    \bigg[\frac{n}{\mu(p-1)}-q(sp'-1)\bigg] \\
    &=
    \frac{n\mu p}{n-\epsilon p}
    >
    0,
\end{align*}
since $\epsilon<\frac{n}{p}$. In particular, this also implies that the exponent $\tilde\alpha p-\theta\mu q$ of $|f|$ is positive. Inserting the definition of $\theta$, the value of $1-\theta$ computed above and the preceding identity into~\eqref{crazy-holder-}, yields the claim~\eqref{crazy-holder}.

\subsection{Stopping time argument.}
Let $\delta\in(0,1)$ be fixed later in dependence on the data, and let 
$\widetilde R_o=\widetilde R_o(\delta,n,p,s,\chi,\tilde\alpha,R_o)\in(0,1]$ 
denote the radius given by Proposition~\ref{prop:comparison}. 
Let $B_R(x_o)\Subset\Omega$ be a ball with $R\le \widetilde R_o$.
In what follows, we suppress the center $x_o$ from the notation and write
$B_R\equiv B_R(x_o)$.
For \(\frac12 R\le r_1< r_2\le R\), we consider the concentric balls
\(
    B_{\frac12R}\subset B_{r_1}\subset B_{r_2}\subset B_R .
\)
With the parameter $\tilde\alpha$ introduced in \S\,\ref{sec:parameters}, we define
\begin{align}\label{def:lambda_0}\nonumber
    \lambda_o^p
    &:=
    \bint_{B_R} |\nabla u|^p \,\dx
    + \mathsf M^p
    |B_R|^{\frac{p}{n}(sp'- 1)}
    \bigg[
        \bint_{B_R} |f|^{\tilde\alpha p} \,\dx
    \bigg]^{\frac{1}{\tilde\alpha(p-1)}}\\
    &\phantom{:=}
    + R^{-p} \Tail\big(u-(u)_R;B_R\big)^p ,
\end{align}
where \(\mathsf M\ge1\) is a parameter to be fixed later. 

For any point $y_o \in B_{r_1}$ and radius $r$ with
\begin{equation}\label{range-r}
    \frac{r_2-r_1}{2^7} \leq r \leq \frac{r_2-r_1}{2},
\end{equation}
we have $B_r(y_o)\subset B_{r_2}$.
We define 
\begin{equation}\label{def:B}
    {\sf B}:=
    \bigg(
    \frac{2^{7}R}{r_2-r_1}
    \bigg)^{\frac{n}{p-1}+1}.
\end{equation}
Then, for 
\begin{equation} \label{eq:lambda-geq-Blambda0}
    \lambda > {\sf B} \lambda_o,
\end{equation}
we have
\begin{align} \label{eq:lambda-sub}
    &\bint_{B_r(y_o)} |\nabla u|^p \,\dx + 
    {\sf M}^p|B_r|^{\frac{p}{n}(sp'- 1)} \bigg[\bint_{B_r(y_o)} |f|^{\tilde\alpha p} \,\dx \bigg]^{\frac{1}{\tilde\alpha(p-1)}} \nonumber\\
    &\quad\le 
    \Big(\frac{R}{r}\Big)^n\bint_{B_R} |\nabla u|^p \,\dx + 
    \Big(\frac{R}{r}\Big)^{{n-\epsilon p}} {\sf M}^p|B_R|^{\frac{p}{n}(sp'- 1)} \bigg[\bint_{B_R} |f|^{\tilde\alpha p} \,\dx \bigg]^{\frac{1}{\tilde\alpha(p-1)}} \nonumber\\
    &\quad \le 
    {\sf B}^p\lambda_o^p 
    < 
    \lambda^p
\end{align}
for every $r$ as in~\eqref{range-r}. 
Here, we have also used the definition of $\epsilon$ from~\eqref{def:eps} and recalled that 
$\varepsilon p < n$, which has been shown from the assumption~\eqref{restrict:tilde-alpha}.
Now let $\lambda$ be as in \eqref{eq:lambda-geq-Blambda0} and consider the {\it super-level set}
$$
    \boldsymbol E(\lambda,r_1) 
    := 
    \big\{\mbox{$x \in B_{r_1}$: $x$ is a Lebesgue point of $ |\nabla u|$ and $|\nabla u|(x) > \lambda$}\big\}.
$$
In view of Lebesgue's differentiation theorem, for every
\(y_o\in \boldsymbol E(\lambda,r_1)\) we have
\[
    \lim_{\rho\downarrow 0}
    \bint_{B_\rho(y_o)} |\nabla u|^p \,\dx
    =
    |\nabla u(y_o)|^p
    >
    \lambda^p .
\]
By the absolute continuity of the integral, and since the contribution of
the inhomogeneity is nonnegative, we conclude that
\begin{align*}
    \bint_{B_\rho(y_o)} |\nabla u|^p \,\dx
    &+
    {\sf M}^p
    |B_\rho|^{\frac{p}{n}(sp'- 1)}
    \bigg[
        \bint_{B_\rho(y_o)} |f|^{\tilde\alpha p} \,\dx
    \bigg]^{\frac{1}{\tilde\alpha(p-1)}}
    >
    \lambda^p ,
\end{align*}
for all sufficiently small \(\rho>0\).
In combination with~\eqref{eq:lambda-sub}, the absolute continuity of the
integral guarantees the existence of a maximal radius
\(\rho_{y_o}\in\big(0,\tfrac{r_2-r_1}{2^7}\big)\) such that
\begin{equation*}
    \bint_{B_{\rho_{y_o}}(y_o)} |\nabla u|^p \,\dx + 
    {\sf M}^p|B_{\rho_{y_o}}|^{\frac{p}{n}(sp'- 1)} \bigg[\bint_{B_{\rho_{y_o}}(y_o)} |f|^{\tilde\alpha p} \,\dx \bigg]^{\frac{1}{\tilde\alpha(p-1)}}  
    =
    \lambda^p, 
\end{equation*}
and 
\begin{equation*}
    \bint_{B_\rho(y_o)} |\nabla u|^p \,\dx +
    {\sf M}^p|B_\rho|^{\frac{p}{n}(sp'- 1)} \bigg[\bint_{B_\rho(y_o)} |f|^{\tilde\alpha p} \,\dx \bigg]^{\frac{1}{\tilde\alpha(p-1)}}
    <
    \lambda^p 
\end{equation*}
for every $\rho\in (\rho_{y_o}, \frac{r_2-r_1}{2}]$.

\subsection{Covering of super-level sets}
As before, let \(\lambda\) satisfy~\eqref{eq:lambda-geq-Blambda0} and
consider the family of balls
\begin{equation*}
    \mathcal{F}
    :=
    \big\{
        B_{\rho_{y_o}}(y_o)
        \colon
        y_o\in \boldsymbol E(\lambda,r_1)
    \big\},
\end{equation*}
constructed in the previous step.
By Vitali's covering theorem, we can extract a countable subfamily of
pairwise disjoint balls \(\{B_i\equiv B_{\rho_{x_i}}(x_i)\}_{i\in\mathbb N}\subset\mathcal F\)
such that
\[
    \boldsymbol E(\lambda,r_1)
    \subset
    \bigcup_{i\in\mathbb N} B_{5\rho_{x_i}}(x_i).
\]
In particular, for every \(i\in\mathbb N\) we have
\(x_i\in \boldsymbol E(\lambda,r_1)\) and
\(\rho_i:=\rho_{x_i}\in\big(0,\tfrac{r_2-r_1}{2^7}\big)\), with
\(\rho_i\) determined by the stopping-time argument, so that
\begin{equation}\label{eq:intr}
    \bint_{B_i} |\nabla u|^p \,\dx +
    {\sf M}^p|B_i|^{\frac{p}{n}(sp'- 1)} \bigg[\bint_{B_i} |f|^{\tilde\alpha p} \,\dx \bigg]^{\frac{1}{\tilde\alpha(p-1)}} 
    = 
    \lambda^p,
\end{equation}
and
\begin{equation} \label{eq:stopping-time-above-maxradius}
    \bint_{B_\rho(x_i)} |\nabla u|^p \,\dx +
    {\sf M}^p|B_\rho|^{\frac{p}{n}(sp'- 1)} \bigg[\bint_{B_\rho(x_i)} |f|^{\tilde\alpha p} \,\dx \bigg]^{\frac{1}{\tilde\alpha(p-1)}}
    < 
    \lambda^p 
\end{equation}
for every $\rho\in \big(\rho_{i}, \frac12 (r_2-r_1)\big]$.
Finally, we introduce the notation
\[
    B_i^{(j)} := B_{2^j\rho_i}(x_i),
    \qquad i,j\in\mathbb N,
\]
and note that, by construction of the radii \(\rho_i\), we have
\(B_i^{(7)}\subset B_R\) for every \(i\in\mathbb N\).

\subsection{Comparison}
Let $v_i\in W^{s,p}(B_i^{(6)}) \cap L^{p-1}_{sp} (\R^n)$ be the unique weak solution to the Dirichlet problem
\begin{equation*} 
\left\{
\begin{array}{cl}
      \big(- a_{x_i,2^6\rho_i}\Delta_p\big)^s v_i  = 0 
      & 
      \mbox{in $ B_i^{(6)}$,} \\[7pt]
      v_i = u & \mbox{a.e.~in $\R^n \setminus B_i^{(6)}$,}
   \end{array}
\right.
\end{equation*}
where the frozen coefficient $a_{x_i,2^6\rho_i}(x,y)$ is defined according to~\eqref{def:frozen}. 
Let \(\mathsf A>1\) be a fixed constant and let \(\theta>p\) be an
integrability exponent, to be specified later depending only on the
data.
Invoking Lemma~\ref{lem:elementary-superlevel} with
\((\gamma,\alpha,K,a,b)\) replaced by
\((p,\theta,\mathsf A\lambda,\nabla u(x),\nabla v_i(x))\), and integrating
over the set
\(B_i^{(3)}\cap\{|\nabla u|>\mathsf A\lambda\}\), we infer that
\begin{align}\label{eq:grad-u-superlevel}\nonumber
    &\int_{B_i^{(3)}\cap\{|\nabla u|>\mathsf A\lambda\}} |\nabla u|^p \,\dx \\ 
    &\qquad\le
    2^{p}\int_{B_i^{(3)}} |\nabla u-\nabla v_i|^p \,\dx
    +
    \frac{2^{\theta-1}}{(\mathsf A\lambda)^{\theta-p}}
    \int_{B_i^{(3)}} |\nabla v_i|^{\theta} \,\dx =:\,\sfI+\sfI\sfI,
\end{align}
where the terms \(\sfI\) and \(\sfI\sfI\) are defined by the
preceding right-hand side. The contributions \(\sfI\) and \(\sfI\sfI\) are treated
separately.
For any \(\delta\in(0,1]\), the comparison estimate at the gradient level
from Proposition~\ref{prop:comparison}, applied on \(B_i^{(4)}\), yields
\begin{align*}
    \sfI
    &\leq 
    2^p \delta^p
    \bigg[ 
    \int_{B_i^{(6)}}|\nabla u|^p\,\d x
    + 
    \big(2^6\rho_i\big)^{n -p} \mathrm{Tail} 
    \big(u-(u)_{B_i^{(6)}};B_i^{(6)}
    \big)^p \bigg]\\
    &\phantom{\le\,}
    + 
    \frac{c}{\delta^\frac{p}{\gamma}} \big|B_i^{(6)}\big|^{\frac{\epsilon p}{n}}
    \bigg[\int_{B_i^{(6)}}
    |f|^{\tilde\alpha p}\,\d x \bigg]^{\frac{1}{\tilde\alpha(p-1)}},
\end{align*}
where $c=c(n,p,s,C_o,C_1,\chi,\tilde\alpha)$ and $\gamma = \gamma(p,s,\chi,\tilde{\alpha})>0$. 
Note that \(\rho_i< 2^{6}\rho_i \le \tfrac12 (r_2-r_1)\), so that
\eqref{eq:stopping-time-above-maxradius} can be applied.
First, we infer
\begin{equation*}
    \int_{B_i^{(6)}} |\nabla u|^p \,\dx
    <
    |B_i^{(6)}|\,\lambda^p
    =
    2^{6n}|B_i|\,\lambda^p .
\end{equation*}
Moreover, again by \eqref{eq:stopping-time-above-maxradius} and the definition of $\epsilon$ in~\eqref{def:eps} the last term is bounded by
\begin{equation*}
    |B_i^{(6)}|^{\frac{\varepsilon p}{n}}
    \bigg[
        \int_{B_i^{(6)}} |f|^{\tilde\alpha p}\,\dx
    \bigg]^{\frac{1}{\tilde\alpha(p-1)}}
    <
    \frac{|B_i^{(6)}|}{\mathsf M^p}\,\lambda^p
    =
    \frac{2^{6n}}{\mathsf M^p}|B_i|\,\lambda^p .
\end{equation*}
As for the tail term, observe that, since
\(2^6\rho_i \in \big(0,\tfrac12 (r_2-r_1)\big]\), by a dyadic scaling argument there exists
\(\ell\in\mathbb N_0\) such that
\begin{equation*}
    \tfrac14(r_2-r_1)
    <
    2^\ell\,2^6\rho_i
    \le 
    \tfrac12 (r_2-r_1).
\end{equation*}
Applying Lemma~\ref{lem:tail-est} on $B_i^{(6)}$, we obtain
\begin{align*}
    \mathrm{Tail} \big(u-(u)_{B_i^{(6)}}; B_i^{(6)}\big)
    &\le
    \sfT_1
    +
    \left[ \sum_{k=1}^\ell (\sfT_{2,k})^\tau \right]^\frac{1}{\tau} ,
\end{align*}
where
\begin{align*}
    \sfT_1
    &:=
    c\, 2^{-\ell sp'}
    \Tail \big(u-(u)_{x_i,2^\ell 2^6\rho_i}; B_{2^\ell 2^6\rho_i}(x_i)\big),
    \qquad
    \tau  =\min \{1,p-1\},
\end{align*}
and
\begin{align*}
    \sfT_{2,k}
    &:=
    c  2^{-ksp'}
    \bigg[
        \bint_{B_{2^k2^6\rho_i}(x_i)}
        |u-(u)_{x_i,2^k2^6\rho_i}|^p\, \d x
    \bigg]^\frac1p.
\end{align*}
To estimate the first term, we apply Lemma~\ref{lem:tail} with
$(x_o,r,y_o,R)$ replaced by $(x_i,2^\ell 2^6\rho_i,x_o,R)$ and obtain
\begin{align*}
    \sfT_1 
    &\le 
    c\, 2^{-\ell sp'}\bigg[1 + \frac{|x_i-x_o|}{2^\ell 2^6\rho_i}\bigg]^{\frac{n}{p-1}+sp'} 
    \Big(\frac{2^\ell 2^6\rho_i}{R}\Big)^{sp'} \Tail \big(u-(u)_{x_o,R}; B_R(x_o)\big) \\
    &\phantom{\le\,} + 
    c\, 2^{-\ell sp'}\Big(\frac{R}{2^\ell 2^6\rho_i}\Big)^{\frac{n}{p-1}}
    \bigg[ \bint_{B_{R}(x_o)} \big|u-(u)_{x_o,R}\big|^{p} \, \d x \bigg]^\frac{1}{p} \\
    &\le 
    c\,2^{-\ell sp'}\Big(\frac{R}{r_2-r_1}\Big)^{\frac{n}{p-1}} 
    \Bigg[ \Tail \big(u-(u)_{x_o,R}; B_R(x_o)\big) + 
    R\bigg[ \bint_{B_{R}(x_o)} |\nabla u|^{p} \, \d x \bigg]^\frac{1}{p} \Bigg] \\
    &\le 
    c\,2^{-\ell sp'}\Big(\frac{R}{r_2-r_1}\Big)^{\frac{n}{p-1}} 
    R \lambda_o \\
    &= 
    c\,2^{-\ell sp'}\Big(\frac{R}{r_2-r_1}\Big)^{\frac{n}{p-1}+1} 
    (r_2-r_1) \lambda_o \\
    &\le 
    c\,\underbrace{2^\ell 2^{-\ell sp'}}_{\le 1}
    \underbrace{\Big(\frac{2^7R}{r_2-r_1}\Big)^{\frac{n}{p-1}+1}}_{={\sf B}}\lambda_o
    \rho_i\\
    &\le c\,{\sf B} \lambda_o\rho_i\\
    &
    \le 
    c\,\lambda\rho_i.
\end{align*}
Using Poincar\'e's inequality, \eqref{eq:stopping-time-above-maxradius} and the assumption $sp'>1$, we estimate
    $$
    \left[ \sum_{k=1}^\ell (\sfT_{2,k})^{\tau} \right]^\frac{1}{\tau} \leq c   \lambda \rho_i \left[ \sum_{k=1}^\ell 2^{-k\tau (sp'-1)}\right]^\frac{1}{\tau} \leq  c  \lambda \rho_i.
    $$
Combining the above estimates, we obtain
\begin{equation}\label{est:tail-B5}
     \mathrm{Tail} \big(u-(u)_{B_i^{(6)}}; B_i^{(6)}\big)
    \le c\,\lambda\rho_i.
\end{equation}
Substituting the preceding inequalities into the right-hand side of the
above estimate for $\sfI$, we obtain
\begin{align*}
    \sfI 
    \le 
    c \bigg[\delta + 
    \frac{1}{\delta^\frac{p}{\gamma} \mathsf M^p}\bigg]  |B_i| 
    \lambda^p
    .
\end{align*}
The estimate holds for all $\delta\in(0,1]$.

To estimate the term $\sfI\sfI$ in~\eqref{eq:grad-u-superlevel}, we apply
Lemma~\ref{lem:vi-a} to $v_i$ on $B_i^{(6)}$. For the resulting tail term we recall that
$v_i=u$ on $\R^n \setminus B_i^{(6)}$, which allows us to apply
Lemma~\ref{lem:tv-u}. In this way, we obtain
\begin{align*}
    \sfI\sfI 
    &\leq 
    \frac{c }{(\mathsf A\lambda)^{\theta-p}}  |B_i|
    \left[ \rho_i^{-p(1-s) - n} [v_i]_{W^{s,p}(B_i^{(5)})}^p + 
    \rho_i^{-p}\mathrm{Tail} \big(v_i-(v_i)_{B_i^{(4)}}; B_i^{(4)}\big)^p \right]^\frac{\theta}{p} \\
    &\leq 
    \frac{c}{(\mathsf A\lambda)^{\theta-p}}|B_i| \big[ 
    \sfI\sfI_1 + 
    \sfI\sfI_2 + 
    \sfT_1 + \sfT_2 \big]^\frac{\theta}{p}, 
\end{align*}
where $c = c(n,p,s,\theta)> 0$ and we have set 
\begin{align*}
    \sfI\sfI_1
    :=
    \rho_i^{-(1-s)p-n} [u]_{W^{s,p}(B_i^{(5)})}^p ,
    \qquad 
    \sfI\sfI_2
    :=
    \rho_i^{-p(1-s)-n} [v_i-u_i]_{W^{s,p}(B_i^{(5)})}^p
\end{align*}
and 
\begin{equation*}
    \sfT_1
    :=
    \rho_i^{-p}\mathrm{Tail} \big(u-(u)_{B_i^{(4)}}; B_i^{(4)}\big)^p,
    \qquad
    \sfT_2
    :=
    \rho_i^{-p-n} \int_{B_i^{(6)}} |u-v_i|^p \,\dx  .
\end{equation*}
The first term can be estimated by means of Lemma~\ref{lem:FS-S}
and inequality~\eqref{eq:stopping-time-above-maxradius}. Indeed, we have
\begin{align*}
    \sfI\sfI_1
    \le 
    c \,\rho_i^{-n} \, \|\nabla u\|_{L^{p}(B_i^{(5)})}^p
    \le 
    c\,\lambda^p.
\end{align*}
Moreover, applying Lemma~\ref{lem:tail} with
$(x_o,r,y_o,R)=(x_i,2^4\rho_i,x_i,2^6\rho_i)$, we obtain
\begin{align*}
    \sfT_1
    &\le 
    c\, \rho_i^{-p}\mathrm{Tail} \big(u-(u)_{B_i^{(6)}}; B_i^{(6)}\big)^p
    +
    c \,\rho_i^{-p}
    \mint_{B_i^{(6)}} \big|u-(u)_{B_i^{(6)}}\big|^p\,\dx \\
    &\le
    c\, \rho_i^{-p}\mathrm{Tail} \big(u-(u)_{B_i^{(6)}}; B_i^{(6)}\big)^p +
    c \mint_{B_i^{(6)}} |\nabla u|^p\,\dx \\
    &\le 
    c\,\lambda^p ,
\end{align*}
where in the second inequality we used Poincar\'e's inequality,
while the last estimate follows from
\eqref{eq:stopping-time-above-maxradius} and \eqref{est:tail-B5}.
From the comparison estimate in Lemma~\ref{lem:compar-var-coeff},
applied with $\delta=1$, we infer
\begin{align*}
    \sfI\sfI_2 + \sfT_2
    &\le 
    c\,\rho_i^{-(1-s)p-n} [u]_{W^{s,p}(B_i^{(6)})} +
    c\,\rho_i^{\epsilon p-n} \|f\|_{L^{\tilde\alpha p}(B_i^{(6)})}^{p'} \\
    &\le 
    c \mint_{B_i^{(6)}} |\nabla u|^p\,\dx +
    c\,  \frac{\lambda^p}{\mathsf M^p}\\
    &\le 
    c\,\lambda^p .
\end{align*}
Here, the second inequality follows from Lemma~\ref{lem:FS-S},
while the last estimate is a consequence of the definition of
$\epsilon$ in~\eqref{def:eps} together with
\eqref{eq:stopping-time-above-maxradius}.
Moreover, we used that $\rho_i\le 1$ and $\mathsf M\ge 1$ in the application
of Lemma~\ref{lem:compar-var-coeff} and in the passage to the final line. 
Overall, we have shown that 
$$
    \sfI\sfI 
    \leq 
    \frac{c}{\mathsf A^{\theta-p}} \lambda^p |B_i|,
$$
for a constant $c = c(n,p,s,C_o,C_1,\theta)> 0$.
Combining this estimate with the bound for $\sfI$, we conclude that
\begin{align}\label{est:super-nabla-u}
    \int_{B_i^{(3)}\cap \left\{ |\nabla u| >\mathsf A \lambda \right\}} |\nabla u|^p \, \d x 
    \le 
    c\, {\sf Q} \, \lambda^p |B_i| ,
\end{align}
where 
\begin{equation}\label{def:M}
    {\sf Q}
    :=
    \delta + 
    \frac{1}{\delta^\frac{p}{\gamma} \mathsf M^p} +
    \frac{1}{\mathsf A^{\theta-p}}.
\end{equation}

\subsection{Estimates on super-level sets}
We estimate the second term on the left-hand side of \eqref{eq:intr}. By the definition of $\epsilon$ in~\eqref{def:eps} and Young’s inequality with exponents $\frac{n}{\epsilon p}$ and $\frac{n}{n-\epsilon p}$, we have
\begin{align*}
    {\sf M}^p|B_i|^{\frac{p}{n}(sp'- 1)} & \bigg[\bint_{B_i} |f|^{\tilde\alpha p} \,\dx \bigg]^{\frac{1}{\tilde\alpha(p-1)}} \\
    &=
    \mathsf M^p|B_i|^{\frac{\epsilon p}{n}-1} \|f\|^{p'}_{L^{\tilde\alpha p}(B_i)} \\
    &= 
    \lambda^{\frac{\epsilon p^2}{n}} \lambda^{-\frac{\epsilon p^2}{n}}\mathsf M^p|B_i|^{\frac{\epsilon p}{n}-1} \|f\|^{p'}_{L^{\tilde\alpha p}(B_i)} \\
    &\le 
    \tfrac14 \lambda^p +
    4^\frac{\epsilon p}{n-\epsilon p}|B_i|^{-1}
    \Big[\lambda^{-\frac{\epsilon p^2}{n}}\mathsf M^p\|f\|^{p'}_{L^{\tilde\alpha p}(B_i)}\Big]^{\frac{n}{n-\epsilon p}} \\
    &= 
    \tfrac14 \lambda^p +
    |B_i|^{-1}\Big[4^\frac{\epsilon p}{n}
    \lambda^{-\frac{\epsilon p^2}{n}}\mathsf M^p
    \Big]^{\frac{n}{n-\epsilon p}} 
    \|f\|_{L^{\tilde\alpha p}(B_i)}^{\frac{np'}{n-\epsilon p}} .
\end{align*}
By the interpolation inequality~\eqref{crazy-holder}, we have
\begin{align*}
    \|f\|_{L^{\tilde\alpha p}(B_i)}^{\frac{np'}{n-\epsilon p}}
    &\le 
    \mathsf F^{\frac{p(sp'-1)}{n-\epsilon p}}
    \int_{B_i} |f|^{\frac{n\mu p}{n-\epsilon p}} \,\dx,
\end{align*}
where 
\begin{equation*}
    \mathsf F
    :=
    \int_{B_R} |f|^{\mu q} \,\dx .
\end{equation*}
Combining the previous estimates, we obtain 
\begin{align*}
    {\sf M}^p|B_i|^{\frac{p}{n}(sp'- 1)} & \bigg[\bint_{B_i} |f|^{\tilde\alpha p} \,\dx \bigg]^{\frac{1}{\tilde\alpha(p-1)}} 
    &\le 
    \tfrac14 \lambda^p +
    |B_i|^{-1}\mathsf 
    \lambda^{-\frac{\epsilon p^2}{n-\epsilon p}}
    \mathsf K^{\frac{n p}{n-\epsilon p}}
    \int_{B_i} |f|^{\frac{n\mu p}{n-\epsilon p}} \,\dx ,
\end{align*}
where
\begin{equation*}
    \mathsf K
    :=
    \big[4^{\epsilon p}
    \mathsf F^{p(sp'-1)} \big]^{\frac{1}{n p}}\mathsf M .
\end{equation*}
We insert this inequality into \eqref{eq:intr} and decompose the domain of integration into the super-level sets \(B_i \cap \{|\nabla u| > \tfrac18 \lambda \}\) and \(B_i \cap \{|f|^\mu > \frac{1}{8\sf{K}}\lambda\}\)
and their complements. On the corresponding sub-level sets, $u$ and $f$ are estimated by their upper bounds. Consequently,
\begin{align*}
    |B_i|\lambda^p
    &=
    \int_{B_i} |\nabla u|^p \, \d x + 
   |B_i|^{\frac{p}{n}(sp'- 1)+1 } {\sf M}^p\bigg[\mint_{B_i} |f|^{\tilde\alpha p} \,\dx \bigg]^{\frac{1}{\tilde\alpha(p-1)}} 
    \\
    &\le 
    \tfrac14 \lambda^p|B_i| + 
    \int_{B_i} |\nabla u|^p \, \d x + 
    \lambda^{-\frac{\epsilon p^2}{n-\epsilon p}}
    \mathsf K^{\frac{n p}{n-\epsilon p}}
    \int_{B_i} |f|^{\frac{n\mu p}{n-\epsilon p}} \,\dx
    \\
    &= 
    \tfrac14 \lambda^p|B_i| \\
    &\phantom{\le\,} + 
    \int_{B_i \cap \{|\nabla u| > \lambda/8 \}} |\nabla u|^p \, \d x + 
    \lambda^{-\frac{\epsilon p^2}{n-\epsilon p}}
    \mathsf K^{\frac{n p}{n-\epsilon p}} \int_{B_i \cap \{|f|^\mu > \lambda/(8\mathsf K) \}} |f|^{\frac{n\mu p}{n-\epsilon p}} \, \d x \\
    &\phantom{\le\,} +
    \int_{B_i \cap \{|\nabla u| \le \lambda/8 \}} |\nabla u|^p \, \d x + 
    \lambda^{-\frac{\epsilon p^2}{n-\epsilon p}}
    \mathsf K^{\frac{n p}{n-\epsilon p}} \int_{B_i \cap \{|f|^\mu \le \lambda/(8\mathsf K) \}} |f|^{\frac{n\mu p}{n-\epsilon p}} \, \d x \\
    &\le 
    \Big[ \tfrac14 
    +
    \tfrac1{8^p}
    +
    \tfrac{1}{8^{\frac{np}{n-\epsilon p}}}\Big]
    |B_i| \lambda^p \\
    &\quad + 
    \int_{B_i \cap \{|\nabla u| > \lambda/4 \}} |\nabla u|^p \, \d x + 
    \mathsf K^{\frac{n p}{n-\epsilon p}} 
    \lambda^{-\frac{\epsilon p^2}{n-\epsilon p}}
    \int_{B_i \cap \{|f|^\mu > \lambda/(8\mathsf K) \}} |f|^{\frac{n\mu p}{n-\epsilon p}} \, \d x .
\end{align*}
On the right-hand side the  pre-factor of $|B_i|\lambda^p$ is less than $\frac12$. Absorbing
$\tfrac12 |B_i|\lambda^p$
into the left-hand side and inserting the resulting estimate into \eqref{est:super-nabla-u} gives
\begin{align*}
    &\int_{B_i^{(3)}\cap \left\{ |\nabla u| > \mathsf A \lambda \right\}} |\nabla u|^p \, \d x \\
    &\quad\leq 
    c\,{\sf Q}
    \bigg[\int_{B_i \cap \{|\nabla u| > \lambda/4 \}} |\nabla u|^p \, \d x + 
    \mathsf K^{\frac{n p}{n-\epsilon p}} 
    \lambda^{-\frac{\epsilon p^2}{n-\epsilon p}}
    \int_{B_i \cap \{|f|^\mu > \lambda/(8\mathsf K) \}} |f|^{\frac{n\mu p}{n-\epsilon p}} \, \d x \bigg].
\end{align*}
Since \(\{B_i^{(3)}\}_{i\in \N}\) covers the super-level set \(\boldsymbol E(\lambda,r_1)\), and hence \(\boldsymbol E(\mathsf A \lambda,r_1)\), and the balls \(B_i\) are pairwise disjoint and contained in \(B_{r_2}\), one has
\begin{align*}
    &\int_{\boldsymbol E(\mathsf A \lambda,r_1)} |\nabla u|^p \, \d x \\
    &\quad\le
    \sum_{i\in\N}
    \int_{B_i^{(3)}\cap \left\{ |\nabla u| >\mathsf A \lambda \right\}} |\nabla u|^p \, \d x\\
    &\quad\le
    c\,{\sf Q}
    \sum_{i\in\N}
    \bigg[\int_{B_i \cap \{|\nabla u| > \lambda/4 \}} |\nabla u|^p \, \d x\\
    &\qquad\qquad\quad\;+ 
    \mathsf K^{\frac{n p}{n-\epsilon p}} 
    \lambda^{-\frac{\epsilon p^2}{n-\epsilon p}}
    \int_{B_i \cap \{|f|^\mu > \lambda/(8\mathsf K) \}} |f|^{\frac{n\mu p}{n-\epsilon p}} \, \d x \bigg]\\
    &\quad\leq 
    c\,{\sf Q} \bigg[\int_{\boldsymbol E(\lambda/4,r_2)} |\nabla u|^p \, \d x + 
    \mathsf K^{\frac{n p}{n-\epsilon p}} 
    \lambda^{-\frac{\epsilon p^2}{n-\epsilon p}} 
    \int_{B_{r_2} \cap \{|f|^\mu > \lambda/(8\mathsf K) \}} |f|^{\frac{n\mu p}{n-\epsilon p}} \, \d x \bigg].
\end{align*}
Instead of \(\boldsymbol{E}(\lambda,r_1)\), we consider, for \(k \ge {\sf B} \lambda_o\), the set
\begin{equation*}
    \boldsymbol{E}_k(\lambda,r_1) 
    := 
    \Big\{ x \in B_{r_1} : x \text{ is a Lebesgue point of } |\nabla u| \text{ and } |\nabla u|_k(x) > \lambda \Big\},
\end{equation*}
where \(|\nabla u|_k := \min \{ |\nabla u|, k \}\) denotes the truncation of \(|\nabla u|\) at height \(k\).
 We rewrite the last inequality in terms of $\boldsymbol E_k$ and obtain
\begin{align*}
    &\int_{\boldsymbol E_k(\mathsf A \lambda,r_1)} |\nabla u|^p \, \d x \\
    &\qquad\leq 
    c\, {\sf Q}
    \bigg[\int_{\boldsymbol E_k(\lambda/4,r_2)} |\nabla u|^p \, \d x + 
    \mathsf K^{\frac{n p}{n-\epsilon p}} 
    \lambda^{-\frac{\epsilon p^2}{n-\epsilon p}} 
    \int_{B_{r_2} \cap \{|f|^\mu > \lambda/(4\mathsf K) \}} |f|^{\frac{n\mu p}{n-\epsilon p}} \, \d x \bigg].
\end{align*}
Multiplying the previous inequality by \(\lambda^{q-p-1}\) and integrating over \(({\sf B}\lambda_o, \infty)\) with respect to \(\lambda\) gives
\begin{align*}
\int_{{\sf B}\lambda_o}^\infty\lambda^{q-p-1} 
\int_{\boldsymbol E_k(\mathsf A \lambda,r_1)} |\nabla u|_k^p \, \d x \, \d \lambda
&\le 
c\,{\sf Q}\big[\sfI+\sfI\sfI]
\end{align*}
where
\begin{align*}
    \sfI
    &:=
    \int_{{\sf B}\lambda_o}^\infty \lambda^{q-p-1} 
    \int_{\boldsymbol E_k(\lambda/4,r_2)} |\nabla u|_k^p \, \d x\d\lambda
   ,
\end{align*}
and
\begin{align*}
     \sfI\sfI
     &:=
    \mathsf K^{\frac{np}{n-\epsilon p}}
    \int_{{\sf B}\lambda_o}^\infty\lambda^{q-p-1}
    \lambda^{-\frac{\epsilon p^2}{n-\epsilon p}}
    \int_{B_{r_2} \cap \{|f|^\mu > \lambda/(8\mathsf K)\}} |f|^{\frac{n\mu p}{n-\epsilon p}} 
    \, \d x\d \lambda.
\end{align*}
Applying a Fubini-type argument to the integral on the left-hand side yields
\begin{align*}
    &\int_{{\sf B} \lambda_o}^\infty \lambda^{q-p-1} \int_{\boldsymbol E_k(\mathsf A\lambda,r_1)} |\nabla u|^p \, \d x \d \lambda\\
    &\qquad = 
    \frac{1}{q-p} \bigg[ \mathsf A^{p-q} \int_{B_{r_1}} |\nabla u|_k^{q-p} |\nabla u|^p  \, \d x  - 
    ({\sf B} \lambda_o)^{q-p} \int_{B_{r_1}} |\nabla u|^p \, \d x \bigg].
\end{align*}
The terms on the right-hand side are estimated similarly. For $\sfI$, one has
\begin{align*}
   \sfI
    &\leq 
    \frac{4^{q-p}}{q-p} \int_{B_{r_2}} |\nabla u|_k^{q-p} |\nabla u|^p  \, \d x,
\end{align*}
and for $\sfI\sfI$, noting that $\tilde\alpha p<\mu q$ implies $q > \frac{np}{n-\eps p}$, we find
\begin{align*}
    \sfI\sfI
    &\leq 
    \mathsf K^{\frac{np}{n-\epsilon p}}
    \frac{(8\mathsf K)^{q-\frac{np}{n-\epsilon p}}}{q-\frac{np}{n-\epsilon p}} \int_{B_{r_2}} |f|^{\mu q} \, \d x =
    \frac{8^{q-\frac{np}{n-\epsilon p}}}{q-\frac{np}{n-\epsilon p}} \mathsf K^{q}
    \int_{B_{r_2}} |f|^{\mu q} \, \d x.
\end{align*}
Combining these estimates and recalling that $\mu q=r$ results in
\begin{align*}
    \int_{B_{r_1}} &|\nabla u|_k^{q-p} |\nabla u|^p  \, \d x \\
    &\leq 
    c_\ast \mathsf A^{q-p} {\sf Q} \int_{B_{r_2}} |\nabla u|_k^{q-p} |\nabla u|^p  \, \d x\\
    &\phantom{\le\,} +
     (\mathsf A{\sf  B}\lambda_o)^{q-p} \int_{B_{r_1}} |\nabla u|^p \, \d x  + 
    c_\ast \mathsf A^{q-p} {\sf Q} 
   \mathsf M^{q} \mathsf F^\frac{q(sp'-1)}{n} \int_{B_{r_2}} |f|^{r} \, \d x,
\end{align*}
where ${\sf Q}$ is defined in \eqref{def:M}. 

The strategy is now to absorb the first term on the right-hand side into the left-hand side. This requires careful choices of the parameters in ${\sf Q}$, i.e.~of $\theta$, $\mathsf A$, $\delta$, and $\mathsf M$.
We start by choosing $\theta = 2q$. Next, we choose $\mathsf A$ large enough such that
\begin{equation*}
     \frac{c_\ast}{\mathsf A^{q}} 
    = 
    \tfrac 16\quad\Longleftrightarrow\quad
    \mathsf A=
     (6c_\ast)^\frac1{q}.
\end{equation*}
Subsequently, we choose $\delta$ 
small enough such that 
\begin{equation*}
    c_\ast \mathsf A^{q-p} \delta \equiv 
    c_\ast(6 c_\ast)^{\frac{q-p}{q}}\delta
    = \tfrac16\quad\Longleftrightarrow\quad
    \delta= (6c_\ast)^{\frac{p}{q}-2}.
\end{equation*}
This also fixes $\widetilde R_o$ in dependence on $n,p,q,s,C_o,C_1,\chi, R_o$. Finally we choose $\mathsf M$ large enough such that 
\begin{align*}
    \frac{c_\ast \mathsf A^{q-p}}{\delta^\frac{p}{\gamma} \mathsf M^p} 
    \equiv
    \frac{c_\ast (6c_\ast)^{\frac{q-p}{q} + \frac{p}{\gamma} (2-\frac{p}{q})}}{\mathsf M^p}
    = 
    \tfrac16
    \quad\Longleftrightarrow\quad
    \mathsf M^p=(6c_\ast)^{(2-\frac{p}{q}) (1+ \frac{p}{\gamma})}.
\end{align*}
These choices ensure that, on the one hand, $c_\ast  \mathsf A^{q-p} {\sf Q}=\frac12$, and on the other hand $\mathsf M$ depends only on $n,p,q,s,C_o,C_1,\chi$ (recall that $\gamma=\gamma(p,s,\chi,\tilde{\alpha})$ and $\tilde{\alpha} = \tilde{\alpha}(n,p,q,s)$). 
With the aforementioned choices, we obtain
\begin{align*}
    \int_{B_{r_1}} &|\nabla u|_k^{q-p} |\nabla u|^p  \, \d x \\
    &\leq 
   \tfrac12 \int_{B_{r_2}} |\nabla u|_k^{q-p} |\nabla u|^p  \, \d x \\
   &\quad +  
   (\mathsf A\mathsf B\lambda_o)^{q-p} \int_{B_{R}} |\nabla u|^p \, \d x  + 
    \tfrac12 \mathsf M^{q} \left[ \int_{B_R} |f|^{r} \, \d x \right]^\frac{q(sp'-1)}{n} \int_{B_{r_2}} |f|^{r} \, \d x \\
    &\leq 
    \tfrac12 \int_{B_{r_2}} |\nabla u|_k^{q-p} |\nabla u|^p \, \d x \\
    &\quad + 
    \frac{c\, R^\beta\, \lambda_o^{q-p}}{(r_2-r_1)^{\beta }} \int_{B_R} |\nabla u|^p \, \d x  + 
    c \left[ \int_{B_R} |f|^{r} \, \d x \right]^\frac{1}{\mu (p-1)},
\end{align*}
where $\beta = (\frac{n}{p-1}+1)(q-p)$. To obtain the last line we also used the definition of $\sf B$ in~\eqref{def:B}. Note that the last inequality holds for any $\frac12 R \le r_1<r_2\le R$. Therefore, the standard iteration lemma~\cite[Lemma 6.1]{Giusti} can be applied. The application results in
\begin{equation*}
     \int_{B_{\frac12 R}} |\nabla u|_k^{q-p} |\nabla u|^p  \, \d x
     \le
     c\,
     \lambda_o^{q-p} \int_{B_R} |\nabla u|^p \, \d x + c  \left[ \int_{B_R} |f|^{r} \, \d x \right]^\frac{1}{\mu (p-1)}.
\end{equation*}
Using Fatou's lemma we pass to the limit  $k \to \infty$ and obtain
\begin{align*}
    \int_{B_{\frac12 R}} |\nabla u|^{q}  \, \d x  
    &\leq 
    c\, \lambda_o^{q-p} \int_{B_R} |\nabla u|^p \, \d x + 
    c \left[ \int_{B_R} |f|^{r} \, \d x \right]^\frac{1}{\mu (p-1)}.
\end{align*}
Recalling the definition of $\lambda_o$ from \eqref{def:lambda_0}, passing to the mean values and using Young's inequality with exponents $\frac{q}{p}$ and $\frac{q}{q-p}$, and H\"older's inequality together with the fact that $\tilde \alpha p < \mu q=r$, we obtain
\begin{align*}
    \bigg[ \bint_{B_{\frac12 R}} |\nabla u|^{q} \, \d x \bigg]^\frac{1}{q}
    &\leq 
    c \,\bigg[\bint_{B_{R}}  |\nabla u|^p  \, \d x\bigg]^\frac{1}{p} + c R^{sp'-1}\bigg[ \bint_{B_R} |f|^{r} \, \d x \bigg]^\frac{1}{\mu q(p-1)} \\
    &\phantom{\le\,} + \frac{c}{R} \, \mathrm{Tail} \big(u-(u)_{R};B_R\big).
\end{align*}
This concludes the proof of Theorem \ref{Thm:1}.

\medskip
\noindent
{\bf Acknowledgments.}  
This research was funded in whole or in part by the Austrian Science Fund (FWF) [10.55776/P36295] and [10.55776/P36272] and [10.55776/PAT1850524]. For open access purposes, the author has applied a CC BY public copyright license to any author accepted manuscript version arising from this submission.


\begin{thebibliography}{10}
\bibitem{Acerbi-Fusco}
\newblock E. Acerbi and N. Fusco. 
\newblock Regularity for minimizers of non-quadratic functionals: the case $1<p<2$. 
\newblock \emph{J. Math. Anal. Appl.} 140 (1989), 115--135.

\bibitem{AcMi}
\newblock E. Acerbi and G.~Mingione.
\newblock Gradient estimates for the $p(x)$-Laplacean system.
\newblock \emph{J. Reine Angew. Math.} 584 (2005), 117--148.

\bibitem{Acerbi-Min}
\newblock E.~Acerbi and G.~Mingione.
\newblock Gradient estimates for a class of parabolic systems.
\newblock \emph{Duke Math. J.} 136 (2007), no. 2, 285--320.

\bibitem{Adams-75}
R. A. Adams, Sobolev Spaces. New York Academic Press, 1975.

\bibitem{AKM-18}
\newblock B. Avelin, T. Kuusi, and G. Mingione. 
\newblock Nonlinear Calder\'on-Zygmund
theory in the limiting case. 
\newblock Arch. Ration. Mech. Anal. 227 (2018), no. 2,
pp. 663--714.


%

\bibitem{BS-25} 
\newblock A.~Biswas and A.~Sen. 
\newblock Improved Hölder regularity of fractional $(p,q)$-Poisson equation with regular data. 
\newblock arXiv:2507.09920.


\bibitem{BT-25} 
\newblock A.~Biswas and E.~Topp. 
\newblock Lipschitz regularity of fractional p-Laplacian. 
\newblock \emph{Ann. PDE} 11 (2025), no. 2, Paper No. 27, 43 pp.

\bibitem{BDLMS-1}
\newblock V.~Bögelein, F.~Duzaar, N.~Liao, G.~Molica Bisci, and R.~Servadei. 
\newblock Regularity for the fractional $p$-Laplace equation.
\newblock \emph{J. Funct. Anal.} 289 (2025), no. 9, Paper No. 111078, 69 pp.


\bibitem{BDLMS-2}
\newblock V.~Bögelein, F.~Duzaar, N.~Liao, G.~Molica Bisci, and R.~Servadei. 
\newblock Gradient regularity for $(s,p)$-harmonic functions.
\newblock \emph{Calc. Var. Partial Differential Equations} 64 (2025), no. 8, Paper No. 253, 63 pp.

\bibitem{BDLM}
\newblock V.~Bögelein, F.~Duzaar, N.~Liao, and K.~Moring. 
\newblock Gradient estimates for the fractional $p$-Poisson equation. 
\newblock \emph{J. Math. Pures Appl. (9)} 204 (2025), Paper No. 103764, 25 pp.

\bibitem{Brasco-Lindgren}
\newblock L.~Brasco and E.~Lindgren. 
\newblock Higher Sobolev regularity for the fractional $p$-Laplace equation in the superquadratic case. 
\newblock \emph{Adv. Math.} 304 (2017), 300--354.

\bibitem{Brasco-Lindgren-Schikorra}
\newblock L.~Brasco, E.~Lindgren, and A.~Schikorra. 
\newblock Higher H\"older regularity for the fractional $p$-Laplacian in the superquadratic case.
\newblock \emph{Adv. Math.} 338 (2018), 782--846.

\bibitem{BM}
\newblock H.~Brezis and P.~Mironescu.
\newblock Gagliardo-Nirenberg inequalities and non-inequalities: the full story.
\newblock {\em Ann. Inst. H. Poincaré C Anal. Non Linéaire} 35 (2018), no. 5, 1355--1376.
%

\bibitem{CP-98}
\newblock L.~Caffarelli and I.~ Peral.
\newblock On $W^{1,p}$ estimates for elliptic equations in divergence form.
\newblock \emph{Comm. Pure Appl. Math.} 51 (1998), 1--21.

\bibitem{CZ-1}
\newblock A.~P.~Calder\'on and A.~Zygmund. 
\newblock On the existence of certain singular integrals. 
\newblock \emph{Acta Math.} 88 (1952), 85--139.

\bibitem{CZ-2}
\newblock A.~P.~Calder\'on and A.~Zygmund. 
\newblock On singular integrals. 
\newblock \emph{Amer. J. Math.} 78 (1956),
289--309.

\bibitem{Coz17b}
\newblock M.~Cozzi.
\newblock Regularity results and Harnack inequalities for minimizers and solutions of nonlocal problems: a unified approach via fractional De Giorgi classes.
\newblock {\em J. Funct. Anal.} 272 (2017), no. 11, 4762--4837.

\bibitem{DiBe-Man}
\newblock E.~DiBenedetto and J.~Manfredi.
\newblock On the higher integrability of the gradient of weak solutions of certain degenerate elliptic systems. 
\newblock \emph{Amer. J. Math.} 115 (1993), 1107--1134.


\bibitem{DM-23}
\newblock C. De Filippis and G. Mingione. 
\newblock Nonuniformly elliptic Schauder theory.
\newblock \emph{Invent. Math.} 234. (2023), no. 3, pp. 1109--1196. 
\bibitem{DKP}
\newblock A.~Di Castro, T.~Kuusi, and G.~Palatucci.
\newblock Local behavior of fractional $p$-minimizers. 
\newblock  {\em Ann. Inst. H. Poincaré C Anal. Non Linéaire} 33 (2016), no. 5, 1279--1299.


\bibitem{DKLN-pot}
\newblock L.~Diening, K.-B.~Kim, H.-S.~Lee, and S.~Nowak.
\newblock Nonlinear nonlocal potential theory at the gradient level.
\newblock arXiv:2402.04809.
%
\bibitem{DKLN-higherdiff}
\newblock L.~Diening, K.-B.~Kim, H.-S.~Lee, and S.~Nowak.
\newblock Higher differentiability for the fractional $p$-Laplacian. 
\newblock \emph{Math. Ann.} 391 (2025), no. 4, 5631--5693.
%

\bibitem{Domokos-1} 
A. Domokos, 
On the regularity of $p$-harmonic functions in the Heisenberg group. PhD thesis University of Pittsburgh 2004. 

\bibitem{Hitchhikers-guide}
\newblock E.~Di Nezza, G.~Palatucci, and E.~Valdinoci.
\newblock Hitchhiker’s guide to the fractional Sobolev spaces. 
\newblock \emph{Bull. Sci. Math.} 136 (2012),  no.~5, 521--573.

\bibitem{DuMinKr}
\newblock F.~Duzaar, G.~Mingione, and J.~Kristensen.
\newblock The existence of regular boundary points for non-linear elliptic systems. 
\newblock \emph{J. Reine Angew. Math.} 602 (2007), 17--58.

%
\bibitem{DuMinSt}
\newblock F.~Duzaar, G.~Mingione, and K.~Steffen.
\newblock Parabolic systems with polynomial growth and regularity.
\newblock \emph{Mem. Amer. Math. Soc.} 214 (2011), no. 1005, x+118 pp.

%
\bibitem{GiaquintaModica:1986}
\newblock  M. Giaquinta and G. Modica.
\newblock Partial regularity of minimizers of quasiconvex integrals.
\newblock \emph{Ann. Inst. H. Poincar\'e Anal. Non Lin\'eaire} 3 (1986), no. 3, 185--208.

\bibitem{GJS-25}
\newblock D.~Giovagnoli, D.~Jesus, and L.~Silvestre.
\newblock $C^{1+\alpha}$ regularity for fractional $p$-harmonic functions.
\newblock arXiv:2509.26565.

\bibitem{Giusti}
E.~Giusti. Direct Methods in the Calculus of Variations. World Scientiﬁc Publishing Company, Tuck Link,
Singapore, 2003.

\bibitem{Iwaniecz}
\newblock T.~Iwaniec.
\newblock Projections onto gradient fields and $L^p$-estimates for degenerated elliptic operators. 
\newblock \emph{Studia Math.} 75 (1983), no. 3, 293--312.

\bibitem{Ki-Zh}
\newblock J.~Kinnunen and S.-L.~Zhou.
\newblock A local estimate for nonlinear equations with discontinuous coefficients.
\newblock \emph{Comm. Partial Differential Equations} 24 (1999), no. 11-12, 2043--2068.

\bibitem{Kri-Min:05}
\newblock J.~Kristensen and G.~Mingione.
\newblock The singular set of $\omega$-minima. 
\newblock \emph{Arch. Ration. Mech. Anal.} 177 (2005), no. 1, 93--114.

\bibitem{Kri-Min:06}
\newblock J.~Kristensen and G.~Mingione.
\newblock The singular set of minima of integral
functionals. 
\newblock \emph{Arch. Ration. Mech. Anal.} 180 (2006), no. 3, pp. 331--398. 

%

\bibitem{KNS-22}
\newblock T.~Kuusi, S.~Nowak, and Y.~Sire. 
\newblock Gradient regularity and first-order potential estimates for a class of nonlocal equations.
\newblock arXiv:2212.01950.

\bibitem{Min-07}
\newblock G. Mingione. 
\newblock Gradient potential estimates. 
\newblock \emph{J. Eur. Math. Soc.}
13 (2011), no. 2 pp. 459--486.

\bibitem{Min-11}
\newblock G.~Mingione.
\newblock The Calder\'on-Zygmund theory for elliptic problems with measure data. 
\newblock \emph{Ann. Sc. Norm. Super. Pisa Cl. Sci. (5)} 6 (2007), no. 2,
pp. 195--261.

\bibitem{Min-survey}
\newblock G.~Mingione.
\newblock Short tales from nonlinear Calderón-Zygmund theory. Nonlocal and nonlinear diffusions and interactions: new methods and directions, 159–204, Lecture Notes in Math., 2186, Fond. CIME/CIME Found. Subser., Springer, Cham, 2017.



\bibitem{Stein}
E.~Stein. Singular Integrals and Differentiability Properties of Functions. Princeton Math. Ser., vol. 30, Princeton University Press,
Princeton, N.J., 1970.








\end{thebibliography}
\end{document}